\documentclass{article} 
\usepackage{amsmath,amssymb,amsthm} 
\usepackage[all]{xy}
\usepackage[OT2,T1]{fontenc}
\usepackage[utf8]{inputenc}
\usepackage{enumitem}
\usepackage{graphicx}
\usepackage{geometry}
\usepackage{hyperref}
\usepackage{tikz}
\usetikzlibrary{cd}
\usepackage{tikz-cd}
\usepackage{xcolor}
\usepackage{ stmaryrd }
\usepackage{mathrsfs}
\usepackage{titling}
\usepackage{fullpage}
\usepackage{mathtools}
\usepackage[english]{babel}
\usepackage{setspace}
\usepackage[normalem]{ulem}
\usepackage{mathdots}
\usepackage{microtype}

\usepackage{arydshln}

\newcommand*{\Scale}[2][4]{\scalebox{#1}{$#2$}}

\usetikzlibrary{arrows,positioning,decorations.pathmorphing,
	decorations.markings
}
\tikzset{inner sep=0pt,
	root/.style={circle,draw,minimum size=7pt,thick},
	fatroot/.style={circle,draw,minimum size=10pt,thick},
	short root/.style={circle,fill,minimum size=7pt},
	doublearrow/.style={postaction={decorate},
		decoration={markings,mark=at position .7
			with {\arrow{angle 60}}},double distance=3pt,thick}
}

\DeclareUnicodeCharacter{00A0}{ }

\newtheorem{proposition}{Proposition}[section]
\newtheorem{definition}[proposition]{Definition}

\newtheorem{theorem}[proposition]{Theorem}
\newtheorem{lemma}[proposition]{Lemma}
\newtheorem{corollary}[proposition]{Corollary}

\numberwithin{equation}{section}

\newcommand{\G}{\mathbb{G}}
\DeclareMathOperator{\GL}{GL}

\DeclareMathOperator{\Sp}{Sp}

\DeclareMathOperator{\SO}{SO}
\DeclareMathOperator{\OO}{O}

\DeclareMathOperator{\Spin}{Spin}

\DeclareMathOperator{\Id}{Id}

\DeclareMathOperator{\Hom}{Hom}

\DeclareMathOperator{\End}{End}
\DeclareMathOperator{\Sym}{Sym}

\DeclareMathOperator{\tr}{tr}
\DeclareMathOperator{\rank}{rank}

\DeclareMathOperator{\Stab}{Stab}

\DeclareMathOperator{\Spec}{Spec}

\DeclareMathOperator{\Pic}{Pic}

\newcommand{\sh}[1]{\mathscr{#1}}

\newcommand{\A}{\mathbb{A}}
\renewcommand{\P}{\mathbb{P}}

\renewcommand{\O}{\mathcal{O}}

\newcommand{\HH}{\mathrm{H}}

\newcommand{\cO}{\mathcal{O}}

\newcommand{\R}{\mathbb{R}}
\newcommand{\Q}{\mathbb{Q}}
\newcommand{\bC}{\mathbb{C}}

\newcommand{\Z}{\mathbb{Z}}
\newcommand{\F}{\mathbb{F}}

\DeclareSymbolFont{cyrletters}{OT2}{wncyr}{m}{n}
\DeclareMathSymbol{\Sha}{\mathalpha}{cyrletters}{"58}

\newcommand{\extp}{\@ifnextchar^\@extp{\@extp^{\,}}}
\def\@extp^#1{\mathop{\bigwedge\nolimits^{\!#1}}}

\DeclareUnicodeCharacter{00A0}{ }
\newcommand{\height}{\mathrm{Ht}}

\DeclareMathOperator{\ord}{ord}

\DeclareMathOperator{\OGr}{OGr}
\DeclareMathOperator{\Gr}{Gr}

\title{Kummers, spinors, and heights}
\author{Jef Laga and Jack A. Thorne}

\begin{document}

\maketitle

\begin{abstract}
Let $f(x) = x^{2g+1} + c_1 x^{2g} + \dots + c_{2g+1} \in k[x]$ be a polynomial of nonzero discriminant. The Jacobian $J_f$ of the odd hyperelliptic curve $C_f : y^2 = f(x)$ is an abelian variety of dimension $g$, equipped with a symmetric theta divisor $\Theta$ defining a principal polarisation. 
We show that the morphism $J_f \to \P^{2^g-1}$ associated to the linear system $|2 \Theta|$ may be described explicitly, for any $g \geq 1$, using the theory of pure spinors. 

We apply this theory to study the heights of rational points in $J_f(k)$, when $k$ is a number field. As a particular consequence, we show that $100\%$ of monic, degree $2g+1$ polynomials $f(x) \in \Z[x]$ of nonzero discriminant $\Delta(f)$ have the property that, for any non-trivial point $P \in J_f(\Q)$, the canonical height of $P$ satisfies
\[ \widehat{h}_\Theta(P) \geq \left(\frac{3g-1}{4g(2g+1)} - \epsilon\right) \log | \Delta(f) |. \]
This is a `density 1' form of the Lang--Silverman conjecture \cite{Sil84}.
\end{abstract}

\tableofcontents

\section{Introduction}

\paragraph{Context.} Let $g \geq 1$. In this paper, our focus is the family of odd hyperelliptic curves $C_f$, defined as the smooth projective completions of the affine curves
\[ C_f^0 : y^2 = f(x) = x^{2g+1} + c_1 x^{2g} + \dots + c_{2g+1}, \]
where $f(x) \in \Z[x]$ is a polynomial of nonzero discriminant $\Delta(f)$. There is a unique marked point $P_\infty \in C_f(\Q)$ at infinity, and $(g-1)P_\infty$ is a theta characteristic; therefore the  Jacobian $J_f = \Pic^0_{C_f}$ admits a principal polarisation, defined by the symmetric theta divisor $\Theta = t_{(g-1)P_\infty}^\ast W_{g-1}$.

Our motivation is an explicit question about the heights of rational points $P \in J_f(\Q)$, asked in our previous paper \cite{lagathorne2024smallheightoddhyperelliptic}. (For the purposes of this discussion, a height is any function $h : J_f(\Q) \to \R$ such that, for any $c \in \R$, the set $\{ P \in J_f(\Q) \mid h(P) \leq c \}$ is finite.) One possible choice of height is the dagger height $h^\dagger$, whose definition, given in terms of the Mumford representation of the point $P$, we now recall. 

The points $P \in J_f(\Q)$ are in bijection with the  triples $(U, V, R)$, where 
\[ U(x) = x^m + u_1 x^{m-1} + \dots + u_m, \]
\[ V(x) = x^{2g+1-m} + v_1 x^{2g-m} + \dots + v_{2g+1-m} \in \Q[x], \]
 are monic polynomials of degrees $m$, $2g+1-m$, respectively (for some $0 \leq m \leq g$ depending on $P$), 
 \[ R(x) = r_1 x^{m-1} + r_2 x^{m-2} + \dots + r_{m} \in \Q[x]\]
is a polynomial of degree $\leq m-1$, and we have the relation $f - R^2 = UV$. (The dictionary is that $P$ corresponds to the class of the divisor $D - m P_\infty$, where $D \leq C_f^0$ is the effective divisor defined by the equations $U(x) = y - R(x) = 0$.) We define 
\[ h^\dagger(P) = h(U) = h( 1 : u_1 : \dots : u_m), \]
as the usual logarithmic Weil height of the point $[1 : u_1 : \dots : u_m]$ of projective space $\P^m(\Q)$. 

The following was the main theorem of \cite{lagathorne2024smallheightoddhyperelliptic}.
\begin{theorem}\label{introthm_lower_bound_for_dagger_height}
    Fix $g \geq 1$ and $\epsilon > 0$. Then $100\%$ of polynomials $f(x) = x^{2g+1} + c_2 x^{2g-1} + \dots + c_{2g+1} \in \Z[x]$ have the property that for any point $P \in J_f(\Q) - \{ 0 \}$, $h^\dagger(P) \geq (g - \epsilon) \log \height(f)$.
\end{theorem}
Here we define $\height(f) = \max |c_i|^{1/i}$, and `$100\%$' indicates a set of full density with respect to the ordering given by $\height(f)$. The question is then if one can prove similar results for the heights more usually associated with $J_f$, including:
\begin{itemize}
    \item The naive height $h : J_f(\Q) \to \R$, defined by composing the usual height $\P^{2^g-1}(\Q) \to \R$ on projective space  with the Kummer morphism $J_f \to 
    \P(\HH^0(J_f, \O_{J_f}(2\Theta))^{\vee}) \simeq\P^{2^g-1}$ associated to the linear system $| 2 \Theta |$. (Recall this morphism factors through an embedding of the Kummer variety $K_f = J_f / \{ \pm 1 \}$ in $\P^{2^g-1}$.)
    \item The canonical height $\widehat{h}_\Theta : J_f(\Q) \to \R$, which may be defined, following Tate, by the formula
    \[ \widehat{h}_\Theta = \frac{1}{2} \lim_{n \to \infty} 4^{-n} h([2^n](P)).\]
\end{itemize}
Comparing either height with $h^\dagger$ presents difficulties, especially when $g$ is large. The first problem is that the naive height $h$ is not well-defined, but rather depends on a choice of co-ordinates on $ \P(\HH^0(J_f, \O_{J_f}(2\Theta))^{\vee})$ (equivalently, a choice of basis for the $2^g$-dimensional $\Q$-vector space $\HH^0(J_f, \cO_{J_f}(2 \Theta))$). If $g = 1$, then $\Theta = P_\infty$ and there is the `obvious' choice of basis $\{1, x\}$. When $g = 2$, Cassels and Flynn \cite[Ch. 3]{CasselsFlynnProlegomena} describe an explicit basis, and give the equation of the Kummer surface in $\P^3$; a generic point $P \in J_f$, with Mumford triple $(U, V, R)$, is sent to the point $[1 : u_1 : u_2 : u_2 v_1 + v_3]$, showing immediately that $h^\dagger(P) \leq h(P)$. Stubbs, Duquesne, and M\"uller \cite{Stu00, Duq01, Mul14} gave similar formulae in the case $g = 3$, again leading to a formula $h^\dagger(P) \leq h(P)$ for any point $P \in J_f(\Q)$. (In another work, Stoll \cite{Stoll-heightsgenus3} gave formulae valid also for even hyperelliptic curves of genus 3.) This implies that Theorem \ref{introthm_lower_bound_for_dagger_height} holds also with $h^\dagger$ replaced by the naive height $h$ when $g \leq 3$. However, we are not aware of any analogous results on a normalisation of $h$ when $g \geq 4$ that would allow one to prove a result for all $g$. 

The second problem is that comparing the canonical height $\widehat{h}_\Theta$ (which at least does not depend on any auxiliary choices) to $h^\dagger$ or $h$ explicitly is difficult without, e.g., a precise description of the embedding $K_f \to \P^{2^g-1}$ and the associated duplication morphism $[2] : K_f \to K_f$. Using such a description, Stoll gave a concrete lower bound for $\widehat{h}_\Theta$ in terms of $h$ in the cases $g = 2, 3$ \cite{Sto99, Stoll-heightsgenus3}. Using Arakelov theory, Holmes \cite{holmes-arakelovapproachnaiveheighthyperelliptic} proved a bound $\widehat{h}_\Theta \geq \frac{1}{2} h^\dagger + O_f(1)$ 
for any $g \geq 1$, but without an explicit description of the dependence of the constant on $f$, and therefore not sufficient for a version of Theorem \ref{introthm_lower_bound_for_dagger_height} with $h^\dagger$ replaced by $\widehat{h}_\Theta$.

Finally, it should be mentioned that if we pass to an extension $k / \Q$ large enough to be able to define a theta structure (in the sense of Mumford \cite{Mum66}) on the invertible sheaf $\cO_{J_f}(2 \Theta)$, then the Stone--Von Neumann theorem determines a choice of co-ordinates on $\HH^0(J_f, \cO_{J_f}(2 \Theta)) \otimes_\Q k$ (well-defined up to scalar multiple), and explicit bounds between the canonical height and the induced naive height are then available (see e.g.\ \cite{Zar72, Paz12}). However, passing to this extension $k / \Q$  requires adjoining at least all of the roots of the polynomial $f(x)$, and this approach does not seem useful for studying the variation of heights in the family $C_f$, or for applications to computing the set $J_f(\Q)$ of rational points in examples. 

\paragraph{Results of this paper.} In this paper, we give an explicit description of the Kummer embedding, for any $g \geq 1$, and prove analogues of Theorem \ref{introthm_lower_bound_for_dagger_height} for both the naive and canonical height on $J_f$. This includes the following results:
\begin{itemize}
    \item We  show that there is a canonical choice of basis for the $\Q$-vector space $\HH^0(J_f, \cO_{J_f}(2 \Theta))$, and show how to compute the induced morphism $J_f \to \P^{2^g-1}$ explicitly in terms of the Mumford representation $(U, V, R)$ of a point $P \in J_f(\Q)$. (See \S \ref{subsec_explicit_Kummer_embedding} --in fact, this works over any field $k$ of characteristic not 2, but we stick to the case $k = \Q$ in the discussion here for simplicity.) 
    \item The space $\HH^0(J_f, \cO_{J_f}(2 \Theta))$ is a representation of the theta group $\mathcal{G}(2 \Theta)$, which sits in a short exact sequence
    \[ 0 \to \G_m \to \mathcal{G}(2 \Theta) \to J[2] \to 0 \]
    of linear algebraic $\Q$-groups. For any 2-torsion point $T \in J[2]$, we write down lifts $\widetilde{T} \in \mathcal{G}(2 \Theta)$ and show how to compute the matrices $M_{\widetilde{T}}$ of their action on $\HH^0(J_f, \cO_{J_f}(2 \Theta))$ in our given basis (see \S \ref{subsec_action_of_Heisenberg_group}).
    \item We prove the existence of quartic polynomials $\delta_1, \dots, \delta_{2^g} \in \Q[x_1, \dots, x_{2^g}]$ such that the induced rational map $\P^{2^g-1} \dashrightarrow \P^{2^g-1}$ is defined on $K_f$ and agrees there with the duplication map $[2] : K_f \to K_f$. An explicit description of the $\delta_i$ may be given in terms of the matrices $M_{\widetilde{T}}$, and it is therefore possible to compute the $\delta_i$ in examples (see \S \ref{subsec_duplication_on_the_kummer}). This formulation is also well-adapted to the study of the canonical height $\widehat{h}_\Theta$.
    \item We describe explicit equations for the subvariety $K_f \leq \P^{2^g-1}$. We show that the ideal of quadrics vanishing on $K_f$ is independent of $f$, and that $K_f$ is defined scheme-theoretically in $\P^{2^g-1}$ by quadrics and quartics. This makes it possible to decide whether a given point $P \in \P^{2^g-1}(\Q)$ lies in $K_f(\Q)$. Finally, we give a simple criterion to decide whether a point $P \in K_f(\Q)$ lifts to $J_f(\Q)$. (See \S \ref{subsec_image_of_Kummer_embedding})  
\end{itemize}
A particular consequence of the first point above is that the naive height $h : J_f(\Q) \to \R$ is well-defined. Using our explicit description of the Kummer embedding, we can show that $h\geq h^{\dagger}$ and hence that Theorem \ref{introthm_lower_bound_for_dagger_height} holds for this height: 
\begin{theorem}[Theorem \ref{thm_naive_height_lower_bound}]\label{introthm_lower_bound_for_naive_height}
    Fix $g \geq 1$ and $\epsilon > 0$. Then $100\%$ of polynomials $f(x) = x^{2g+1} + c_2 x^{2g-1} + \dots + c_{2g+1} \in \Z[x]$ have the property that for any point $P \in J_f(\Q) - \{ 0 \}$, $h(P) \geq (g - \epsilon) \log \height(f)$.
\end{theorem}
We then apply these results to the canonical height $\widehat{h}_\Theta$. First, we give an explicit lower bound for $\widehat{h}_\Theta$ in terms of the naive height $h$ that applies to all polynomials $f(x)$ (not just those in a density 1 subset):
\begin{theorem}\label{introthm_lower_bound_for_canonical_height_in_terms_of_naive_height}
    Fix $g \geq 1$. Then there is a constant $c(g) \in \R$ such that for any $f(x) = x^{2g+1} + c_1 x^{2g} + \dots + c_{2g+1} \in \Z[x]$ of nonzero discriminant, and any $P \in J_f(\Q)$, we have
    \[ \widehat{h}_\Theta(P) \geq \frac{1}{2} h(P) - \frac{1}{12} g(7g+5) \log \height(f) + c(g). \]
\end{theorem}
(See Theorem \ref{thm_canonical_height_difference} for a more precise result.) Finally, we prove an analogue of Theorem \ref{introthm_lower_bound_for_dagger_height} for the canonical height.
\begin{theorem}[Theorem \ref{thm_density_1_lower_bound_for_canonical_height}]\label{introthm_lower_bound_for_canonical_height}
    Fix $g \geq 1$ and $\epsilon > 0$. Then $100\%$ of polynomials $f(x) = x^{2g+1} + c_2 x^{2g-1} + \dots + c_{2g+1} \in \Z[x]$ have the property that for any point $P \in J_f(\Q) - \{ 0 \}$, $\widehat{h}_\Theta(P) \geq (\frac{3g-1}{2} - \epsilon) \log \height(f)$.
\end{theorem}
The statement in the abstract follows from this one, using the lower bound $|\Delta(f)| \ll \height(f)^{2g(2g+1)}$.

We conclude this description of our results with an explicit example. Suppose we take $g = 5$; then to any Mumford triple $(U, V, R)$ should be associated a point of $\P^{31}$. The generic case is when $\deg U = 5$, in which case the co-ordinates of this point are (up to sign and powers of $2$, made explicit in Proposition \ref{prop_computation_of_pure_spinor}) the Pfaffians of the pair
\[ (\xi_i) = (-u_5, -u_4, -u_3, -u_2, -u_1), \]
 \[ \Scale[0.53]{ (\xi_{ij}) = \left(
\begin{array}{ccccc}
 0 & \frac{1}{2} (2 r_3 r_5+u_5 v_4+u_3 v_6) & \frac{1}{2} (2 r_2 r_5+u_5 v_3+u_2 v_6) & \frac{1}{2} (2 r_1 r_5+u_5 v_2+u_1 v_6) & \frac{1}{2} (u_5 v_1+v_6) \\
 \frac{1}{2} (-2 r_3 r_5-u_5 v_4-u_3 v_6) & 0 & \frac{1}{2} (2 r_2 r_4+2 r_1 r_5+u_5 v_2+u_4 v_3+u_2 v_5+u_1 v_6) & \frac{1}{2} (2 r_1 r_4+u_5 v_1+u_4 v_2+u_1 v_5+v_6) & \frac{1}{2} (u_4 v_1+u_5+v_5) \\
 \frac{1}{2} (-2 r_2 r_5-u_5 v_3-u_2 v_6) & \frac{1}{2} (-2 r_2 r_4-2 r_1 r_5-u_5 v_2-u_4 v_3-u_2 v_5-u_1 v_6) & 0 & \frac{1}{2} (2 r_1 r_3+u_4 v_1+u_3 v_2+u_1 v_4+u_5+v_5) & \frac{1}{2} (u_3 v_1+u_4+v_4) \\
 \frac{1}{2} (-2 r_1 r_5-u_5 v_2-u_1 v_6) & \frac{1}{2} (-2 r_1 r_4-u_5 v_1-u_4 v_2-u_1 v_5-v_6) & \frac{1}{2} (-2 r_1 r_3-u_4 v_1-u_3 v_2-u_1 v_4-u_5-v_5) & 0 & \frac{1}{2} (u_2 v_1+u_3+v_3) \\
 \frac{1}{2} (-u_5 v_1-v_6) & \frac{1}{2} (-u_4 v_1-u_5-v_5) & \frac{1}{2} (-u_3 v_1-u_4-v_4) & \frac{1}{2} (-u_2 v_1-u_3-v_3) & 0 \\
\end{array}
\right).} \]
(These Pfaffians $\xi_I$ are indexed by subsets $I \subset \{ 0, \dots, 4 \}$; if $|I|$ is even, then the Pfaffian $\xi_I$ is the minor of the corresponding submatrix of $(\xi_{ij})$, while if $|I|$ is odd, then the definition is as given in \S \ref{subsec_properties_of_pure_spinors} below.) Thus the first few co-ordinates of this point are
\[ [ 1:-u_1:u_2:-u_3:u_4:-u_5:-u_3-u_2 v_1-v_3:u_4+u_3 v_1+v_4:-u_5-u_4 v_1-v_5:u_5 v_1+v_6: \dots \]
Proposition \ref{prop_properties_of_explicit_Kummer_morphism} shows that in this case, and in general, that the co-ordinates are polynomials, with coefficients in $\Z$, in the coefficients of $U$, $V$, and $R$.

For an example featuring the duplication polynomials $\delta_i$, see \S \ref{subsec_generic_spin_basis} below.
\paragraph{Methods.} Our description of the Kummer embedding is based on the well-known identification of $J_f$ with a subvariety of the Grassmannian $\Gr(g, 2g+2)$: more precisely, with the locus of $g$-planes contained in the base locus of a nondegenerate pencil of quadrics. This description is studied variously in \cite{Rei72, Don80, Des76, Wan18} from both an `Albanese' and a `Picard' point of view. The starting point for our work is that the Kummer variety admits a similar description.

We first introduce some notation. Let us take now $k$ to be any field of characteristic not 2, and $f(x) = x^{2g+1} + c_1 x^{2g} + \dots + c_{2g+1} \in k[x]$ a polynomial of nonzero discriminant. We define a nondegenerate quadratic space $V = k[x] / (f(x))$, equipped with the symmetric bilinear form $\psi(a, b) = \tau(ab)$, where the multiplication is in the ring $k[x] / (f(x))$ and $\tau : V \to k$ is the linear form defined by $\tau( \sum_{i=0}^{2g} a_i x^i) = a_{2g}$. We write $\OGr(g, V)$ for the orthogonal Grassmannian (of isotropic $g$-places in $V$). 

The first key construction of our paper is the definition of morphisms
\[
\Psi : J_f \rightarrow \OGr(g,V), \,\, \Sigma : \OGr(g, V) \hookrightarrow \P^{2^g-1}
\]
whose composition is isomorphic to the Kummer embedding.
The morphism $\Psi$ may be described concretely as follows. Let $W \subset C_f^0$ denote the finite \'etale $k$-scheme defined by the equation $y = 0$, so there is an identification $\HH^0(W, \O_W) = V$.  
Suppose that $P \in J_f(k)$ is a point corresponding to the class of the divisor $D-mP_{\infty}$. 
If $D = \sum P_i$ with $P_i \in C_f^{\circ}(\bar{k})$, write $U = \prod (x- x(P_i))$ and suppose for simplicity that $D$ is disjoint from $W$; in other words, that $U$ is a unit in the \'etale algebra $V=k[x]/(f(x))$.
Then we show that the image of the composition
\[
\HH^0(C_f, \O_{C_f}((2g+2m-1)P_{\infty}-2D)) \rightarrow \HH^0(W, \O_W) = V\xrightarrow{\times U^{-1}} V
\]
is a $g$-dimensional isotropic subspace $F_P$ of $V$, where the first map is the natural `reduction mod $y$' map.
We set $\Psi(P) = F_P$.
We show that this description defines a morphism $\Psi\colon J_f\rightarrow \OGr(g,V)$ that factors through an embedding of the Kummer $K_f\hookrightarrow \OGr(g,V)$.

The determinant of the tautological quotient bundle on $\OGr(g, V)$ is very ample, and corresponds to the Pl\"ucker embedding of the orthogonal Grassmannian. This line bundle has a very ample square root, that we denote $\cO_{\OGr(g, V)}(1)$; the global sections of this line bundle form the $2^g$-dimensional spin representation $S$ of the Spin group $\Spin(V)$ (double cover of the special orthogonal group $\SO(V)$ -- in fact, we use the full Clifford group $\Gamma$, which is a $\G_m$-extension of $\OO(V)$ with derived group $\Spin(V)$). We write $\Sigma : \OGr(g, V) \to \P(S)$ for the induced embedding. We then show:
\begin{theorem}
    There is an isomorphism $\Psi^\ast \Sigma^\ast \cO_{\P(S)}(1) = \Psi^\ast \cO_{\OGr(g, V)}(1) \cong \cO_{J_f}(2 \Theta)$, which determines an isomorphism $S^\vee = \HH^0(\OGr(g, V), \cO_{\OGr(g, V)}(1)) \cong\HH^0(J_f, \cO_{J_f}(2 \Theta))$.
\end{theorem}
We can now explain how to define an explicit basis for $\HH^0(J_f, \cO_{J_f}(2 \Theta))$, and therefore an explicit embedding $K_f \hookrightarrow \P^{2^g-1}$. The quadratic space $V$ comes equipped with a basis $p_0, \dots, p_{2g}$ satisfying $\psi(p_i, p_j) = \delta_{i, 2g-j}$ (this comes by `straightening' the power basis $1, x, \dots, x^{2g}$ of $V$ in a minimal way, cf. Lemma \ref{lem_straightening_basis}). The spin representation of $S$ may be identified, as a vector space, with the exterior algebra $\bigwedge^\ast \langle p_{2g}, \dots, p_{g+1} \rangle$: a basis is provided by the standard basis of monomials in $p_{2g}, \dots, p_{g+1}$. (This ordering is taken in order that the formulae for the embedding $K_f \to \P^{2^g-1}$ line up with the existing ones in \cite{CasselsFlynnProlegomena, Mul14} in the cases $g = 2, 3$.)

 The properties of the Kummer embedding are then an exercise in the study of the spin representation, and pure spinors in particular. (Elements of the spin representation are called spinors, and those representing points in the image of the embedding $\Sigma : \OGr(g, V) \to \P(S)$ are called pure spinors -- hence the name of this paper.)
Especially, we show that the theta group $\mathcal{G}(2 \Theta)$ may be identified with a subgroup of the Clifford group $\Gamma$, itself a subgroup of the units in the Clifford algebra $C(V)$. The matrices $M_{\widetilde{T}}$ giving the action of an element $M_{\widetilde{T}} \in \mathcal{G}(2 \Theta)$ on $\HH^0(J_f, \cO_{J_f}(2 \Theta))$ may then be computed using multiplication in the Clifford algebra. The comparison between the dagger height $h^\dagger$ and the induced naive height $h$ is a consequence of the explicit description of the morphism $\Psi$.

It seems likely that our description of the image $K_f$ of $\Psi$ as a subvariety of a rational subvariety of $\P^{2^g-1}$ of dimension $g(g+1)/2$ might be useful in efficiently searching for rational points on $K_f$. Relatedly, we also describe, for any $P \in J_f(k)$, a twisted analogue of morphism $\Psi$, which is closely related to the soluble 2-covering of $J_f$ determined by the point $P$ (see \S \ref{subsec_twisted_Kummer}). 

It remains to explain how we approach the proofs of Theorems \ref{introthm_lower_bound_for_canonical_height_in_terms_of_naive_height} and \ref{introthm_lower_bound_for_canonical_height} on the canonical height. Theorem \ref{introthm_lower_bound_for_canonical_height_in_terms_of_naive_height} is proved by following the method of \cite{Sto99}, using our description of the embedding $\Psi$ and the duplication morphism in terms of linear algebra. To prove Theorem \ref{introthm_lower_bound_for_canonical_height}, we need to take a slightly different approach, introducing a fourth height $\widetilde{h} : J_f(\Q) \to \R$, that we call the reduction height, and which plays an important role in the proofs of the main theorems of \cite{lagathorne2024smallheightoddhyperelliptic}. For a point $P \in J_f(\Q)$ with Mumford triple $(U, V, R)$, say with $(U, f) = 1$, this may be defined by the formula
\[ \widetilde{h}(P) = \frac{1}{2} h^\infty( 1 : u_1 : \dots : u_m) + \frac{1}{2} \log \sum_{i=1}^{2g+1} \frac{ |U(\omega_i)|}{|f'(\omega_i)|}, \]
where $\omega_1, \dots, \omega_{2g+1} \in \bC$ are the roots of $f(x)$, and $h^\infty$ denotes the non-archimedean contribution to the naive height of the point $[1 : u_1 : \dots : u_m]\in \P^m(\Q)$. This reduction height differs from $\frac{1}{2} h^\dagger$ only in the contribution of the infinite place; the methods of \emph{op. cit.} show that, as $f$ varies, $\widetilde{h}$ cannot be `too small, too often'. Again using our description of the Kummer embedding, combined with the rather explicit results on theta functions of hyperelliptic curves given in Mumford's book \cite{MumfordTataII}, we compare explicitly the local reduction height and the local canonical height at the infinite place. Inputting this result into the strategy of \cite{lagathorne2024smallheightoddhyperelliptic} then allows us to prove Theorem \ref{introthm_lower_bound_for_canonical_height}.

\paragraph{Structure of this paper.}  We now describe the structure of this paper. In \S \ref{sec_odd_orthogonal_spaces}, we develop the needed theory of the Clifford algebra, its spin representation $S$, and its pure spinors, in the case of an odd-dimensional quadratic space. Nothing here is new, but it is important for our applications to have precise statements (and clear sign conventions!). In \S \ref{sec_hyperelliptic_curves_and_kummers}, we carry out our main constructions, defining the morphism $\Psi : J \to \P(S)$ and relating it to the $|2 \Theta|$-linear system of $J$, and characterizing its image in terms of linear algebra. 

In \S \ref{sec_duplication_on_the_Kummer}, we make the action of the theta group $\mathcal{G}(2 \Theta)$ on $S$ explicit, and use this to give a description of the duplication morphism $[2] : K \to K$. We work out an explicit example in the case $g = 4$, $k = \F_5$. Finally, in \S \ref{sec_applications_to_heights}, we use this theory to study the canonical height $\widehat{h}_\Theta$ on the Jacobian of an odd hyperelliptic curve, and establish our main Theorem \ref{introthm_lower_bound_for_canonical_height}.

\section{Odd orthogonal spaces and the spinor variety}\label{sec_odd_orthogonal_spaces}

In this section, we record some basic results about quadratic spaces of odd rank, and develop some of the theory of the Clifford algebra and its associated spin representation $S$, including the  description of the embedding $\Sigma : \OGr(g, V) \to \P(S)$ in terms of pure spinors. This material has been well-understood since at least the works of Cartan and Chevalley \cite{Car67, Che97}. However, we have developed what we need from scratch here, for several reasons. First, it is important for our intended applications to have explicit formulae, without sign ambiguities, which requires a careful description of our conventions. Second, we have not been able to find a reference that treats the exactly the results we need in the required level of generality. 

Throughout \S \ref{sec_odd_orthogonal_spaces}, we fix a field $k$ of characteristic not 2, an integer $g \geq 1$, and write $V = (V, \psi)$ for a regular quadratic space over $k$ of dimension $2g + 1$ (i.e. a $k$-vector space equipped with a non-degenerate symmetric bilinear form $\psi$).

\subsection{The basic example}\label{subsec_basic_definitions}

\begin{lemma}\label{lem_straightening_basis}
    Let $\mathcal{B} = (b_0, \dots, b_{2g})$ be a $k$-basis of $V$ such that for each $0 \leq i, j \leq 2g$ such that $i+j < 2g$, $\psi(b_i, b_j) = 0$. Then there exists a unique $k$-basis $\mathcal{P} = (p_0, \dots, p_{2g})$ of $V$ satisfying the following conditions:
    \begin{enumerate}
        \item $p_i = b_i$ for each $i = 0, \dots, g$.
        \item $p_{g+i} \in b_{g+i} + \langle b_{g+i-1}, \dots, b_{g-i+1}, b_{g-i} \rangle$ for each $i = 1, \dots, g$.
        \item For each $i = 0, \dots, 2g$ and for each $0 \leq i, j \leq 2g$ such that $i +j \neq 2g$, $\psi(p_i, p_j) = 0$.
    \end{enumerate}
\end{lemma}
\begin{proof}
    We show by induction on $i = 0, \dots, 2g$ that we can find unique vectors $p_j$ ($j = 0, \dots, i$) satisfying (1) and (2), and which moreover satisfy $\psi(p_r, p_s) = 0$ for all $0 \leq r \leq s \leq i$ such that $r + s \neq 2g$. If $0 \leq i \leq g$, then it is clear that $p_j = b_j$ is the only choice. Suppose instead that $i = g + l$ for some $l \geq 1$. Then $p_0, \dots, p_{g+{l-1}}$ are uniquely determined, by induction. We are looking for $p_i = p_{g+l}$ of the form $p_{g+l} = b_{g+l} + \lambda_{l-1} b_{g+{l-1}} + \dots + \lambda_{-l} b_{g-l}$ for constants $\lambda_{l-1}, \dots, \lambda_{-l} \in k$. 

    We have $\psi(p_{g-r}, p_{g+l}) = 0$ if $r > l$, while $\psi(p_{g-l+1}, p_{g+l}) = \psi(b_{g-l+1}, b_{g+l}) + \lambda_{l-1} \psi(b_{g-l+1}, b_{g+l-1})$. Since $\psi(b_{g-l+1}, b_{g+l-1}) \neq 0$ by the nondegeneracy of $\psi$, this gives a unique choice for $\lambda_{l-1}$. Considering the equations $\psi(p_{g-r}, p_{g+l}) = 0$ for $r = l-2, \dots, 1-l$ shows that similarly $\lambda_{l-2}, \dots, \lambda_{1-l}$ are determined. Finally, we find $\psi(p_{g+l}, p_{g+l}) = 2 \lambda_{-l} \psi(b_{g+l}, b_{g-l}) + \dots$. Since $\psi(b_{g+l}, b_{g-l}) \neq 0$, there is again a unique choice for $\lambda_{-l}$. This concludes the proof. 
\end{proof}
Our standard example of a quadratic space will be $V = k[x] / (f(x))$, where $f(x) = x^{2g+1} + c_1 x^{2g} + \dots + c_{2g+1} \in k[x]$ is a monic polynomial of degree $2g+1$, equipped with the symmetric bilinear form $\psi(a, b) = \tau(ab)$, where the linear form $\tau : V \to k$ is defined by $\tau(a_0 + a_1 x + \dots + a_{2g} x^{2g}) = a_{2g}$. In this case, the basis $\mathcal{B} = (b_0, b_1, \dots, b_{2g}) = (1, x, \dots, x^{2g})$ of $V$ satisfies the hypotheses of Lemma \ref{lem_straightening_basis}. The associated basis $\mathcal{P}$ may be given explicitly as follows:
\begin{align}\label{eq_straightened_basis}
 p_i &= x^i, \,\, i = 0, \dots, g; \\
p_{g+i} &= x^{g+i} + c_1 x^{g+i-1} + c_2 x^{g+i-2} + \dots + c_{2i-1} x^{g-i+1} + \frac{1}{2} c_{2i} x^{g-i},\,\,\,i = 1, \dots, g. 
\end{align}
In particular, the Gram matrix of $\psi$ with respect to this basis is given by $[\psi]_\mathcal{P} = J$, where $J$ denotes the matrix with $1$'s on the antidiagonal and $0$'s elsewhere. This shows that $\langle p_0, \dots, p_{g-1} \rangle$ and $\langle p_{g+1}, \dots, p_{2g} \rangle$ are transverse maximal isotropic subspaces of $V$.
\begin{lemma}\label{lem_change_of_basis_matrix_is_integral}
    With notation as in the previous paragraph, define a matrix $P$ by $b_j = \sum_i P_{ij} p_i$. Then the coefficients $P_{ij}$ lie in $\frac{1}{2} \Z[c_1, \dots,  c_{2g+1}]$.
\end{lemma}
\begin{proof}
    We have $b_i = p_i$ if $0 \leq i \leq g$ and \[ b_{g+i} = p_{g+i} - c_1 b_{g+i-1} - \dots - c_{2i-1} b_{g-i+1} - \frac{1}{2} c_{2i} p_{g-i} \]
    if $1 \leq i \leq g$. The result follows by induction on $i$. 
\end{proof}

\subsection{The Clifford algebra and its spin representation}\label{subsec_Clifford_algebra}

We recall that the Clifford algebra $C(V)$ of $V$ is defined by a universal property: it is a unital associative (not-necessarily-commutative) $k$-algebra, and comes equipped with a $k$-linear map $\alpha : V \to C(V)$ such that for any $v \in V$, $\alpha(v)^2 = \psi(v, v)$. Moreover, for any analogous pair $(C', \alpha')$, there is a unique $k$-algebra map $\beta : C(V) \to C'$ such that $\alpha' = \beta \circ \alpha$.

The Clifford algebra $C(V)$ can be constructed explicitly as the quotient of the tensor algebra $T(V) = \oplus_{k \geq 0} V^{\otimes k}$ by the two-sided ideal generated by the relations $v \otimes v - \psi(v, v)$, $v \in V$. In particular, it has a natural $\Z / 2 \Z$-grading $C(V) = C_0(V) \oplus C_1(V)$, inherited from the $\Z$-grading of $T(V)$. The following lemma is standard.
\begin{lemma}
    Let $\mathcal{B} = (b_0, \dots, b_{2g})$ be a $k$-basis of $V$. Then the monomials $b_I = b_{i_1} \dots b_{i_r}$ associated to subsets $I = \{ i_1 < i_2 < \dots < i_r \} \subset \{0, \dots, 2g \}$ form a $k$-basis of $C(V)$. In particular, $\dim_k C(V) = 2^{\dim_k V}$.
\end{lemma}
We now suppose that $V$ admits a basis $\mathcal{P} = (p_0, \dots, p_{2g})$, where $\psi(p_i, p_j) = 1$ if $i + j = 2g$, and $0$ otherwise. We set $F = \langle p_0, \dots, p_{g-1} \rangle$, $E = \langle p_{g+1}, \dots, p_{2g}\rangle$. In this case, we can write down a $C(V)$-module $S$ that we call the spin representation of $C(V)$. 
The underlying vector space of $S$ is the exterior algebra $\bigwedge^\ast E$. To write down a map $\alpha_S\colon C(V)\rightarrow \End(S)$ encoding the $C(V)$-action on $S$, it suffices, by the universal property of $C(V)$, to describe the action of degree 1 elements in $V \leq C(V)$. The action on an element $\omega \in \bigwedge^\ast E = S$ is given according to the following formulae: 
\begin{itemize}
    \item If $e \in E$, then $\alpha_S(e)(\omega) = e \wedge \omega$.
    \item If $f \in F$, then $\alpha_S(f)(\omega) = f \lrcorner \, \omega$ (defined on monomials $\omega = w_1 \wedge \dots \wedge w_k$ by 
    \[ f \lrcorner \, \omega = \sum_{i=1}^k (-1)^{i-1} 2 \psi(f, w_i) w_1 \wedge \dots \wedge \hat{w}_i \wedge \dots \wedge w_k). \]
    \item If $\omega \in \wedge^k E$, then $\alpha_S(p_g)(\omega) = (-1)^k \omega$.
\end{itemize}
By a standard result in Clifford algebras, there is an isomorphism of $k$-algebras $C(V) = C_0(V) \otimes_k k[z] / (z^2 - 1)$, where $z \in C(V)$ is the central element $p_g \prod_{i=0}^{g-1} (1 - p_{2g-i} p_{i})$, and $C_0(V)$ is a central simple algebra over $k$. Thus $S$ is, up to isomorphism, the unique simple $C(V)$-module on which $z$ acts as the identity. 

Let $x \mapsto x^\ast$ be the anti-involution on $C(V)$ that acts on monomials  as $(v_1 \dots v_r)^\ast = (-1)^r v_r \dots v_1$ (its existence follows from the universal property of $C(V)$). It is the composite of the involution $\sigma(v_1 \dots v_r) = (-1)^r v_1 \dots v_r$ with the anti-involution $\tau(v_1 \dots v_r) = v_r \dots v_1$. We define the Clifford group $\Gamma$ (either as a group, or as a linear algebraic group over $k$) by the formula
\[ \Gamma = \{ x \in C(V)^\times \mid \sigma(x) V x^{-1}\leq V \}. \]
The even Clifford group $S \Gamma$ is the subgroup
\[ S \Gamma = \{ x \in C_0(V)^\times \mid \sigma(x) V x^{-1} \leq V \}.\]
By its very definition, $\Gamma$ acts on $V$. We record some properties of the Clifford group.
\begin{lemma}
    \begin{enumerate}
        \item The action of $\Gamma$ on $V$ is linear and preserves $\psi$, therefore defining a morphism $\rho : \Gamma \to \OO(V)$.
        \item The kernel of $\rho$ equals the centre $k^\times$ of $C_0(V)$ (and can therefore be identified with $\G_m$). We get short exact sequences of linear algebraic $k$-groups
        \[ 1 \to \G_m \to \Gamma \to \OO(V) \to 1 \]
        and
        \[ 1 \to \G_m \to S \Gamma \to \SO(V) \to 1. \]
        \item $S \Gamma$ is a connected reductive group.
        \item If $x \in \Gamma$, then $x x^\ast \in \G_m$. We write $N : \Gamma \to \G_m$ for the induced homomorphism, and call it the Clifford norm. The element $z$ satisfies $N(z) = (-1)^{g+1}$, $\rho(z) = -1$.
        \item Define $\Spin(V) = \ker(N) \cap S \Gamma$. Then $\Spin(V)$ is a connected reductive group, and we have a further short exact sequence
        \[ 1 \to \mu_2 \to \Spin(V) \to \SO(V) \to 1. \]
    \end{enumerate}
\end{lemma}
\begin{proof}
    See e.g.\ \cite[Appendix C]{Con14}; our group $\Gamma$ agrees with the group denoted $\operatorname{GPin}$ there. 
\end{proof}
The $C(V)$-action on $S$ restricts to an action of $\Gamma$ on $S$. In fact, $S$ is a highest weight module of the reductive group $S \Gamma$ (see \S \ref{subsec_orthogonal_grassmannian} below). The $\Gamma$-module $S$ carries a covariant bilinear form $\beta$, that we define next. If $\omega, \omega' 
\in S = \bigwedge^\ast E$, then we define $\beta(\omega, \omega')$ to be the coefficient of $p_{2g} \wedge \dots \wedge p_{g+1}$ in $\omega \wedge (\omega')^\ast$, where $(\cdot)^\ast : \bigwedge^\ast E \to \bigwedge^\ast E$ is again defined on monomials by $(v_1 \wedge \dots \wedge v_r)^\ast = (-1)^r v_r \wedge \dots \wedge v_1$.

It follows from the definition that $\beta(\omega, \omega') = (-1)^{g(g+1)/2} \beta(\omega', \omega)$ for all $\omega, \omega' \in S$. Thus $\beta$ is symmetric if $g \equiv 0, 3 \text{ mod }4$ and antisymmetric if $g \equiv 1, 2 \text{ mod }4$. 
\begin{lemma}\label{lem_covariance_of_beta}
    For any $\omega, \omega' \in S$, $x \in \Gamma$, we have $\beta(x \omega, x \omega') = N(x) ( \det \rho(x) )^{g+1} \beta(\omega, \omega')$. 
\end{lemma}
As a sanity check for the above formula, note that $z \in \Gamma$ acts trivially on $S$ but has $N(z) = (-1)^{g+1}$, $\det \rho(z) = -1$. 
\begin{proof}
    Let $f = (1 + p_g) p_{g-1} \dots p_0$. Then the map $S = \bigwedge^\ast E \to C(V) f$, $\omega \mapsto \omega f$, is an isomorphism of left $C(V)$-modules. We give another definition of $\beta$ using this realisation of $S$. There is a surjection $C(V) \to \bigwedge^{2g+1} V$ that annihilates monomials of degree $< 2g+1$ and sends $v_1 \dots v_{2g+1}$ to $v_1 \wedge v_2 \wedge \dots \wedge v_{2g+1}$. We write $\pi' : C(V) \to k$ for the composite of this map with division by $p_{2g} \wedge \dots \wedge p_0$. We will use the relations $\pi'(zxy) = \pi'(zyx)$ and $\pi'(z \tau(x)) = \pi'(z x^\ast) = \pi'( zx)$ for any $x, y \in C(V)$ (these are most easily checked by diagonalising $\psi$ over the algebraic closure and then checking on monomials in orthonormal basis elements -- note that the map $C(V) \to \bigwedge^{2g+1} V$ underlying $\pi'$ is independent of any choice of basis). 

    Suppose first that $g$ is even. Then $\tau(f) = (-1)^{g(g-1)/2} f$, and we define $\beta' : C(V) f \times C(V) f \to k$ by the formula $\beta'(xf, yf) = \pi'(z x f \tau(y))$. This is easily seen to be a well-defined bilinear form, which is $(-1)^{g(g-1)/2}$-symmetric (use $\pi'(z x f \tau(y)) = \pi'(z y \tau(f) \tau(x)) = (-1)^{g(g-1)/2} \pi'(z y f \tau(x))$). If $w \in \Gamma$, then we have
    \[ \beta'(w x f, w y f) = \pi'(z w x f \tau(y) \tau(w)) = N(w) \det(\rho(w)) \beta'(x f, y f) \]
    (here we use the formula $\tau(w) = \det \rho(w) w^\ast$ for any $w \in \Gamma$, which holds since every element in $\Gamma$ lies either in $C_0(V)$ or $C_1(V)$). The proof of the lemma in the case $g$ even follows on noting that the isomorphism $S \cong C(V) f$ identifies $\beta$ with a scalar multiple of $\beta'$. 

    Now suppose that $g$ is odd. Then $f^\ast = (-1)^{g(g+1)/2} f$, and we define $\beta'(x f, y f) = \pi'(z x f y^\ast)$. We check again that $\beta'$ is $(-1)^{g(g+1)/2}$-symmetric and that $\beta$ is identified with a scalar multiple of $\beta'$. This completes the proof.
\end{proof}
Since we have identified the underlying vector space of $S$ with $\bigwedge^\ast E$, $S$ has a basis of monomials $p_I = p_{i_r} \wedge \dots \wedge p_{i_1}$, indexed by subsets $I = \{ i_1 < \dots < i_r \} \subset \{g+1, \dots, 2g\}$. We give the monomials the unique ordering such that $I < J$ if $|I| > |J|$ or if $|I| = |J|$ and $I$ precedes $J$ in the reverse lexicographic ordering. 
For example, the first monomial in the basis is $p_{2g} \wedge p_{2g-1} \wedge \dots \wedge p_{g+1}$, the second is $p_{2g} \wedge p_{2g-1} \wedge \dots \wedge p_{g+2}$, and if $g=2$ then the ordering is $(p_4 \wedge p_3, p_4, p_3, 1)$, where $1  = p_{\emptyset}$.
(This ordering will the the right one for later applications to the embedding of the Kummer variety in $\P(S)$.) The following lemma is immediate from the above discussion.
\begin{lemma}\label{lem_computation_of_beta}
    If $I = \{ i_1 < \dots < i_r \}$, $J = \{ j_1 < \dots < j_s \}$ are subsets of $\{ g+1, \dots, 2g \}$, then we have
    \[ \beta(p_I, p_J) = \left\{ \begin{array}{cc} 0 & \text{ if } I \neq J^c; \\
    (-1)^{s(s+1)/2} \epsilon(\sigma)  & \text{ if }I = J^c, \end{array}\right. \]
    where $\sigma$ is the permutation of $\{ g+1, \dots, 2g \}$ sending $i_r, \dots, i_1$ to $2g, \dots, 2g+1-r$, and $j_s, \dots, j_1$ to $2g-r, \dots, g+1$, respectively. 
\end{lemma}
The form $\beta$ defines an invariant line in the tensor product $S \otimes_k S$. In fact, a similar device can be used to decompose the whole of $S \otimes_k S$, as we now explain. 
\begin{proposition}
    \begin{enumerate}
        \item The exterior algebra $\bigwedge^\ast V$ has a natural structure of a $C(V)$-module, a vector $v \in V \leq C(V)$ acting by $v \cdot \omega = v \wedge \omega + v \lrcorner \,\,\omega$ (contraction defined on monomials by $v \lrcorner (w_1 \wedge \dots \wedge w_k) = \sum_{i=1}^k (-1)^{i-1} \psi(v, w_i) w_1 \wedge \dots \wedge \hat{w}_i \wedge \dots w_k$; note the missing factor of 2 compared to the definition of $\alpha_S$ above). 
        \item The map $\theta : C(V) \to \bigwedge^\ast V$, $x \mapsto x \cdot 1$, is an isomorphism of left $C(V)$-modules.
        \item The map $\theta$ satisfies $\theta(\tau(x)) = \tau(\theta(x))$ for all $x \in C(V)$. 
    \end{enumerate}
\end{proposition}
Here $\tau : \bigwedge^\ast V \to \bigwedge^\ast V$ is again defined on monomials by $\tau( v_1 \wedge \dots \wedge v_r ) = v_r \wedge \dots \wedge v_1$.
\begin{proof}
    (1) The given formula defines a linear map $f : V \to \End(\bigwedge^\ast V)$. By the universal property of $C(V)$, this will extend to $C(V)$ if we can show that $f(v) \circ f(w) + f(w) \circ f(v) = 2 \psi(v, w) \cdot \operatorname{Id}_{\bigwedge^\ast V}$. This is a simple calculation.

    (2) Fix an orthogonal basis $x_1, \dots, x_{2g+1}$ of $V$. Then we can check that for any $i_1 < \dots < i_r$, we have $\theta(x_{i_1} \dots x_{i_r}) = x_{i_1} \wedge \dots \wedge x_{i_r}$. In particular, $\theta$ is an isomorphism.

    (3) This property can be checked on basis elements, in which case it follows from the calculation in the previous paragraph. 
\end{proof}
\begin{proposition}
    Let $\wedge \rho : \Gamma \to \GL(\bigwedge^\ast V)$ denote the exterior product of the standard representation $\rho : \Gamma \to \OO(V)$, and let $\wedge' \rho$ denote the representation of $\Gamma$ on $\oplus_{r = 0}^{2g+1} (\bigwedge^r V ) \otimes \det(\rho)^{r-1}$. Then for any $x \in \Gamma$, $y \in C(V)$, we have the formula $(\wedge' \rho)(x)(\theta(y)) = \theta(\sigma(x) y x^{-1})$. In other words, $\theta$ defines an isomorphism $C(V) \to \bigwedge^\ast V$ of $\Gamma$-modules, where $\Gamma$ acts on $C(V)$ by $\sigma$-conjugation and on $\bigwedge^\ast V$ by $\wedge' \rho$.
\end{proposition}
\begin{proof}
    The restriction of $\theta$ to $V$ is the identity map $V \to V$ (where $V$ is identified with a subspace of $C(V)$ and of $\bigwedge^\ast V$ in the natural way).  If $v \in V$ and $x \in \Gamma$, then we have $\rho(x)(v) = \sigma(x) v x^{-1}$, by definition of $\rho$; thus the result of the proposition holds if $y =v \in V$. To prove the general case, it suffices to check that the formula $\theta(\sigma(x) y x^{-1}) = (\wedge' \rho)(x)(\theta(y))$ holds for all $y$ in a spanning set of $C(V)$. We can take the set of vectors $y$ of the form $y = y_1 \dots y_r$, where $y_1, \dots, y_r \in V$ are pairwise orthogonal. Then, using that $\rho(x)(y_1), \dots, \rho(x)(y_r)$ are also pairwise orthogonal, we find
    \begin{multline*}  \theta( \sigma(x) y x^{-1}) = \det \rho(x) \theta( x y_1 x^{-1} \dots x y_r x^{-1}) = \det \rho(x)^{r-1} \theta(\rho(x)(y_1) \dots \rho(x)(y_r)) \\ = \det \rho(x)^{r-1}\rho(x)(y_1) \wedge \dots \wedge \rho(x)(y_r) = (\wedge' \rho)(x)(y_1 \wedge \dots \wedge y_r) = (\wedge' \rho)(x)(\theta(y)), 
    \end{multline*}
    as required. 
\end{proof}
\begin{proposition}\label{prop_decomposing_S_tensor_S}
    There are linear maps $\beta_r : S \otimes_k S \to \bigwedge^r V$, $0 \leq r \leq 2g+1$ even, with the following properties:
        \begin{enumerate}
            \item $\oplus_{0 \leq r \leq 2g+1 \text{ even }} \beta_r : S \otimes_k S \to \oplus_{0 \leq r \leq 2g+1 \text{ even}} \bigwedge^r V$ is an isomorphism of $k$-vector spaces.
            \item For any $x \in \Gamma$ and $\omega, \omega' \in S$, we have $\beta_r(x \omega, x \omega') = \det \rho(x)^{g+1} N(x) (\wedge^r \rho)(x)(\beta_r(\omega, \omega'))$.
            \item For any $\omega, \omega' \in S$, we have $\beta_r(\omega', \omega) = (-1)^{g(g+1)/2+r(r+1)/2} \beta_r(\omega, \omega')$. 
        \end{enumerate}
\end{proposition}
In particular, we see that if $g$ is even, then  $\beta_r$ is symmetric (resp. anti-symmetric) if $r \equiv g \text{ mod }4$ (resp. if $r \equiv g-2 \text{ mod }4$). If $g$ is odd, then $\beta_r$ is symmetric (resp. anti-symmetric) if $r \equiv g+1 \text{ mod }4$ (resp. if $r \equiv g-1 \text{ mod }4$).
\begin{proof}
    Suppose first that $g$ is even, and let $f = (1+p_g) p_{g-1} \dots p_0 \in C(V)$, as in the proof of Lemma \ref{lem_covariance_of_beta}, so that $\tau(f) = (-1)^{g(g-1)/2} f$. We define a map $\zeta : S \otimes_k S \to \bigwedge^\ast V$ by the formula $\zeta(\omega, \omega') = \theta( \omega f \tau(\omega'))$, and define $\beta_r$ as the projection of $\zeta$ onto $\bigwedge^rV$.
    If $x \in \Gamma$, then $x \tau(x) = N(x) \det \rho(x)$, so $\tau(x) = x^{-1} N(x) \det \rho(x)$. We can thus compute
    \begin{multline*} \zeta(x \omega, x \omega') = \theta( x \omega f  \tau(\omega') \tau(x)) = N(x) \theta (\sigma(x) \omega f \tau(\omega') x^{-1}) \\ = N(x) (\wedge' \rho)(x)(\theta( \omega f \tau(\omega'))) = N(x) (\wedge' \rho)(x) (\zeta(\omega, \omega')). 
    \end{multline*}
    This immediately gives
    \[ \beta_r(x \omega, x \omega') = N(x) \det \rho(x)^{r-1} \beta_r(\omega, \omega'). \]
    On the other hand, we have
    \[ \zeta(\omega', \omega) = \theta( \omega' f \tau(\omega)) = (-1)^{g(g-1)/2} \tau( \theta(\omega f \tau(\omega'))), \]
    hence
    \[ \beta_r(\omega', \omega) = (-1)^{g(g-1)/2 + r (r-1)/2} \beta_r(\omega, \omega'). \]
    To complete the proof in the case $g$ even, we need to show that $\zeta$ is an isomorphism onto the even part of $\bigwedge^\ast V$. Recall the central element $z \in \Gamma$ satisfies $z f = f$. It follows that the map $S \otimes_k S \to C(V)$ defined by $\omega \otimes \omega' \mapsto \omega f \omega'$ takes values in $C(V)^{z = 1} = (z+1) C(V)$. This is a 2-sided ideal of $C(V)$ of dimension $2^{2g}$ (because there is an isomorphism $C(V) \cong C_0(V) \otimes_k k[z] / (z^2-1)$, and $C_0(V)$ is a central simple $k$-algebra). Therefore the map $S \otimes_k S \to C(V)^{z=1}$ is an isomorphism. Since $\rho(z) = -1$, $\theta$ restricts to an isomorphism between $C(V)^{z=1}$ and the even part of $\bigwedge^\ast V$, completing the proof in this case.
    
    Now suppose that $g$ is odd; with the same definition of $f$, we have $f^\ast = (-1)^{g(g+1)/2} f$. In this case we define a map $\zeta : S \otimes_k S \to \bigwedge^\ast V$ by the formula $\zeta(\omega, \omega') = \theta( \omega f (\omega')^\ast)$, and again define $\beta_r$ as the projection of $\zeta$ onto $\bigwedge^rV$. If $x \in \Gamma$, then $x x^\ast = N(x)$, so $x^\ast = x^{-1} N(x)$. We compute
    \begin{multline*} \zeta(x \omega, x \omega') = \theta(x \omega f (\omega')^\ast x^\ast) \\ = N(x) \det \rho(x) \theta(\sigma(x) \omega f (\omega')^\ast x^{-1}) = N(x) \det \rho(x) (\wedge' \rho)(x)(\theta (\omega f (\omega')^\ast) \\ = N(x) \det \rho(x) (\wedge' \rho)(x) \zeta(\omega, \omega'). 
    \end{multline*}
    This immediately gives
    \[ \beta_r(x \omega, x \omega') = N(x) \det \rho(x)^r \beta_r(\omega, \omega'). \]
    On the other hand, writing $\pi_r : \bigwedge^\ast V \to \bigwedge^r V$ for the projection to the degree $r$ part, we have 
    \begin{multline*} \beta_r(\omega', \omega) = \pi_r(\theta(\omega' f \omega^\ast)) \\ = (-1)^{r(r-1)/2+g(g+1)/2} \pi_r (\theta(\sigma( \omega f (\omega')^\ast))) = (-1)^{r(r+1)/2 + g(g+1)/2} \pi_r (\theta (\omega f (\omega')^\ast)) \\ = (-1)^{r(r+1)/2+g(g+1)/2} \beta_r(\omega, \omega'), 
    \end{multline*}
    here using the formula $\pi_r \circ \theta \circ \sigma = (-1)^r \pi_r \circ \theta$, which can be checked on an orthogonal basis. The same argument as in the case $g$ even shows that $\zeta$ is an isomorphism onto the even part of $\bigwedge^\ast V$. This completes the proof. 
\end{proof}

\begin{corollary}\label{cor_decomposing_symmetric_square_of_S}
    \begin{enumerate}
        \item Suppose that $g$ is even. Then there is an $S \Gamma$-equivariant isomorphism $\Sym^2 S \cong \oplus_{\substack{0 \leq r \leq 2g+1 \\ r \equiv g \text{ mod }4}} (\bigwedge^r V) \otimes N$.
        \item Suppose that $g$ is odd. Then there is an $S \Gamma$-equivariant isomorphism $\Sym^2 S \cong \oplus_{\substack{0 \leq r \leq 2g+1 \\ r \equiv g+1 \text{ mod }4}} (\bigwedge^r V) \otimes N$.
    \end{enumerate}
\end{corollary}
The proof of Proposition \ref{prop_decomposing_S_tensor_S} constructs explicit linear maps $\beta_r\colon S\otimes_kS\rightarrow \bigwedge^r V$.
The following proposition completes the discussion.
\begin{proposition}
    The form $\beta$ is a nonzero scalar multiple of $\beta_0$. 
\end{proposition}
\begin{proof}
    We again split into cases, depending on the parity of $g$. If $g$ is even, then the proof of Lemma \ref{lem_covariance_of_beta} shows that $\beta$ is a nonzero scalar multiple of the form $\beta'(x, y) = \pi'(z x f \tau(y))$, where $\pi' : C(V) \to k$ may be defined as the composite of $\theta : C(V) \to \bigwedge^\ast V$ with the projection $\bigwedge^\ast V \to \bigwedge^{2g+1} V \to k$ sending $p_{2g} \wedge \dots \wedge p_0$ to $1$. On the other hand, $\beta_0(x, y)$ is defined as $\pi( x f \tau(y))$, where $\pi : C(V) \to k$ may be defined as the composite of $\theta : C(V) \to \bigwedge^\ast V$ with the projection $\bigwedge^\ast V \to \bigwedge^0 V = k$. 

    To complete the proof, it suffices therefore to show that the forms $C(V) \to k$ given by $x \mapsto \pi(x)$ and $x \mapsto \pi'(zx)$ are scalar multiples of each other. To check this, let $y_1, \dots, y_{2g+1}$ be an orthogonal basis of $V$, so that the monomials in the $y_i$ form a $k$-basis of $C(V)$, and moreover so that $z$ is a scalar multiple of $y_1 y_2 \dots y_{2g+1}$. Then it is clear that $\pi$ and $\pi'$ both annihilate all monomials except the empty one (i.e. $1 \in k \leq C(V)$), so they must be scalar multiples of each other. 

    The proof in the case that $g$ is odd is identical, replacing $\tau$ by $(\cdot)^\ast$.
\end{proof}
\subsection{The orthogonal Grassmannian}\label{subsec_orthogonal_grassmannian}

Let $\mathrm{Gr}(g,V)$ be the usual Grassmannian of $g$-dimensional subspaces of $V$.
This is a flag variety for $\GL(V)$ over $k$.
For a $k$-scheme $T$, the set of $T$-points $\Gr(g,V)(T)$ equals the set of equivalence classes of pairs $(\sh{F}, \alpha)$, where $\sh{F}$ is a locally free sheaf of $\O_T$-modules of rank $g$ and $\alpha\colon \sh{F}\rightarrow V\otimes \O_T$ is an injective sheaf homomorphism whose cokernel is locally free.

Let $\mathrm{OGr}(g,V)$ be the orthogonal Grassmannian of $g$-dimensional isotropic subspaces of $V$.
Given a $k$-scheme $T$, $\OGr(g,V)(T)$ consists of equivalence classes of pairs $(\sh{F}, \alpha)\in \Gr(g,V)(T)$ with the additional property that $\sh{F}$ is isotropic with respect to the form $\psi\colon (V\otimes \O_T) \times (V\otimes \O_T) \rightarrow \O_T$.
The obvious map $\OGr(g,V)\rightarrow \Gr(g,V)$ is a closed immersion.
\begin{lemma}
    Let $L\leq V$ be a $g$-dimensional isotropic subspace.
    Then the assignment $g\mapsto g\cdot L$ induces an isomorphism $\SO(V) / \Stab_{\SO(V)}(L) \cong \OGr(g,V)$.
    In particular, $\OGr(g,V)$ is smooth and $\SO(V)(k)$ acts transitively on $\OGr(g,V)(k)$.
\end{lemma}
\begin{proof}
    Let $P_L = \Stab_{\SO(V)}(L)$, a parabolic subgroup of $\SO(V)$. Then the quotient $\SO(V) / P_L$ is a projective algebraic variety and the morphism $\SO(V) \to \SO(V) / P_L$ is locally trivial in the Zariski topology (cf. \cite[Part II, 1.10]{Jan03}). In particular, the map $\SO(V)(k) \to (\SO(V) / P_L)(k)$ is surjective. On the other hand, $\SO(V)$ acts smoothly on $\OGr(g, V)$, and the action is transitive on $k$-points, by Witt's Lemma; therefore the orbit map gives an isomorphism $\SO(V) / P_L \cong \OGr(g, V)$. The result follows. 
\end{proof}
We now discuss projective embeddings. Let us first recall that if $X$ is a variety over $k$ and $W$ is a finite-dimensional $k$-vector space, then to give a morphism $\varphi : X \to \P(W)$ is equivalent to giving an invertible sheaf $\sh{M}$ on $X$, together with a $k$-linear map $\gamma : W^\vee \to\HH^0(X, \sh{M})$ such that the associated morphism $\cO_X \otimes_k W^\vee \to \sh{M}$ of sheaves of $\cO_X$-modules is surjective. Under this dictionary, we may take $\sh{M} = \varphi^\ast \cO_{\P(W)}(1)$, and $\gamma$ to be the obvious map induced by the identification $W^\vee =\HH^0(\P(W), \cO_{\P(W)}(1))$.

The universal point $\Gr(g,V)\xrightarrow{\Id} \Gr(g,V)$ determines a locally free sheaf $\sh{F}^{\mathrm{univ}}\subset V\otimes \O_{\mathrm{Gr}(g,V)}$. 
Define $\O_{\mathrm{Gr}(g,V)}(1) = (\det \sh{F}^{\mathrm{univ}})^{\vee}$. There is a tautological map $\wedge^g V^\vee \to \HH^0(\Gr(g, V), \O_{\mathrm{Gr}(g,V)}(1))$, corresponding to a morphism $\Gr(g, V) \to \P(\wedge^g V)$, the Pl\"ucker embedding.

We now consider the analogous situation for the orthogonal Grassmannian. 
Assume that we are given a basis $\mathcal{P} = (p_0, \dots, p_{2g})$ of $V$ with $\psi(p_i, p_j) = \delta_{i,2g-j}$ for all $i,j$, and continue with the notation of Section \ref{subsec_Clifford_algebra}.
We have seen that $S \Gamma$ is a connected reductive group, which comes with a surjection $S \Gamma \to \SO(V)$ with central kernel. Let $P = \Stab_{S \Gamma}(F)$. 
\begin{lemma}\label{lem_action_of_maximal_parabolic_on_vacuum_vector}
    $P$ is a parabolic subgroup of $S \Gamma$, which acts by a character $\chi : P \to \G_m$ on the line $\bigwedge^0 E \leq \bigwedge^\ast E = S$.
\end{lemma}
\begin{proof}
    We have seen that $\rho : S \Gamma \to \SO(V)$ is a surjective morphism of reductive groups with kernel $\G_m$. The subgroup $P \leq S \Gamma$, as the pre-image of a parabolic subgroup of $\SO(V)$, is again a parabolic subgroup. 
    
    Consider the subgroup
    \[ T = \left\{ \lambda \prod_{i=0}^{g-1} (1 + \mu_i p_i p_{2g-i}) \mid \lambda \prod_{i=0}^{g-1} (1 + 2 \mu_i) \neq 0 \right\} \]
    of $S \Gamma$. This is a split maximal torus (since it is a split torus, and has rank $g+1$). The restriction of $\rho$ to $T$ sends the generic element of $T$ to the diagonal matrix (with respect to the basis $p_0, p_1, \dots, p_{2g}$)
    \begin{equation}\label{eqn_rho_on_T} \operatorname{diag}(1 + 2 \mu_0, \dots, 1 + 2 \mu_{g-1}, 1, (1 + 2 \mu_{g-1})^{-1}, \dots, (1 + 2 \mu_0)^{-1}). 
    \end{equation}
    Writing $\varepsilon_i : T \to \G_m$ ($i = 0, \dots, g-1$) for the homomorphism given by $1 + 2 \mu_i$, a choice of positive roots is $\varepsilon_i / \varepsilon_j$ ($0 \leq i < j \leq g-1$), $\varepsilon_i$ ($0 \leq i \leq g-1$), and $\varepsilon_i \varepsilon_j$ ($0 \leq i \leq j \leq g-1$). By the structure theory of reductive groups, $P$ may be identified as the subgroup of $S \Gamma$ generated by $T$ and the root subgroups corresponding to the roots $(\varepsilon_i / \varepsilon_j)^{\pm 1}$ ($0 \leq i \neq j \leq g-1$), $\varepsilon_i$ ($0 \leq i \leq g-1$), and $\varepsilon_i \varepsilon_j$ ($0 \leq i \leq j \leq g-1$). 

    We describe the corresponding root subgroups. These are given by
    \[ U_{\varepsilon_i / \varepsilon_j} = \{ 1 + t p_i p_{2g-j} \mid t \in \G_a \} \,\, (0 \leq i \neq j \leq g-1); \]
     \[ U_{\varepsilon_i} = \{ 1 + t p_g p_i \mid t \in \G_a \} \,\, (0 \leq i\leq g-1); \]
   \[ U_{\varepsilon_i^{-1}} = \{ 1 + t p_g p_{2g-i} \mid t \in \G_a \} \,\, (0 \leq i\leq g-1); \]
    \[ U_{\varepsilon_i \varepsilon_j} = \{ 1 + t p_i p_j \mid t \in \G_a \} \,\, (0 \leq i, j \leq g-1); \]
    and
     \[ U_{(\varepsilon_i \varepsilon_j)^{-1}} = \{ 1 + t p_{2g-i} p_{2g-j} \mid t \in \G_a \} \,\, (0 \leq i, j \leq g-1). \]
     To show that the line $\bigwedge^0 E \leq S$ is invariant under $P$, it suffices to check that it is invariant under $T$ and also under the root subgroups corresponding to the roots $(\varepsilon_i / \varepsilon_j)^{\pm 1}$ ($0 \leq i \neq j \leq g-1$), $\varepsilon_i$ ($0 \leq i \leq g-1$), and $\varepsilon_i \varepsilon_j$ ($0 \leq i \leq j \leq g-1$). (Since $P$ is a maximal parabolic subgroup of $S \Gamma$, and $S \Gamma$ does not leave this line invariant, this will also show that we have $P = \Stab_{S \Gamma}(\bigwedge^0 E)$.) 

     We compute first that the generic element $t \in T$ acts on $\bigwedge^0 E$ by the scalar $\chi(t) = \lambda \prod_{i=0}^{g-1} (1 + 2 \mu_i)$. On the other hand, the given root groups all act trivially on $\bigwedge^0 E$. This completes the proof. 
\end{proof}
Under the equivalence between $S \Gamma$-equivariant vector bundles on $S \Gamma / P = \OGr(g, V)$ and representations of $P$ on finite-dimensional $k$-vector spaces, the line $k(\chi)$ corresponds to a line bundle $L^{-1}$ on $\OGr(g, V)$. By Frobenius reciprocity (\cite[I.3.3, Proposition]{Jan03}), we have identifications
\[ \Hom_P(k(\chi), S) = \Hom_P(S^\vee, k(\chi^{-1})) = \Hom_{S \Gamma}(S^\vee, \HH^0(\OGr(g, V), L)). \]
By the Borel--Weil theorem, the morphism $S^\vee \rightarrow  \HH^0(\OGr(g, V), L)$ corresponding to the tautological inclusion $k(\chi) \leq S$ is an isomorphism, which determines an embedding $\Sigma : \OGr(g, V) \to \P(S)$ sending the point $F \in \OGr(g, V)(k)$ to the line $k \cdot 1 \leq S$. For reasons that will become apparent in the next section, the image of $\OGr(g, V)$ under $\Sigma$ is sometimes known as the spinor variety. 

It follows from the results of \cite{Pop74} that $L$ generates the Picard group of $\OGr(g,V)$.
We write this line bundle $L$ as $\mathcal{O}_{\OGr(g,V)}(1)$.
The next lemma will be used in the proof of Proposition \ref{prop_firstpropertieskummerembedding}. 
\begin{lemma}\label{lemma_spinorembedding_squareroot_plucker}
    The restriction of $\O_{\Gr(g,V)}(1)$ to $\OGr(g,V)\hookrightarrow \Gr(g,V)$ is isomorphic to $\O_{\OGr(g,V)}(2) = \O_{\OGr(g,V)}(1)^{\otimes 2}$.
    Consequently, if $\sh{F}\subset V\otimes\O_{\OGr(g,V)}$ denotes the universal rank $g$ subbundle on $\OGr(g,V)$, then $(\det \sh{F})^{\vee} \simeq \O_{\OGr(g,V)}(2)$.
\end{lemma}
\begin{proof}
    Let $P' = \Stab_{\Spin(V)}(F) = P \cap \Spin(V)$. Since $\OGr(g, V)$ is a homogeneous space for $\Spin(V)$, it suffices to show that the two characters $(\chi|_{P'})^{2}$ and $g \mapsto \det \rho(g|_F)$ of $P'$ are equal. This is the case if and only if they have the same restriction to $T \cap P'$, where $T$ is the maximal torus of $S \Gamma$ defined in the proof of Lemma \ref{lem_action_of_maximal_parabolic_on_vacuum_vector}.

    The restriction of $N$ to $T$ is given by $N(t) = \lambda^2 \prod_{i=0}^{g-1} (1 + 2 \mu_i) = \lambda \chi(t)$. Using Equation (\ref{eqn_rho_on_T}), we see that the restriction of $g \mapsto \det \rho(g|_F)$ is given by $\prod_{i=0}^{g-1} (1 + 2 \mu_i)$. Finally, $\chi(t) = \lambda \prod_{i=0}^{g-1} (1 + 2 \mu_i)$. It follows that $\chi(t)^2 / \det \rho(t|_F) = N(t)$. In particular, the restriction of this character to $T \cap P'$ is trivial, as required.
\end{proof}

\subsection{Isotropic subspaces and pure spinors}\label{subsec_properties_of_pure_spinors}

Let us continue to take $V$ to be a quadratic space over $k$ of dimension $2g+1$, equipped with a basis $p_0, \dots, p_{2g}$ satisfying $\psi(p_i, p_j) = \delta_{i, 2g-j}$, and a corresponding realisation $S$ of the spin representation of $C(V)$. In \S \ref{subsec_orthogonal_grassmannian} we have described an embedding $\Sigma : \OGr(g, V) \to \P(S)$. We now give the concrete description of this embedding in terms of pure spinors (defined below). 
\begin{proposition}\label{prop_which_spinors_are_pure}
    Let $\omega \in S$ be non-zero, and define $L(\omega) = \{ v \in V \mid v \omega = 0 \}$ (where $v$ acts through its image in the Clifford algebra $C(V)$). Then:
    \begin{enumerate}
        \item $L(\omega) \leq L(\omega)^\perp$. In particular, $\dim_k L(\omega) \leq g$ and $L(\omega)$ is an isotropic subspace of $V$.
        \item If $\dim_k L(\omega) = g$, then $L(\omega)$ is a maximal isotropic subspace of $V$.
        \item For every isotropic subspace $L \leq V$ of dimension $g$, there exists a non-zero element $\omega_L \in S$, uniquely determined up to $k^\times$-multiple, such that $L(\omega_L) = L$. 
    \end{enumerate}
\end{proposition}
\begin{proof}
    (1) Suppose $x, y \in L(\omega)$. Then $x y + y x = 2 \psi(x, y)$, hence $(xy  + yx) \omega = 2 \psi(x, y) \omega = 0$. Since $\omega \neq 0$, this implies $\psi(x, y) = 0$.

    (2) Immediate from (1).

    (3) Let $\omega = 1 \in \wedge^0 E \leq S$. Then $F \leq L(\omega)$; since $\dim_k L(\omega) \leq g$, we must have $L(\omega) = F$. The group $\SO(V)(k)$ acts transitively on $\OGr(g, V)(k)$, and the map $S \Gamma(k) \to \SO(V)(k)$ is surjective, so for any other maximal isotropic subspace $L \leq V$, we can find an element $x \in S \Gamma(k)$ such that $x F x^{-1} = L$. A simple computation shows that $L(x \omega) = \rho(x)( L(\omega) ) = \rho(x)(F) = L$, hence we can take $\omega_L = x \omega$. 

    This shows the existence of $\omega_L$. We need to show uniqueness. It suffices to show that if $\omega' \in S$ is a non-zero vector with $L(\omega') = F$, then $\omega' \in k \cdot 1$. Suppose for contradiction that $\omega' \not\in k \cdot 1$. We can write $\omega' = \sum_I a_I p_I$; let $l > 0$ be maximal such that there is $I \subset \{ g+1, \dots, 2g \}$ with $|I| = l$ and $a_I \neq 0$. If $i \in I$, then $p_{2g-i} \cdot \omega' \neq 0$, contradicting the assumption that $p_{2g-i} \in L(\omega')$. This contradiction concludes the proof. 
\end{proof}
An element of $S$  of the form $\omega_L$ is called a pure spinor. Proposition \ref{prop_which_spinors_are_pure} shows that the set of maximal isotropic subspaces of $V$ is in bijection with the subset of $\P(S)(k)$ consisting of lines spanned by pure spinors. In fact, the closed embedding $\OGr(g, V) \to \P(S)$ described in \S \ref{subsec_orthogonal_grassmannian} can be described on $k$-points as $L \mapsto \omega_L$: 
\begin{proposition}
    The morphism $\Sigma : \OGr(g, V) \to \P(S)$ described in \S \ref{subsec_orthogonal_grassmannian} sends $L \in \OGr(g, V)(k)$ to the line in $\P(S)(k)$ spanned by the associated pure spinor $\omega_L$.
\end{proposition}
\begin{proof}
    The discussion of \S \ref{subsec_orthogonal_grassmannian} shows that $F$ is sent to the pure spinor $\omega_F$. The morphism is $S \Gamma$-equivariant, by construction, so if $x \in S \Gamma$, then $\rho(x)( F )$ is sent to $x \omega_F$. However, it follows from the definition of pure spinor that $x \omega_F = \omega_{\rho(x)( F )}$. 
\end{proof}

\begin{proposition}\label{prop_OG_cut_out_by_quadrics}
    The homogeneous ideal \[\ker\left(\bigoplus_{n\geq 0} \HH^0(\P(S), \O_{\P(S)}(n)) \rightarrow \bigoplus_{n\geq 0}\HH^0(\OGr(g,V), \O_{\OGr(g,V)}(n))\right)\] is generated by quadrics, i.e., by the subspace $I(2)=\ker(\HH^0(\P(S), \O_{\P(S)}(2)) \rightarrow \HH^0(\OGr(g,V), \O_{\OGr(g,V)}(2)))$.
\end{proposition}
\begin{proof}
    This follows from \cite[Theorem 3.11, Part (i)]{Ramanathan-equationsschubert} by taking $X=Y = S\Gamma/P = \OGr(g,V)$.
\end{proof}
In fact, $I(2)$ is spanned by the `Cartan relations' $\beta_r(\omega, \omega)=0$ for $0\leq r < 2g$ even, where $\beta_r$ are the maps of Proposition \ref{prop_decomposing_S_tensor_S}; compare \cite[\S 107]{Car67} (which works over $\R$) and \cite[III.4.4]{Che97} (which works over any field of characteristic not 2). We will not use this description in this paper, although it is implicit in e.g.\ the proof of Proposition \ref{prop_quadrics_vanishing_on_K}.

We give some examples of pure spinors. First, we have $\omega_F = 1 \in \bigwedge^0 E \leq S$, and $\omega_E = p_{\{ 2g, \dots, g +1 \}} =  p_{2g}  \wedge p_{2g-1} \wedge \dots \wedge p_{g+1} \in \wedge^g E \leq S$. More generally, if $I = \{ i_1 < i_2 < \dots < i_r \} \subset \{ 2g, \dots, g+1 \}$ is a subset, we consider the subspace $L_I \leq V$ spanned by the vectors $\{ p_{i} \mid i \in I \} \cup \{ p_{2g-i} \mid i \not\in I \}$. 
Then $\omega_{L_I} = p_I \in \bigwedge^r E \leq S$. 

We now show how to compute the pure spinor associated to a general isotropic subspace $L$. We first need to define Pfaffians. Suppose given a matrix $(\xi_{ij})_{0 \leq i, j \leq g-1} \in M_{g \times g}(k)$ satisfying $\xi_{ij} + \xi_{ji} = 0$ for all $i, j$, together with a vector $(\xi_i)_{0 \leq i \leq g-1} \in k^g$. If $I \subset \{ 0, \dots, g-1 \}$ is a subset of size $2r$, then we define the Pfaffian
\[ \xi_I = \sum_{P} \epsilon(P) \xi_P, \]
where the sum is over all partitions $P$ of $I$ into pairs $\{ i_1 < j_1 \}, \dots, \{ i_r < j_r \}$ with $i_1 < i_2 < \dots < i_r$, $\epsilon(P)$ is the signature of the permutation sending the elements of $I$, in increasing order, to $i_1, j_1, i_2, j_2, \dots, i_r, j_r$, respectively, and $\xi_P = \prod_{k=1}^r \xi_{i_k j_k}$. 
(If $r=0$, then by definition $\xi_I = 1$.)
If instead $I \subset \{ 0, \dots, g-1 \}$ is a subset of size $2r+1$, then we define
\[ \xi_I = \sum_P \epsilon(P) \xi_P, \]
where now the sum is over all partitions $P$ of $I$ into a singleton $\{ i_0 \}$ and pairs $\{ i_1 < j_1 \}, \dots, \{ i_r < j_r \}$ with $i_1 < i_2 < \dots < i_r$, $\epsilon(P)$ is the signature of the permutation sending the elements of $I$, in increasing order, to $i_0, i_1, j_1, i_2, j_2, \dots, i_r, j_r$, respectively, and $\xi_P = \xi_{i_0} \prod_{k=1}^r \xi_{i_k j_k}$. (Similar expressions may be found in \cite[Ch. V]{Car67}; for example, if $g = 3$ and $I = \{ 1, 2, 3 \}$, then $\xi_I = \xi_1 \xi_{23} - \xi_2 \xi_{1 3} + \xi_3 \xi_{12}$.)

Now suppose that $L \leq V$ is a maximal isotropic subspace. To compute the associated pure spinor $\omega_L$, we first select a subset $J \subset \{ 2g, \dots, g+1 \}$ such that $L \cap L_{J^c} = 0$. (This is an open condition in $\OGr(g, V)$, and the corresponding subvarieties (as $J$ varies) give an open cover. We will see that if $L \cap L_{J^c} = 0$ then $\omega_L = \sum_I a_I p_I$ has the property that $a_J \neq 0$, and so we can normalise by taking $a_J = 1$.) To fix ideas, let us first assume that we can take $J^c = \emptyset$, so that $L_{J^c} = F$. In this case, we have $L \cap F^\perp = 0$, so the natural map $L \to V / F^\perp$ is an isomorphism. We can therefore find a unique basis $l_{2g}, \dots, l_{g+1}$ for $L$ such that $l_j = p_j + \xi_{2g-j} p_g + \sum_{i=0}^{g-1} A_{i 2g-j} p_i$ for each $j = 2g, \dots, g+1$ and for some constants $\xi_i$, $A_{ij} \in k$ ($0 \leq i, j \leq g-1$).\footnote{At this point, the reader may begin to feel that our choice of indices is somewhat eccentric. In applications, $p_j$ will be a monic polynomial of degree $j$, and it will be important to have formulae that can be applied easily; for this reason, we build this choice into the notation now, rather than labeling our basis elements as $p_0, \dots, p_{2g} = f_1, \dots, f_g, u, e_g, \dots, e_1$, a choice that would otherwise be convenient.} It is easy to see that the condition that $L$ is isotropic is equivalent to the condition that the matrix $\xi_{ij} = A_{ij} + \frac{1}{2} \xi_i \xi_j$ is antisymmetric. 

Given a subset $I \subset \{ 0, \dots, g-1 \}$ (resp. $\subset \{ 2g, \dots, g+1 \}$), let $\widehat{I} = \{ 2g-m \mid m \in I \} \subset \{ 2g, \dots, g+1 \}$ (resp. $\subset \{ 0, \dots, g-1 \}$).   
\begin{proposition}\label{prop_computation_of_pure_spinor}
    Let $L \leq V$ be a maximal isotropic subspace such that $L \cap F = 0$. Then we have 
    \[ \omega_L = \sum_{\substack{ I=\{i_1 < \dots < i_r\} \\ \subset \{ 0, \dots, g-1 \}}} (-1)^{i_1 + \dots + i_r + r(g+1)} 2^{\lfloor |I| /2 \rfloor} \xi_{I} p_{\widehat{I}^c}. \]
\end{proposition}
\begin{proof}
    We will use the action of the Clifford group. From Proposition \ref{prop_which_spinors_are_pure}, we see that if $\rho : \Gamma\to \OO(V)$ denotes the standard representation of the Clifford group, and $\gamma \in \Gamma$, then $\omega_{\rho(\gamma)(L)} = \gamma \omega_L$. We will prove the proposition by finding an element of $\Gamma$ that takes $E$ to $L$ (under $\rho$). 

    Let $f \in F$. Then a simple computation shows that $1 + p_g f \in \Gamma$, and moreover that for $v \in V$ we have $\rho(1 + p_g f)(v) = v + 2 \psi(v, f) p_g - 2 \psi(v, p_g)f - 2 \psi(v, f) f$. Similarly, if $f, f' \in F$ then $1 + f f' \in \Gamma$ and $\rho(1 + f f ')(v) = v + 2 \psi(v, f') f - 2 \psi(v, f) f'$. We consider the product
    \[ \gamma = \left(1 +   \frac{1}{2} p_g \sum_{0 \leq j \leq g-1} \xi_j p_j \right) \prod_{0 \leq i < j \leq g-1} \left(1 + \frac{1}{2} \xi_{ij} p_i p_j \right) \]
    (note that the bracketed terms can be multiplied in any order with the same result). A computation shows that $\rho(\gamma)(p_r) = l_r$ for each $r = 2g, \dots, g+1$, and therefore that $\omega_L = \gamma \omega_E = \gamma (p_{2g} \wedge \dots \wedge p_{g+1})$. On the other hand, by expanding out the product we find that 
    \[ \gamma = \sum_{k \geq 0} \sum_{\substack{I = \{ i_1 < i_2 < \dots < i_{2k} \} \\ \subset \{ 0, \dots, g-1 \}}} 2^{-k} \xi_I (p_{i_1} \wedge \dots \wedge p_{i_{2k}}) + \sum_{\substack{I = \{ i_1 < i_2 < \dots < i_{2k+1} \} \\ \subset \{ 0, \dots, g-1 \}}} 2^{-(k+1)} \xi_I p_g (p_{i_1} \wedge \dots \wedge p_{i_{2k+1}}). \]
    Note that if $I = \{ i_1 < \dots < i_r \} \subset \{ 0, \dots, g-1 \}$, then  $p_{i_1} \dots p_{i_r} p_{2g} \dots p_{g+1} = 2^r (-1)^{i_1 + \dots + i_r} p_{\widehat{I}^c}$ (using the formulae defining the spin representation in \S \ref{subsec_Clifford_algebra}). If $I = \{ i_1 < \dots < i_{2k} \}$, then
    \[ 2^{-k} \xi_I p_{i_1} \dots p_{i_{2k}} p_{2g} \dots p_{g+1} = 2^k \xi_I (-1)^{i_1 + \dots + i_{2k}} p_{\widehat{I}^c}. \]
    If $I = \{ i_1 < \dots < i_{2k+1} \}$, then
    \[ 2^{-(k+1)} \xi_I p_g p_{i_1} \dots p_{i_{2k}} p_{2g} \dots p_{g+1} = 2^k \xi_I (-1)^{i_1 + \dots + i_{2k+1} + g-(2k+1)} p_{\widehat{I}^c}. \]
    Summing these terms now gives the result in the statement of the proposition. 
\end{proof}
Returning to the general case, let $L \leq V$ be a maximal isotropic subspace, and let $J \subset \{2g, \dots, g+1 \}$ be a subset such that $L \cap L_{J^c} = 0$. In this case, the natural map $L \to V / L_{J^c}^\perp$ is an isomorphism, so we can find a unique basis $l_{2g}, \dots, l_{g+1}$ for $L$ such that $l_j = p_j' + \xi_{2g-j} p_g + \sum_{i=0}^{g-1} A_{i 2g-j} p_i'$ for each $j = 2g, \dots, g+1$, where we define $p'_j = p_j$ and $p'_{2g-j} = p_{2g-j}$ if $j \in J$ and $p'_j = p_{2g-j}$ and $p_{2g-j}' = p_j$ if $j \in J^c$. Once again, the condition that $L$ is isotropic is equivalent to the condition that the matrix $\xi_{ij} = A_{ij} + \frac{1}{2} \xi_i \xi_j$ is antisymmetric. In this case, the result is as follows: 
\begin{proposition}\label{prop_computation_of_pure_spinor_nongeneric}
    Let $J \subset \{2g, \dots, g+1 \}$ be a subset, and let $L \leq V$ be a maximal isotropic subspace such that $L \cap L_{J^c} = 0$. Then we have 
    \begin{align}\label{eq_computation_purespinor_nongeneric}
     \omega_L = \sum_{I \subset \{ 0, \dots, g-1 \}} \epsilon(I, J) 2^{-\lceil |I| /2 \rceil + |I \cap \widehat{J}| } \xi_I p_{\widehat{I} \Delta J}, 
    \end{align}
    where $\epsilon(I, J) \in \{ \pm 1 \}$ is a sign depending only on $I$ and $J$.
\end{proposition}
In the statement above, $I \Delta J = (I\cup J)\setminus (I\cap J)$ denotes the symmetric difference of the sets $I$ and $J$.
\begin{proof}
    Repeating the start of the proof of Proposition \ref{prop_computation_of_pure_spinor} verbatim shows that we have $\omega_L = \gamma \omega_{L_J} = \gamma p_J$, where now
    \begin{multline*} \gamma = \left(1 +   \frac{1}{2} p_g \sum_{0 \leq j \leq g-1} \xi_j p'_j \right) \prod_{0 \leq i < j \leq g-1} \left(1 + \frac{1}{2} \xi_{ij} p'_i p'_j \right) \\ =  \sum_{k \geq 0} \sum_{\substack{I = \{i_1 < \dots < i_{2k}\} \\ \subset \{ 0, \dots, g-1 \}}} 2^{-k} \xi_I p'_{i_1} \dots p'_{i_{2k}} + \sum_{\substack{I = \{ i_1 < \dots < i_{2k+1} \} \\ \subset \{ 0, \dots, g-1 \}}} 2^{-(k+1)} \xi_I p_g p'_{i_1} \dots p'_{i_{2k+1}}. 
    \end{multline*}
    To complete the proof, we observe that if $I = \{ i_1 < \dots < i_r \} \subset \{ 0, \dots, g-1 \}$, then we can write $p'_{i_1} \dots p'_{i_r} = \pm p_{I \cap \widehat{J}} p_{\widehat{I} \cap J^c}$, hence $p'_{i_1} \dots p'_{i_r} p_J = \pm 2^{|I \cap \widehat{J} |} p_{J \Delta \widehat{I}}$. 
\end{proof}
The following proposition shows that the form $\beta$ captures incidence of isotropic subspaces of $V$. 
\begin{proposition}\label{prop_incidence_in_terms_of_beta}
    Let $L, L' \leq V$ be maximal isotropic subspaces. Then the following are equivalent:
    \begin{enumerate}
        \item $L \cap L' \neq 0$.
        \item $\beta(\omega_L, \omega_{L'}) = 0$.
    \end{enumerate}
\end{proposition}
\begin{proof}
    Since $S\Gamma(k)$ acts transitively on $\OGr(g, V)(k)$, we can assume that $L = F$. In this case, the condition $\beta(\omega_L, \omega_{L'}) = 0$ is equivalent to the condition $a_{\{g+1, \dots, 2g\}} = 0$ (where we write $\omega_{L'} = \sum_I a_I p_I$). 

    If $F \cap L' = 0$, then Proposition \ref{prop_computation_of_pure_spinor} shows that $a_{\{ 2g, \dots, g+1 \}} \neq 0$. This shows that (2) $\Rightarrow$ (1). To show (1) $\Rightarrow$ (2), suppose that $F \cap L' \neq 0$. Then there exists a non-zero vector $f = \sum_{i=0}^{g-1} \lambda_i p_i$ such that $f \omega_{L'} = 0$. Suppose for contradiction that $a_{\{ 2g, \dots, g+1 \}} \neq 0$. If $i_0 \in \{ 0, \dots, g-1 \}$  is such that $\lambda_{i_0} \neq 0$, then we see that the coefficient of $p_{\{2g, \dots, g+1\} - \{ 2g-i_0 \}}$ in $f \omega_{L'}$ is non-zero -- in particular, $f \omega_{L'} \neq 0$, a contradiction. This completes the proof. 
\end{proof}

\begin{corollary}\label{corollary_nonvanishingcoeff_LI}
    A pure spinor $\omega_L = \sum_I a_I p_I$ satisfies $a_I \neq 0$ if and only if $L \cap L_{I^c} = 0$.
\end{corollary}
\begin{proof}
    Combine Proposition \ref{prop_incidence_in_terms_of_beta} with Lemma \ref{lem_computation_of_beta}.
\end{proof}

\section{The Kummer embedding via pure spinors}\label{sec_hyperelliptic_curves_and_kummers}

In this section, we describe the embedding of the Kummer variety of an odd hyperelliptic Jacobian, determined by the  linear system of $2 \Theta$, in terms of spinors. The realisation of the Jacobian of a hyperelliptic curve of genus $g$ as the variety of $g$-planes contained in the base locus of a pencil of quadrics in a $2g+2$-dimensional vector space is well-studied; see e.g.\ \cite{Des76} for a `Picard'-style construction and \cite{Don80} for an `Albanese'-style construction, both without regard to rationality questions. Wang \cite{Wan18} applied the point of view of Donagi from an arithmetic point of view, proving results that played a foundational role in the applications to arithmetic statistics in e.g. \cite{BhargavaGross}. Our realisation of the Kummer is closely related to these constructions; see especially \cite[\S 4.1]{Rei72}. In any case, our focus is on making this construction as explicit as possible. We first, in \S \ref{subsec_basic_version_keyconstruction}, define the embedding at the level of points. This description is used in \S \ref{subsec_explicit_Kummer_embedding}, together with the results on pure spinors developed in \S \ref{sec_odd_orthogonal_spaces}, to compute the morphism explicitly. In the remainder of \S \ref{sec_hyperelliptic_curves_and_kummers}, we show that the morphism exists at the level of schemes and study its geometric properties, including the relation to the linear system of $2 \Theta$ on the Jacobian and the system of equations that cut out the image. 

\subsection{Associating a pure spinor to an element of $J(k)$}\label{subsec_basic_version_keyconstruction}

Let $k$ be a field of characteristic not 2, let $g \geq 1$, and choose $f(x) = x^{2g+1} + c_1 x^{2g} + c_2x^{2g-1} + \cdots + c_{2g+1}\in k[x]$ a polynomial of nonzero discriminant. We define  the associated affine hyperelliptic curve $C^0 :y^2 =f(x)$, its smooth projective completion $C$, its marked Weierstrass point $P_{\infty} \in (C - C^0)(k)$, and its affine Weierstrass locus $W = (y = 0) \subset C^0$. We write $\iota : C \to C$ for the hyperelliptic involution.

We write $J = \Pic^0_C$ for the Jacobian of $C$, a $g$-dimensional abelian variety, and $K = J / \{ \pm 1 \}$ for the Kummer variety. 
The divisor $(g-1)P_\infty$ defines a $k$-point of $\Pic^{g-1}_ C$ and translating the locus of effective divisor classes $W_{g-1}\subset \Pic_C^{g-1}$ along $(g-1)P_{\infty}$ defines a theta divisor $\Theta = t_{(g-1)P_\infty}^\ast W_{g-1}$ on $J$.

Let $V = k[x] / (f(x))$. Let $\tau\colon V\rightarrow k$ be the linear functional sending $x^i$ to $0$ if $0\leq i \leq 2g-1$ and $x^{2g}$ to $1$.
Recall that the symmetric bilinear form $\psi(a,b)= \tau(ab)$ endows $V$ with the structure of a  split nondegenerate quadratic space over $k$.

Let $\sh{L}$ be an invertible sheaf on $C$ of degree $2g-1$.
Let $V_{\sh{L}} = \HH^0(W, \sh{L}|_W)$.
Then $\sh{L}|_W$ is an invertible $\O_W$-module, so $\dim V_{\sh{L}} = 2g+1$.
By the  Riemann--Roch theorem, $\dim \HH^0(C, \sh{L}) =g$. Since $\iota$ acts as $[-1]$ on $J$, we can and do choose an isomorphism $\sh{L}\otimes \iota^* \sh{L} \rightarrow \O_C((4g-2)P_{\infty})$ of invertible sheaves on $C$, which is determined uniquely up to $k^\times$-multiple.  By restriction, this determines a $k$-bilinear map $V_\sh{L} \otimes_k V_\sh{L} \to \HH^0(W, \O_C((4g-2)P_{\infty})|_W)$.
There is a canonical identification of $\O_C((4g-2)P_\infty)|_W$ with $\cO_W$; applying the induced identification of $\HH^0(W, \O_C((4g-2)P_\infty)|_W)$ with $V$ and postcomposing with $\tau\colon V\rightarrow k$, we obtain a bilinear map $\psi_{\sh{L}} : V_{\sh{L}} \times V_{\sh{L}} \rightarrow k$, uniquely determined up to $k^{\times}$-multiple.
\begin{lemma}\label{lem_space_of_sections_is_isotropic}
    The form $\psi_{\sh{L}}$ is a nondegenerate symmetric bilinear form, the restriction map $\HH^0(C, \sh{L}) \rightarrow V_{\sh{L}}$ is injective, and the image $F_{\sh{L}}$ is an isotropic subspace of dimension $g$. 
\end{lemma}
\begin{proof}
    The fact that $\psi_{\sh{L}}$ is nondegenerate follows from \cite[Lemma 2.6]{Thorne-remark}.
    Write $j\colon W \hookrightarrow C$ for the inclusion. Applying $\HH^0(C,-)$ to the exact sequence  
    \[
    0\rightarrow \sh{L}(-W) \rightarrow \sh{L}\rightarrow j_*j^*\sh{L}\rightarrow 0
    \]
    and using the fact that $\deg(\sh{L}(-W)) = -2$ so $\HH^0(C, \sh{L}(-W)) = 0$ shows that the restriction map $F_\sh{L}\rightarrow V_{\sh{L}}$ is injective.

    We calculate that the image of the composition $F_{\sh{L}}\times F_{\sh{L}} \rightarrow V_{\sh{L}} \times V_{\sh{L}} \xrightarrow{\psi_{\sh{L}}} k$ equals the image of the composition
    \[
    \HH^0(C, \sh{L}) \times \HH^0(C, \iota^*\sh{L})
    \rightarrow \HH^0(C,\O_C((4g-2)P_{\infty})) \rightarrow V\xrightarrow{\tau}k,
    \]
    where the first map is induced by the isomorphism $\sh{L}\otimes \iota^*\sh{L}\rightarrow \O_C((4g-2)P_{\infty})$ and the second map is restriction along $W\subset C$.
    The image of every element of $\HH^0(C, \O_C((4g-2)P_\infty)) = \langle 1, x, \dots, x^{2g-1}, y, \dots, x^{g-2}y\rangle$ in $V$ lies in the subspace $\langle 1, x, \dots, x^{2g-1}\rangle = \ker \tau$. So $F_{\sh{L}}$ is indeed isotropic.
\end{proof}
Supposing $\sh{L}$ to be fixed, take $\sh{N}$ to be an invertible sheaf on $C$ of degree 0, and let $\sh{M} = \sh{L}\otimes \sh{N}^{\otimes 2}$. Then there is an isomorphism $\sh{N} \otimes \iota^* \sh{N} \rightarrow \O_{C}$ of sheaves of $\cO_C$-modules, uniquely determined up to $k^\times$-multiple. The sheaves $\sh{N}^{\otimes 2}$ and $\sh{N} \otimes \iota^\ast \sh{N}$ become canonically isomorphic after pullback to $W$, so we obtain an isomorphism $\varphi^{\sh{N}}_{\sh{M}, \sh{L}}\colon V_{\sh{M}} \rightarrow V_{\sh{L}}$, determined up to $k^{\times}$-multiple.
This isomorphism identifies the bilinear forms $\psi_{\sh{M}}$ and $\psi_{\sh{L}}$ up to a fixed scalar, and so determines a canonical isomorphism $\OGr(g, V_{\sh{M}}) \to \OGr(g, V_{\sh{L}})$. The image $\varphi^{\sh{N}}_{\sh{M}, \sh{L}}(F_{\sh{M}})$ is an isotropic subspace of $V_{\sh{L}}$ which depends only on $\sh{N}$.

Now take $\sh{L} = \O((2g-1)P_{\infty})$, so that we have an identification of quadratic spaces $V_{\sh{L}} = V$.
Then the above recipe defines a map 
\begin{align}\label{equation_firstversion_keyconstruction}
\Psi(k)\colon J(k) \rightarrow \OGr(g,V)(k),
\:\:\sh{N}\mapsto \varphi_{\sh{M}, \sh{L}}^{\sh{N}}(F_{\sh{M}}).
\end{align}
The remainder of this section (and much of this paper) is devoted to studying this important map. 
In the following subsections we will:
\begin{enumerate}
    \item Explicitly describe $\Psi(k)$, and its composition with $\Sigma(k)$, in terms of the Mumford representation of a point $[D] \in J(k)$.
    \item Extend $\Psi(k)$ to a morphism of schemes $\Psi\colon J\rightarrow \OGr(g, V)$. (In fact, we will do this universally and for singular curves.)
    \item Show that the composition of $\Psi$ with the spinor embedding $\Sigma: \OGr(g,V)\hookrightarrow \P(S)$ is isomorphic to the Kummer embedding associated to the complete linear system $|2\Theta|$.
\end{enumerate}

\subsection{Explicit description of $\Psi(k)$}\label{subsec_explicit_Kummer_embedding}

We first recall the definition of the Mumford representation, which is based on the following lemma. We write $\bar{k} / k$ for a choice of algebraic closure. 
\begin{lemma}
    Every element of $J(k)$ has a unique representative of the form $D = E - m P_\infty$, where $0\leq m\leq g$ and $E$ is an effective divisor with the property that when writing $E_{\bar{k}} = \sum_{i=1}^m P_i$ with $P_i \in C^0(\bar{k})$, we have $\iota(P_i) \neq P_j$ if $i \neq j$.
\end{lemma}
\begin{proof}
    Given an element of $J(k)$ corresponding to a degree 0 line bundle $\sh{L}$ on $C$, let $m$ be the minimal nonnegative integer such that $\HH^0(C, \sh{L}(mP_{\infty}))\neq0$. 
    Then $m\leq g$ by Riemann--Roch.
    Write $\sh{L} = \O_C(E - mP_{\infty})$ where $E$ is effective.
    Then $E_{\bar{k}} = \sum_{i=1}^m P_i$ with $P_i \in C^0(\bar{k})$. 
    If $P_i = \iota(P_j)$ for some $i\neq j$, then using the relation $P_i + P_j \sim 2P_{\infty}$, we see that $\HH^0(C_{\bar{k}}, \sh{L}_{\bar{k}}((m-2)P_{\infty})\neq 0$, so $\HH^0(C, \sh{L}((m-2)P_{\infty})\neq 0$, contradicting minimality of $m$. 
    
    This proves existence of such a representative. We now show uniqueness. Suppose that $E - m P_\infty$, $E' - m' P_\infty$ are linearly equivalent divisors of the given form, where $m' \geq m$. Let $f \in k(C)^\times$ satisfy $(f) = E - E' + (m' - m) P_\infty$, and let $U, U' \in k[x]$ be the polynomials vanishing on the $x$-co-ordinates of the points in $E, E'$ respectively. Then $(f U') = E + \iota(E') - (m + m') P_\infty$. We have $\HH^0(C, \cO_C((m + m')P_\infty) = \langle 1, x, \dots, x^{\lfloor (m + m')/2 \rfloor} \rangle$. Since $E'_{\overline{k}} = \sum_{i=1}^{m'} P_i'$ satisfies $P_i \neq \iota(P'_j)$ if $i \neq j$, the condition that $f U'$ is a polynomial in $x$ that vanishes on $\iota(E')$ implies that $f U'$ is divisible by $U'$. Since $\deg f U' \leq (m + m')/ 2 \leq m'$, the only possibility is that $m = m'$, $f$ is constant, and $E = E'$, as required.
\end{proof}

A divisor $D=E- mP_{\infty}$ satisfying the conclusion of the above lemma is called (Mumford) reduced. 
We call the integer $m$ the Mumford degree of the divisor class $[D]$; the divisor classes of Mumford degree $g$ are precisely those in the complement of $\Theta = t_{(g-1)P_\infty}^\ast W_{g-1} \leq J$.

\begin{lemma}[Mumford representation]\label{lemma_Mumford_rep}
    Let $0\leq m \leq g$.
    The following two sets are in bijection:
    \begin{enumerate}
        \item The set of divisor classes  $[D] \in J(k)$ of Mumford degree $m$.
        \item The set of triples $(U, V, R)$, where $U, V, R \in k[x]$, $U$ is monic of degree $m$, $V$ is monic of degree $2g + 1 - m$, $\deg R \leq m-1$, and $f - R^2 = UV$.
    \end{enumerate}
    This bijection may be described as follows: to a triple $(U, V, R)$ we associate the class of the divisor $E - m P_\infty$, where $E \subset C^0$ is the effective divisor described by the ideal $(U, y-R) \leq k[x] / (y^2 - f(x))$. 
\end{lemma}
\begin{proof}
    When $k=  \mathbb{C}$, this is \cite[Chapter IIIa, \S1, Proposition 1.2]{MumfordTataII}; we briefly sketch why the same proof works for a general $k$ in which $2$ is invertible.
    Using the previous lemma, it suffices to prove that the association $(U,V,R)\mapsto E - mP_{\infty}$ induces a bijection between triples $(U,V,R)$ and Mumford reduced divisors of Mumford degree $m$.
    To show this, we write down an inverse.
    Let $E - mP_{\infty}$ be such a divisor.
    Since $E$ is effective, it corresponds to a closed subscheme $i\colon Z\hookrightarrow C^{0}$.
    Let $\pi\colon C^{0}\rightarrow \A^1$ be the double cover map.
    We claim that $\pi\circ i \colon Z\rightarrow \A^1$ is again a closed immersion. To check this claim, we may assume $k$ is algebraically closed, in which case it follows from the assumption that $E$ is (Mumford) reduced.
    
    Since the scheme $Z\rightarrow \Spec(k)$ is finite of degree $m$, there exists a unique monic degree $m$ polynomial $U(x)$ such that $\pi\circ i\colon Z\hookrightarrow \A^1$ is cut out by the ideal of $k[x]$ generated by $U(x)$.
    Let $\pi_Z\colon \tilde{Z}\rightarrow Z$ be the pullback of $\pi$ along $\pi \circ i$.
    This morphism $\pi_Z$ of affine schemes corresponds to the inclusion of $k$-algebras $\alpha\colon A \rightarrow A[y]/(y^2 - f(x))$, where $A = k[x]/(U(x))$.
    The morphism $i\colon Z\rightarrow C^{0}$ determines a section of $\pi_Z$, in other words a $k$-algebra homomorphism $\beta\colon A[y]/(y^2- f(x)) \rightarrow A$ satisfying $\beta\circ \alpha = \Id_{A}$.
    The element $\beta(y) \in A = k[x]/((U(x))$ has a unique representative $R(x)$ of degree $\leq m-1$. 
    Since $\beta(y)^2 = \beta(y^2) = f(x) \mod U(x)$, the polynomial $f(x)-R(x)^2$ is divisible by $U(x)$, hence there exists a unique polynomial $V(x) \in k[x]$, necessarily monic of degree $2g+1-m$, such that $f-R^2 = UV$.

    We have just assigned a triple $(U,V,R)$ to a reduced divisor of degree $m$. 
   We omit the verification that this provides an inverse to the map $(U,V,R)\mapsto E - mP_{\infty}$.
\end{proof}

Recall that the morphism $\Sigma(k) \circ \Psi(k) : J(k) \to \P(S)(k)$ can be thought of as associating to $[D] \in J(k)$ the pure spinor representing the image of the isotropic subspace $\HH^0(C, \cO_C((2g-1)P_\infty + 2D) \leq \HH^0(W, \cO_C((2g-1)P_\infty + 2D)|_W)$ under the identification $\HH^0(W, \cO_C((2g-1)P_\infty + 2D) \cong \HH^0(W, \cO_C((2g-1)P_\infty)|_W)$. Our first task is therefore to describe this image in terms of the triple $(U, V, R)$ associated to $D$.

We can assume that $D = E - m P_\infty$ is reduced. Since $D$, $-D$ have the same image under $\Psi(k)$, it suffices to compute
\[ \HH^0(C, \cO_C((2g-1)P_\infty - 2D) = \HH^0(C, \cO_C((2g+2m -1)P_\infty) - 2 E). \]
We compute this space as the intersection of $\HH^0(C^0, \cO_{C^0}(-2E)) = (U, y-R)^2 \leq k[x] / (y^2-f(x))$ and $\HH^0(C, \cO_C((2g+2m-1)P_\infty)) = \langle 1, x, \dots, x^{g+m-1}, y-R, \dots, x^{m-1}(y-R))$. To this end, note that the ideal $I = (U, y-R) \leq k[x,y] / (y^2 - f(x))$ is a free $k[x]$-module of rank 2 with basis $U$, $y-R$ (cf. \cite[\S 4]{Thorne-remark}). In this basis, the square ideal $I^2 = (U, y-R)^2$ is identified with the $k[x]$-submodule generated by the vectors $(U, 0)$, $(0, U)$, and $(V, -2R)$ (where we have used the identity $(y-R)^2 = UV - 2R(y-R)$). For polynomials $a, b \in k[x]$, the vector $(a, b) = a U + b(y-R)$ lies in   $\HH^0(C, \cO_C((2g+2m-1)P_\infty))$ if and only if $\deg a \leq g-1$ and $\deg b \leq m-1$. We see that the $g-m$ elements $(U, 0), \dots, (x^{g-m-1} U, 0)$ lie in this space, together with the $m$ elements $(a_i, b_i)$ for $i = 0, \dots, m-1$, where we have used Euclidean division to write $x^i V = p_i U + a_i, -2 x^i R = q_i U + b_i$ for unique polynomials $a_i, b_i \in k[x]$ of degree $\leq m-1$. (It will be useful to note that $U$ is monic, so this division can be done `universally' without introducing denominators.) 

To compute the images of these elements in $\HH^0(W, \cO_W) = V = k[x] / (f(x))$, we can choose an expression as a sum $\sum_j s_j \otimes s_j'$ of $I \otimes I$; form $\sum_j s_j \otimes \iota(s_j') \in I \otimes \iota(I)$; and then take the image $(\sum_j s_j \iota(s'_j) / U) \text{ mod } y \in \HH^0(W, \cO_W)$ (noting that isomorphism $\O_C(-D) \otimes \O_C(-\iota^\ast(D)) \cong \O_C$ is given by multiplication by $U^{-1}$). The elements $(x^i U, 0) = x^i U \otimes U$ are sent to $x^i U \text{ mod } f(x)$. On the other hand, the elements $(a_i, b_i) = x^i (V, -2 R) - p_i (U, 0) - q_i (0, U) = x^i (y-R) \otimes (y-R) - p_i U \otimes U - q_i U \otimes (y-R)$ are sent to
\[ ( x^i (y-R)(-y-R) - p_i U^2 - q_i U (-y-R) ) / U \text{ mod } y = -x^i V - p_i U + q_i R = -2 x^i V + a_i + q_i R \text{ mod }f(x)). \]
The degree of $x^i U$ is $m +i$, while the degree of $-2 x^j V + a_j + q_j R$ is $2g+1-m+j$. These degrees (for $0 \leq i \leq g-m-1$, $0 \leq j \leq m-1$) lie in the range $0, \dots, 2g$ and are all distinct, so we see that the $g$ elements we have constructed are linearly independent, and therefore form a basis for the given isotropic subspace. We can now record some important properties of the point $\Psi(k)([D])$:
\begin{proposition}\label{prop_properties_of_explicit_Kummer_morphism}
    Let $D$ be a reduced divisor, representing a class $[D] \in J(k)$ of Mumford degree $0 \leq m \leq g$. Let $L = \Psi(k)([D]) \in \OGr(g, V)(k)$. Then:
    \begin{enumerate}
        \item Let $J = \{ 2g, \dots, 2g-m+1 \}$. Then, writing $\omega_L = \sum_I a_I p_I \in S$ for the associated pure spinor, we have $a_J \neq 0$ and $a_I = 0$ if $|I| > m$.
        \item Write $U(x) = x^m + u_{1} x^{m-1} + \dots + u_m$, and normalise $\omega_L$ so that $a_J = 1$. Then $u_{1}, \dots, u_m$ appear, up to sign, as co-ordinates of $\omega_L$.
        \item Continuing to normalise $\omega_L$ so that $a_J = 1$, and writing $V(x) = x^{2g+1-m} + v_1 x^{2g-m} + \dots + v_{2g+1-m}$ and $R(x) = r_1 x^{m-1} + \dots + r_m$, the remaining co-ordinates $a_I$ of $\omega_L$ lie in $\Z[ \{ u_i, v_j, r_k \} ]$.
    \end{enumerate}
\end{proposition}
\begin{proof}
    (1) By Corollary \ref{corollary_nonvanishingcoeff_LI}, we see that to show $a_J \neq 0$, we need to show that $L \cap L_{J^c} = 0$, where $L_{J^c}$ is spanned by $1, x, \dots, x^{m-1}$ and  the elements $p_{g+1}, \dots, p_{2g-m}$ of our standard basis $\mathcal{P}$ (defined in \S \ref{subsec_basic_definitions}). Since $L$ is itself spanned by polynomials of degrees $m, \dots, g-1$ and $2g+1-m, \dots, 2g$, we see that the intersection is indeed 0. To see that $a_I = 0$ if $|I| > m$, we need to show that $L \cap L_{I^c} \neq 0$ if $|I| > m$. However, if $|I| > m$ then $L_{I^c}$ contains at least $m+1$ out of $1, x, \dots, x^{g-1}$. Since $\dim_k  L \cap \langle 1, x, \dots, x^{g-1} \rangle = g - m$, the intersection must be non-zero.

    (2) We split into cases according to whether $m = g$ or $m < g$. Suppose first that $m = g$. We recall the definitions of $(\xi_i)$ and $(\xi_{ij})$. The isotropic subspace  $L$ is the span of the uniquely defined vectors ($j = 2g, \dots, g+1$)
    \[ l_j = p_j + \xi_{2g-j}{p_g} + \sum_{i=0}^{g-1} A_{i 2g-j} p_i; \]
    then we set $\xi_{ij} = A_{ij} + \frac{1}{2} \xi_i \xi_j$. We claim that $U \in L^\perp$. Indeed, $U$ is the image of the restriction of the section $U^2$ of $\HH^0(C, \cO_C(2g P_\infty - 2D))$ to $W$ under the map $\varphi_{\sh{M}, \sh{L}}^{\cO_C(-D)}$. The orthogonality follows by a similar argument to the proof of Lemma \ref{lem_space_of_sections_is_isotropic}, using the fact that the restriction of any section of $\HH^0(C, \cO_C((4g-1)P_\infty))$ to $\HH^0(W, \cO_W) = V$ is in the kernel of $\tau$. Since $U = p_g + u_1 p_{g-1} + \dots + u_g p_0$, we compute
    \[ 0 = \psi(l_j, U) = u_{j-g} + \xi_{2g-j}, \]
    hence $\xi_i = -u_{g-i}$ for each $i = 0, \dots, g-1$. By Proposition \ref{prop_computation_of_pure_spinor}, each $u_i$ ($i = 1, \dots, g$) therefore appears, up to sign, as a coefficient of $\omega_L$. 

    Suppose next that $m < g$. In this case, we recall that we have defined $p_j' = p_j$ and $p'_{2g-j} = p_{2g-j}$ if $j \in J$ and $p'_j = p_{2g-j}$ and $p'_{2g-j} = p_j $ if $j \in J^c$, and that $\xi_i, \xi_{ij}$ are defined by the equations ($j = 2g, \dots, g+1$)
    \begin{equation}\label{eqn_rappels_l_j} l_j = p'_j + \xi_{2g-j}{p_g} + \sum_{i=0}^{g-1} A_{i 2g-j} p'_i. 
    \end{equation}
    We see that $U = x^m + u_1 x^{m-1} + \dots + u_m = p'_{2g-m} + u_1 p'_{m-1} + \dots + u_m p'_0 = l_{2g-m}$, so $A_{i m} = u_{m-i}$ for each $i = 0, \dots, m-1$ and $\xi_m = 0$, hence $\xi_{im} = u_{m-i}$ for each $i = 0, \dots, m-1$. By Proposition \ref{prop_computation_of_pure_spinor_nongeneric}, each $u_i$ ($i = 1, \dots, m$) therefore appears, up to sign, as a coefficient of $\omega_L$. 

    (3) By Proposition \ref{prop_computation_of_pure_spinor_nongeneric}, the remaining co-ordinates $a_I$ are polynomials in the entries of $(\xi_i)$ and $(\xi_{ij})$; these are obtained by performing row reduction on the vectors expressing $x^i U$ ($i = 0, \dots, g-m-1)$ and $x^i V + \frac{1}{2} a_i - \frac{1}{2} q_i R$ ($i = 0, \dots, m-1$) in the basis $\mathcal{P}$ of $V = k[x] / (f(x))$ defined in \S \ref{subsec_basic_definitions}. We first check that the entries of $(\xi_i)$ and $(\xi_{ij})$ are polynomials in the coefficients of $U, V, R$ with coefficients in $\frac{1}{2} \Z$. Since the leading term in each row, before performing row reduction, is 1, it suffices to  check that this is true for the coefficients of  $x^i U$ and $x^i V + \frac{1}{2} a_i - \frac{1}{2} q_i R$ with respect to the basis $\mathcal{P}$. For $x^i U$, this is clear, as $\deg x^i U \leq g-1$ and $\{ 1, x, \dots, x^{g-1} \} \subset \mathcal{P}$. For $x^i V + \frac{1}{2} a_i - \frac{1}{2} q_i R$, we can check the claim separately for each of $x^i V$, $\frac{1}{2} a_i$, and $-\frac{1}{2} q_i R$. It is true for $x^i V$ because the coefficients of the change of basis matrix from $\mathcal{P}$ to $\mathcal{B}$ are polynomials in the coefficients of $f$ (and therefore also of $U, V, R$, since $f = U V + R^2$) with coefficients in $\frac{1}{2} \Z$, by Lemma \ref{lem_change_of_basis_matrix_is_integral}. It is true for $\frac{1}{2} a_i$, since $\deg a_i \leq g-1$. It is true for $-\frac{1}{2}q_i R$, since the equation $-2 x^i R = q_i U + b_i$ shows that the coefficients of $q_i$ lie in $2 \Z[ \{ u_r, r_s \}]$, and we can again use the properties of the change of basis matrix. 

    At this point, we can prove that (3) holds in the special case $m = g$. We use the formula 
    \[ \omega_L = \left(1 +    p_g \sum_{0 \leq j \leq g-1} \xi_j \frac{p_j}{2} \right) \prod_{0 \leq i < j \leq g-1} \left(1 +  (2 \xi_{ij}) \frac{p_i}{2}  \frac{p_j}{2} \right) \omega_F \]
    that appears in the proof of Proposition \ref{prop_computation_of_pure_spinor}. The definition of the action of $C(V)$ on $S$ in \S \ref{subsec_Clifford_algebra} shows that each matrix $\frac{p_i}{2}$ ($0 \leq i \leq g-1$) has integer entries (and likewise for $p_g$). Part (2) shows that each  $\xi_j \in \Z[ \{ u_r \}]$, while the argument just given shows that $2 \xi_{ij} \in \Z[\{ u_r, v_s, r_t \}]$. We conclude that $\omega_L$ has all of its co-ordinates in $\Z[\{ u_r, v_s, r_t \}]$ in this case. 

    Now let us suppose instead that $m < g$. In this case, we claim that the following further properties hold:
    \begin{itemize}
        \item If $j \in \widehat{J}^c$, then $\xi_j = 0$.
        \item If $i, j \in \widehat{J}^c$, then $\xi_{ij} = 0$.
        \item If $i \in \widehat{J}^c$, then $\xi_{ij} \in \Z[ \{ u_r \} ]$. 
    \end{itemize}
     (Recall that $\widehat{J} = \{ 0, \dots, m-1 \}$.) Before proving these properties, we show how they imply the desired result in this case also. We use the formula 
    \[ \omega_L = \left(1 +   p_g \sum_{0 \leq j \leq g-1} \xi_j \frac{p'_j}{2} \right) \prod_{0 \leq i < j \leq g-1} \left(1 +  2 \xi_{ij} \frac{p'_i}{2} \frac{p'_j}{2} \right) \omega_{L_J} \]
    from the proof of Proposition \ref{prop_computation_of_pure_spinor_nongeneric}, where $p_j' = p_j$ if $j \in \widehat{J}$ and $p'_j = p_{2g-j}$ if $j \in \widehat{J}^c$. Using the 3 bullet points above we see that each element $\xi_j \frac{p'_j}{2}$ or $2 \xi_{ij} \frac{p'_i}{2} \frac{p'_j}{2}$ appearing in the product acts by a matrix with coefficients in $\Z[\{ u_r, v_s, r_t \}]$, leading to the analogous property for the coefficients of $\omega_L$. 

    We now show that the listed properties hold. Recalling the definition (cf. Equation (\ref{eqn_rappels_l_j})) of the vectors $l_j$, we see that the vectors $l_{g+1}, \dots, l_{2g-m}$ lie in the span $\langle U, \dots, x^{g-m-1} U \rangle$, hence in the span $\langle p_0', \dots, p_{m-1}', p'_{g+1}, \dots, p'_{2g-m} \rangle$. This shows immediately that $\xi_j = 0$ if $j \in \widehat{J}^c$ and that $A_{ij} = 0$ if $i, j \in \widehat{J}^c$ hence $\xi_{ij} = 0$ if $i, j \in \widehat{J}^c$. We see too that $l_{g+1}, \dots, l_{2g-m}$ are $\Z[\{ u_i\}]$-linear combinations of the basis elements $p'_0, \dots, p'_{m-1}, p'_{g+1}, \dots, p'_{2g-m}$, so $A_{ij} \in \Z[ \{ u_i \}]$ if $j \in \widehat{J}^c$. Since $\xi_j = 0$, we have $\xi_{ij} = A_{ij} \in \Z[ \{ u_i \}]$ if $j \in \widehat{J}^c$; since $\xi_{ij}$ is antisymmetric, the result follows. 
\end{proof}
We have now developed enough theory to be able to compute some simple examples. We fix notation. Recall that $f(x) = x^{2g+1} + c_{1} x^{2g} + \dots + c_{2g+1}$. If $(U, V, R)$ represents a divisor class $D$, then we will write
\[ U = x^m + u_1 x^{m-1} + \dots + u_m, \]
\[ V = x^{2g+1-m} + v_1 x^{2g-m} + \dots + v_{2g+1-m}, \]
and
\[ R = r_1 x^{m-1} + \dots + r_m. \]
Suppose first that $g = 1$; we compute $\Psi([P - P_\infty])$ for a point $P = (\alpha, \beta)$ on $C = J$. Factoring $f - \beta^2 = (x-\alpha)V(x)$, we find that we need to compute the pure spinor associated to the isotropic line spanned by $V(x) -\frac{1}{2} V(\alpha)$. In terms of the basis $\mathcal{P} = (p_0, p_1, p_2)$ of $k[x] / (f(x))$, this equals
\[ x^2 + v_1 x + v_2 + \frac{1}{2} \alpha x - \frac{1}{2} \alpha^2 = p_2 - \alpha p_1 - \frac{1}{2} \alpha^2 p_0. \]
Using the formula of Proposition \ref{prop_computation_of_pure_spinor}, we find that the associated pure spinor is $p_2 - \alpha \cdot 1$, representing the point $[1 : -\alpha]$ of $\P(S) = \P^1$. Thus $\Psi$ arises from the morphism $C \to \P^1$ given by $-x$.

Now consider the case $g = m = 2$. In this case, we find that the isotropic subspace $L$ is spanned (in the basis $\mathcal{P}$) by the rows of the matrix
\[ \left(
\begin{array}{ccccc}
 \frac{1}{2} (-u_1 u_2 + u_2 v_1+ v_3) & -\frac{u_1^2}{2} & -u_1 & 1 & 0 \\
 -\frac{u_2^2}{2} & \frac{1}{2} (-u_1 u_2 - u_2 v_1 - v_3) & -u_2 & 0 & 1 \\
\end{array}
\right). \]
The pure spinor $\omega_L$ may be computed in terms of the Pfaffians of the pair
\[ (\xi_i) = (-u_2, -u_1), \,\, (\xi_{ij}) = \left(
\begin{array}{cc}
 0 & \frac{1}{2} (u_2 v_1 + v_3) \\
 -\frac{1}{2} (u_2 v_1 + v_3) & 0 \\
\end{array}
\right). \]
Calculating using Proposition \ref{prop_computation_of_pure_spinor}, we find $\omega_L = [ 1 : -u_1 : u_2 : -u_2 v_1 - v_3 ]$, which is in agreement (up to sign) with the formulae in e.g.\ \cite[Ch. 4, \S 1]{CasselsFlynnProlegomena}.

Now consider the case $g = 2, m = 1$. In this case, we find that the associated isotropic subspace $L$ is spanned by the rows of the matrix
\[ \left(
\begin{array}{ccccc}
 u_1 & 1 & 0 & 0 & 0 \\
 -\frac{1}{2} u_1^4 & 0 & u_1^2 & -u_1 & 1 \\
\end{array}
\right). \]
We find $\omega_L = [ 0 : 1 : -u_1 : u_1^2 ]$, which is again as expected. 

In general, calculating $(\xi_i), (\xi_{ij})$ requires only polynomial division and row reduction for a $g \times (2g+1)$ matrix that is already in echelon form, so can be performed rapidly for a wide range of values of $g$ (cf. the example from the introduction in the case $g=5$).

\subsection{Constructing $\Psi$ universally}

In this section, we show that the construction \eqref{equation_firstversion_keyconstruction} extends to a morphism of varieties.
In fact, we will extend it to the universal Jacobian of the moduli space of (possibly singular) odd hyperelliptic curves.
This will be useful when analyzing local heights, specifically in the proof of Theorem \ref{theorem_weierstrassmodelregular_muv_zero}.

Let $B = \Spec \Z[1/2, c_1, c_2, \dots, c_{2g+1}]$ and consider the universal polynomial $f_{\mathrm{univ}}(x) = x^{2g+1} +c_1x^{2g} + \cdots + c_{2g+1} \in \O_B(B)[x]$.
Let $F_{\mathrm{univ}}(x,z) = z^{2g+2}f(x/z)$ and let $\mathcal{C}\rightarrow \P^1_B$ be the double cover associated to the section $F_{\mathrm{univ}}(x,z) \in \HH^0(\P^1_B, \O_{ \P^1 }(2g+2))$. 
Then the morphism $p \colon \mathcal{C} \rightarrow B$ is flat, projective with geometrically integral fibres of dimension $1$.
Concretely, the restriction of $\mathcal{C}\rightarrow \P^1_B$ to the standard affine open $\A^1_B$ where $z$ is invertible is isomorphic to the closed subscheme $\mathcal{C}^{0}\colon (y^2 = f_{\mathrm{univ}}(x))$ of $\A^2_B$.
The point at infinity with $z=0$ defines a section $P_\infty \in \mathcal{C}(B)$ landing in the smooth locus of $p\colon \mathcal{C}\rightarrow B$.
Let $\mathcal{W}\hookrightarrow \mathcal{C}^{0}$ denote the universal affine Weierstrass locus given by the vanishing of $y$, in other words the complement of $P_{\infty}$ in the ramification locus of $\mathcal{C}\rightarrow \P^1_B$.
Let $B^s\subset B$ be the open subscheme where the discriminant of $f$ is invertible.
Then the restriction $\mathcal{C}|_{B^s}\rightarrow B^s$ is smooth.
If $k$ is a field of characteristic not 2, then specifying a point $b\in B^s(k)$ is the same as specifying a monic degree $2g+1$ polynomial $f(x)\in k[x]$ of nonzero discriminant, and the fibre $\mathcal{C}_b$ is the smooth projective hyperelliptic curve $C = C_f$ over $k$ with affine equation $y^2 = f(x)$.

Let $\Pic_{\mathcal{C}/B}\rightarrow B$ be the Picard scheme of $\mathcal{C}\rightarrow B$ \cite[Section 9.3, Theorem 1]{BLR-neronmodels}, which has the property that for each $B$-scheme $S$, $\Pic_{\mathcal{C}/B}(S)$ is the set of equivalence classes of pairs $(\sh{L}, \alpha)$, where $\sh{L}$ is a line bundle on $\mathcal{C}\times_B S$ and $\alpha$ is a rigidification of $\sh{L}$, i.e., an isomorphism $(P_{\infty,S})^*\sh{L}\simeq \O_S$ of sheaves of $\O_S$-modules.
There is a decomposition into open and closed subschemes $\Pic_{\mathcal{C}/B} = \sqcup_{k\in \Z} \Pic_{\mathcal{C}/B}^k$, where $\Pic_{\mathcal{C}/B}^k$ parametrizes those $(\sh{L}, \alpha)$ such that $\deg(\sh{L}_s) = \chi(\sh{L}_s) - 1+ g=k$ for all geometric points $s$ of $S$.
Write $\mathcal{J} = \Pic^0_{\mathcal{C}/B}\rightarrow B$ for the universal Jacobian.
Note that tensoring with $\O_{\mathcal{C}}(k P_{\infty})$ defines an isomorphism $\mathcal{J}\xrightarrow{\sim} \Pic^k_{\mathcal{C}/B}$.
By \emph{loc. cit.}, the structure morphism $p_{\mathcal{J}} \colon \mathcal{J}\rightarrow B$ is a quasi-projective group scheme, smooth of finite type with geometrically connected fibers.

Write the composition $\mathcal{W}\hookrightarrow \mathcal{C}^{0}\rightarrow B$ as $p_{\mathcal{W}}$.
Since the morphism $p_{\mathcal{W}}$ is affine, we can use the equivalence of categories $\sh{F}\mapsto (p_{\mathcal{W}})_*\sh{F}$ between quasi-coherent $\O_{\mathcal{W}}$-modules and quasi-coherent $(p_{\mathcal{W}})_* \O_{\mathcal{W}}$-modules.
Consider the locally free $\O_B$-algebra $\sh{V}= (p_{\mathcal{W}})_* \O_{\mathcal{W}}= \O_B[x]/(f_{\mathrm{univ}}(x))$.
Let $\psi_{\sh{V}} \colon \sh{V} \times \sh{V}\rightarrow \O_B$ be the $\O_B$-bilinear pairing given by $(a,b)\mapsto \tau(ab)$, where $\tau\colon \sh{V}\rightarrow \O_B$ maps $a_0 + a_1x  +\cdots + a_{2g}x^{2g}$ to $a_{2g}$.
For each $b\in B(k)$ corresponding to a polynomial $f = x^{2g+1} + c_1x^{2g} +\cdots + c_{2g+1} \in k[x]$, the fibre $(\sh{V}_b,\psi_{\sh{V}_b})$ recovers the quadratic space $V = k[x]/(f(x))$ considered before.

Let $\OGr(g, \sh{V})\rightarrow B$ be the orthogonal Grassmannian of $(\sh{V}, \psi_{\sh{V}})$, which has the property that for a $B$-scheme $S$, elements of $\OGr(g, \sh{V})(S)$ are in bijection with isomorphism classes of pairs $(\sh{F}, \alpha)$, where $\sh{F}$ is a locally free $\O_S$-module of rank $g$ and $\alpha\colon \sh{F}\rightarrow \sh{V}\otimes_{\O_B}\O_S$ is an injective homomorphism whose cokernel is locally free and whose image is isotropic with respect to $\psi_{\sh{V}}$.

\begin{proposition}\label{proposition_universal_kummerembedding}
    There exists a (necessarily unique) morphism of $B$-schemes 
    \[
    \Psi^{\mathrm{univ}} \colon \mathcal{J}\rightarrow \OGr(g, \sh{V})
    \]
    with the property that for every field $k$ and $b\in B^s(k)$ (corresponding to a hyperelliptic curve $C/k$), $(\Psi^{\mathrm{univ}}|_{\mathcal{J}_b})(k)\colon \mathcal{J}_b(k) \rightarrow \OGr(g,\sh{V}_b)(k)$ equals (under the identifications $\mathcal{J}_b = J$ and $\sh{V}_b = V$) the map $\Psi(k)$ of \eqref{equation_firstversion_keyconstruction}.
\end{proposition}

Consequently, by pulling back along a $k$-point $b\in B^s(k)$, we see that for every hyperelliptic curve $C = C_f$ of the form studied in \S\ref{subsec_basic_version_keyconstruction}, there exists a morphism $\Psi = \Psi_f\colon J \rightarrow \OGr(g,V)$ whose $k$-points agrees with \eqref{equation_firstversion_keyconstruction}.
The proof of Proposition \ref{proposition_universal_kummerembedding} is a fairly formal adaptation of the recipe \eqref{equation_firstversion_keyconstruction} to the relative setting and is given below.

Let $S$ be a $B$-scheme and consider a pair $(\sh{L}, \alpha)$ corresponding to a morphism $S\rightarrow \Pic^{2g-1}_{\mathcal{C}/B}$. 
For ease of notation, we write the base changes of the morphisms $p\colon \mathcal{C}\rightarrow B$, $p_{\mathcal{J}}\colon \mathcal{J} \rightarrow B$, $j\colon \mathcal{W}\hookrightarrow \mathcal{C}$ and $p_{\mathcal{W}} \colon \mathcal{W} \rightarrow B$ along $S\rightarrow B$ by the same symbols.
Let $\sh{V}_S = (p_{\mathcal{W}})_*(\O_{\mathcal{W}_S}) = \sh{V} \otimes_{\O_B}\O_S$.
Let $V_{\sh{L}} = (p_{\mathcal{W}})_*(j^*\sh{L})$, an invertible $\sh{V}_S$-module.
We equip it with an $\O_S$-bilinear pairing $\psi_{\sh{L}}\colon V_{\sh{L}} \times V_{\sh{L}} \rightarrow \O_S$, as follows.
Let $\iota\colon \mathcal{C}_S\rightarrow \mathcal{C}_S$ be the hyperelliptic involution. 
The rigidification $\alpha\colon (P_{\infty})^* \sh{L}\xrightarrow{\sim} \O_S$ induces rigidifications of $\iota^*\sh{L}$ and $\sh{L}\otimes \iota^*\sh{L}$.
The function $x^{2g-1}$ induces a rigidification of $\O_{\mathcal{C}_S}((4g-2)P_{\infty})$ at $P_{\infty}$.
There exists a unique isomorphism  $\alpha\colon \sh{L}\otimes \iota^*\sh{L}\rightarrow \O_{\mathcal{C}_S}((4g-2)P_{\infty})$ of $\O_{\mathcal{C}_S}$-modules that respects these rigidifications.
The map $(p_{\mathcal{W}})_*j^*\alpha$ is an isomorphism $V_{\sh{L}}\otimes V_{\sh{L}} \rightarrow  \sh{V}_S$ of $\sh{V}_S$-modules.
Postcomposing this map with $\tau\colon \sh{V}_S\rightarrow \O_S$ defines an $\O_S$-bilinear map $\psi_{\sh{L}}\colon V_{\sh{L}} \times V_{\sh{L}} \rightarrow \O_S$.
Let $F_{\sh{L}} = p_* \sh{L}$, an $\O_S$-module.
\begin{lemma}\label{lemma_constructionFLandVL_relativeversion}
    \begin{enumerate}
        \item The $\O_S$-module $F_{\sh{L}}$ is locally free of rank $g$ and the adjunction morphism $F_{\sh{L}} \rightarrow V_{\sh{L}}$ is injective with locally free cokernel. Moreover, the formation of $F_{\sh{L}}$ and $V_{\sh{L}}$ is compatible with arbitrary base change on $S$.
        \item The image of $F_{\sh{L}}\rightarrow V_{\sh{L}}$ is isotropic with respect to $\psi_{\sh{L}}$.
    \end{enumerate}
\end{lemma}
\begin{proof}
\begin{enumerate}
    \item Tensoring the exact sequence of $\O_{\mathcal{C}_S}$-modules
    \[
    0\rightarrow \O_{\mathcal{C}_S}(-\mathcal{W}_S)\rightarrow \O_{\mathcal{C}_S}\rightarrow j_*\O_{\mathcal{W}_S}\rightarrow 0.
    \]
    with $\sh{L}$ and pushing along $p\colon \mathcal{C}_S\rightarrow S$ gives an exact sequence of $\O_S$-modules
    \[
    p_*\sh{L}(-\mathcal{W}_S) \rightarrow F_{\sh{L}} \rightarrow V_{\sh{L}}\rightarrow R^1p_*\sh{L}(-\mathcal{W}_S)\rightarrow R^1p_*\sh{L}.
    \]
    So to prove the first part of the lemma, it suffices to prove that $p_*\sh{L}(-\mathcal{W}_S)=0$, $F_{\sh{L}}$ is locally free of rank $g$ and its formation commutes with base change on $S$, $R^1p_*(\sh{L}(-\mathcal{W}_S))$ is locally free of rank $g+1$ and $R^1p_*\sh{L}=0$.
    To prove these claims, we may assume that we are in the universal case, in other words that $S = \Pic_{\mathcal{C}/B}^{2g-1}$.
    So in particular we may assume $S$ is integral and Noetherian.
    By Grauert's theorem (\cite[Corollary 12.9, Proposition 12.5]{Harthshorne-AG}), it suffices to prove that for every field-valued point $s\colon \Spec(k) \rightarrow S$, $\HH^0(\mathcal{C}_s, \sh{L}(-\mathcal{W}_S)|_{\mathcal{C}_s}) = 0$, $\dim \HH^0(\mathcal{C}_s, \sh{L}|_{\mathcal{C}_s}) = g$, $\dim \HH^1(\mathcal{C}_s, \sh{L}(-\mathcal{W}_S)|_{\mathcal{C}_s}) = g+1$ and $\HH^1(\mathcal{C}_s, \sh{L}|_{\mathcal{C}_s}) = 0$.
    If $\mathcal{C}_s\rightarrow \Spec(k)$ is smooth, this follows from Riemann--Roch. 
    In general, $\mathcal{C}_s$ is an integral Gorenstein curve of arithmetic genus $g$, so we can draw the same conclusions using a version of Riemann--Roch for singular curves \cite[Tag \href{https://stacks.math.columbia.edu/tag/0BS6}{0BS6}]{stacks-project} and the fact that line bundles of negative degree on integral curves have no nonzero global sections. 
    \item We need to prove that the composition $\gamma\colon F_{\sh{L}} \otimes F_{\sh{L}} \rightarrow V_{\sh{L}} \otimes V_{\sh{L}} \xrightarrow{\psi_{\sh{L}}}\O_S$ is zero. 
    Since the formation of $\gamma$ commutes with arbitary base change on $S$, we may again assume that $S = \Pic_{\mathcal{C}/B}^{2g-1}$.
    Since $\Pic_{\mathcal{C}/B}^{2g-1}\rightarrow B$ is flat with geometrically irreducible fibers and $B$ is irreducible, $\Pic_{\mathcal{C}/B}^{2g-1}$ is itself irreducible by \cite[Tag \href{https://stacks.math.columbia.edu/tag/004Z}{004Z}]{stacks-project}. 
    Since the composition $\Pic_{\mathcal{C}/B}^{2g-1}\rightarrow B\rightarrow \Spec \Z[1/2]$ is smooth, $\Pic_{\mathcal{C}/B}^{2g-1}$ is also reduced. 
    Therefore, the map $\gamma$ between locally free sheaves is zero if and only if it is zero at the generic point of $\Pic_{\mathcal{C}/B}^{2g-1}$.
    The last claim is true and follows from Lemma \ref{lem_space_of_sections_is_isotropic}.
\end{enumerate}
\end{proof}

\begin{proof}[Proof of Proposition \ref{proposition_universal_kummerembedding}]
    We have just associated, to each $S$-point $(\sh{L}, \alpha)$ of $\Pic_{\mathcal{C}/B}^{g-1}$, a triple $(F_{\sh{L}}, V_{\sh{L}}, \psi_{\sh{L}})$.
In particular, take an $S$-point $(\sh{N}, \beta)$ of $\mathcal{J} = \Pic_{\mathcal{C}/B}^0$, and let $\sh{L} = \O((2g-1)P_{\infty}) \otimes \sh{N}^{\otimes 2}$ with rigidification $\alpha$ induced by $\beta$. 
There is a unique isomorphism $\sh{N} \otimes \iota^*\sh{N}\xrightarrow{\sim} \O_{\mathcal{C}_S}$ respecting rigidifications, which restricts to an isomorphism $V_{\sh{L}} \xrightarrow{\sim} \sh{V}_S$ of $\sh{V}_S$-modules intertwining the pairings $\psi_{\sh{L}}$ and $\psi_{\sh{V}_S}$.
Let $\sh{F}\subset \sh{V}_S$ be the image of $F_{\sh{L}}$ under this isomorphism. 
Then we have just associated to every $S$-point of $\mathcal{J}$ a locally free sheaf $\sh{F}$ of rank $g$ and an isotropic embedding $\sh{F} \leq \sh{V}_S = \sh{V} \otimes_{\O_B} \O_S$ with locally free cokernel.
Moreover this association is compatible with base change on $S$.
By the Yoneda lemma, this defines a morphism $\Psi^{\mathrm{univ}} \colon \mathcal{J}\rightarrow \OGr(g, \sh{V})$ of $B$-schemes recovering the explicit description before on $k$-points.
\end{proof}

Let $V_0$ be the free $\Z[1/2]$-module with basis $e_0, \dots, e_{2g}$.
Let $b_0$ be the symmetric bilinear form on $V_0$ satisfying $b_0(e_i, e_{2g-j}) = \delta_{ij}$ for all $0\leq i,j\leq g$.
Using the formulae \eqref{eq_straightened_basis}, we can uniquely modify the $\O_B$-basis $1, x, \dots, x^{2g}$ of $\sh{V}$ to a basis $p_0, \dots, p_{2g}$ satisfying $\psi_{\sh{V}}(p_i,p_{j}) = \delta_{i,2g-j}$ for all $0\leq i,j\leq 2g$.
The basis $\{p_i\}$ defines an isomorphism of quadratic spaces $(\sh{V}, \psi_{\sh{V}})\simeq (V_0,b_0)\otimes_{\Z[1/2]} \O_B$.
This determines an isomorphism of $B$-schemes $\OGr(g,\sh{V}) \simeq \OGr(g,V_0) \times_{\Z[1/2]} B$.
Since the spinor embedding of \S\ref{subsec_orthogonal_grassmannian} extends to a morphism of $\Z[1/2]$-schemes $\OGr(g,V_0) \hookrightarrow \P^{2^g-1}_{\Z[1/2]}$, we thus obtain a morphism
\begin{align}\label{eq_universal_kummer_embedding}
    \Sigma\circ \Psi^{\mathrm{univ}}\colon \mathcal{J} \rightarrow \OGr(g, V_0)\times_{\Z[1/2]} B \hookrightarrow \P^{2^g-1}_B.
\end{align}

\subsection{Connecting $\Psi$ to the $|2\Theta|$-linear system}

Continue with the notation of \S\ref{subsec_basic_version_keyconstruction}. 
We have thus constructed a morphism $\Psi : J \to \OGr(g, V)$. We next connect the composite  $\Sigma \circ \Psi\colon J\rightarrow \OGr(g,V)\hookrightarrow \P^{2^g-1}$ to the $|2\Theta|$-linear system.
Denoting the co-ordinates of $\P^{2^g-1}$ by $[x_1 : x_2 : \cdots : x_{2^g}]$, let $H\subset \P^{2^g-1}$ be the hyperplane $(x_1 = 0)$.

\begin{proposition}\label{prop_firstpropertieskummerembedding}
The image of $\Psi$ is not contained in $H$, and there is an equality of Cartier divisors $\Psi^{*}H =2\Theta$. 
    Therefore, there exists a unique isomorphism $\Psi^*\O_{\OGr(g, V)}(1) \cong \O_J(2\Theta)$ such that $\Psi^*(x_1) = 1 \in \HH^0(J, \O_J(2\Theta))$.
\end{proposition}
\begin{proof}
    We may assume $k$ is algebraically closed.
    Proposition \ref{prop_properties_of_explicit_Kummer_morphism} (specifically, Part 1) shows that if $P\in J(k)$ has Mumford degree $m$, then $\Psi(P) \in H$ if and only if $0\leq m \leq g-1$, in other words if and only if $P$ lies on the theta divisor $\Theta$.
    Therefore $\Psi^*H = n \Theta$ for some integer $n\geq 1$.
    We will show that $n = 2$ using a Grothendieck--Riemann--Roch calculation.

    Recall that we may view points on $J = \Pic^0_C$ as isomorphism classes of degree zero line bundles on $C$ rigidified along $P_{\infty}$. 
    Let $\sh{N}$ be the universal such invertible sheaf on $C\times J$, equipped with the data of a rigidification $P_{\infty}^* \sh{N} \simeq \O_J$, and let $\sh{L} = \O_C((2g-1)P_{\infty})$.
    Consider the projections $\pi_C\colon C\times J \rightarrow C$ and $\pi_J\colon C\times J\rightarrow J$ and write $\sh{M} = \pi_C^*\sh{L} \otimes \sh{N}^{\otimes 2}$, a line bundle on $C\times J$.
    Let $\sh{F} = (\pi_{J})_*\sh{M}$.
    By Proposition \ref{proposition_universal_kummerembedding} and its proof, $\sh{F}$ is a locally free sheaf of rank $g$ on $J$ and comes equipped with an embedding $\sh{F}\rightarrow V \otimes \O_J$.

    Let $i\colon C\hookrightarrow J$ be the Abel--Jacobi embedding sending $x$ to $\O_C(x)\otimes \O_C(P_{\infty})^{-1}$.
    Denote the two projections $C\times C\rightarrow C$ by $\pi_1$ and $\pi_2$.
    Let $\sh{N}_0 = (1,i)^*\sh{N}$ and $\sh{M}_0 = (1,i)^*\sh{M} = \pi_1^*\sh{L} \otimes \sh{N}_0^{\otimes 2}$, both line bundles on $C\times C$.
    Let $\sh{F}_0 = i^*\sh{F}$, which by proper base change is isomorphic to $(\pi_2)_*\sh{M}_0$.
    We apply Grothendieck--Riemann--Roch to the morphism $\pi_2\colon C\times C\rightarrow C$ and line bundle $\sh{M}_0$ on $C\times C$.
    This gives the following identity in $\mathrm{CH}^*(C)\otimes \Q$, the Chow ring of $C$ modulo rational equivalence and with $\Q$-coefficients:
    \begin{align*}
        \mathrm{ch}(f_!\sh{M}_0) = f_*(\mathrm{ch}(\sh{M}_0) \cdot \pi_1^* \mathrm{td}(T_C)).
    \end{align*}
    We refer to \cite[Theorem A.5.3]{Harthshorne-AG}, which also explains the notation.
    Given a line bundle $\sh{K}$ or divisor $D$, write $[\sh{K}]$ or $[D]$ for the associated divisor class. Write $K_C$ for the canonical divisor class of $C$.
    Using the fact that $R^i \pi_{2,*}\sh{M}_0=0$ if $i\geq 1$ and the identities $\mathrm{ch}(\sh{F}_0) = \mathrm{rk}(\sh{F}) + c_1(\sh{F}) +\cdots$, $\mathrm{td}(T_C) = 1- K_C/2$ and $\mathrm{ch}(\sh{M}_0) = 1+ [\sh{M}_0] + [\sh{M}_0]^2/2 + \cdots$, we get the following equality in $\mathrm{CH}^1(C)\otimes \Q \simeq \Pic(C)\otimes \Q$:
    \begin{align*}
        [\det\sh{F}_0] = c_1(\sh{F}_0) = 
        \frac{1}{2}(\pi_2)_*([\sh{M}_0]\cdot [\sh{M}_0] - \pi_1^*K_C \cdot [\sh{M}_0]).
    \end{align*}
    We calculate the right hand side. 
    We have $\pi_1^*K_C= (2g-2) [\{P_{\infty}\} \times C]$
    and $[\sh{M}_0 ]  = (2g-1)[\{P_{\infty}\}\times C] + 2[\sh{N}_0]$.
    Let $\Delta_C\subset C\times C$ be the diagonal.
    The line bundle $\sh{N}_0$ on $C\times C$ has the property that $\sh{N}_0|_{C\times \{x\}}\simeq \O_C(x-P_{\infty})$ and $\sh{N}_0|_{\{P_{\infty}\}\times C}\simeq \O_C$. 
    Since a line bundle on $C\times C$ is uniquely determined with these properties and the line bundle $\O_{C\times C}(\Delta_C - \{P_{\infty}\}\times C- C\times\{P_{\infty}\})$ also has this property, we have 
    \[
    [\sh{N}_0] = [\Delta_C] - [\{P_{\infty}\}\times C]- [C\times\{P_{\infty}\}].
    \]
    We now have expressed all relevant divisors in terms of $\ell_1 = [\Delta_C]$, $\ell_2 = [\{P_{\infty}\}\times C]$ and $\ell_3 = [C\times \{P_{\infty}\}]$.
    By adjunction, we have $\ell_1\cdot \ell_1 = -(\Delta_C)_*(K_C) = -(2g-2)[P_{\infty}\times P_{\infty}]$.
    The other intersection products are $\ell_1\cdot \ell_2 = \ell_1\cdot \ell_3  =\ell_2\cdot \ell_3 = [P_{\infty}\times P_{\infty}]$ and $\ell_2 \cdot \ell_2 = \ell_3\cdot \ell_3 = 0$.
    We then calculate that 
    \[
    [\sh{M}_0]\cdot ([\sh{M}_0] - \pi_1^*K_C)=
    (2\ell_1 + (2g-3)\ell_2 - 2\ell_3)\cdot
    (2\ell_1 -\ell_2 -2\ell_3)=
    -8g[P_{\infty}\times P_{\infty}].
    \]
    Combining the last two centered equalities shows that $[\det \sh{F}_0] = [\O_C(-4gP_{\infty})]$ in $\mathrm{CH}^1(C)\otimes \Q$, so $\det  \sh{F}_0$ and $\O_C(-4gP_{\infty})$ determine the same class in $\Pic(C) \otimes \Q$.
    It follows from \cite[Section 11.12, Exercise 10]{BirkenhakeLangeAV} that $i^*\O_J(\Theta)\simeq \O_C(gP_{\infty})$.
    By Lemma \ref{lemma_spinorembedding_squareroot_plucker}, we have an isomorphism $(\det\sh{F})^{\vee}\simeq \Psi^*\O_{\P^{2^g-1}}(2)$. 
    Therefore $\det \sh{F} \simeq \O_J(-2n\Theta)$, so $\det\sh{F}_0 \simeq \O(-2ng P_{\infty})$. 
    It follows that $\O_C(-4gP_{\infty})\simeq \det \sh{F}_0 \simeq \O_C(-2ng P_{\infty})$.
    We conclude that $n=2$, as desired.
\end{proof} 
Proposition \ref{prop_firstpropertieskummerembedding} shows that pullback induces a map 
\begin{equation}\label{eqn_Psi_pullback} \Psi^\ast : \HH^0(\OGr(g, V), \cO_{\OGr(g, V)}(1)) \to \HH^0(J, \cO_J(2 \Theta)).
\end{equation}
We want to show that this map is an isomorphism. We will do this by building on the well-known identification between $J[2]$ and a subgroup of $\SO(V)$:
\begin{proposition}\label{prop_identification_of_stabiliser_and_2_torsion}
    Let $T_f : V \to V$ denote the $k$-linear map given by multiplication by $x \in V = k[x] / (f(x))$. Then the map $\sh{N} \mapsto \varphi^{\sh{N}}_{\sh{L}, \sh{L}}$ (in the notation of \S\ref{subsec_basic_version_keyconstruction}) defines an isomorphism between the following finite $k$-groups:
    \begin{enumerate}
      \item The 2-torsion subgroup $J[2] \leq J$.
        \item The centralizer $\operatorname{Cent}_{\SO(V)}(T_f)$.
    \end{enumerate}
\end{proposition}
\begin{proof}
The existence of \emph{an} isomorphism between these two groups is contained in \cite[Proposition 11]{BhargavaGross}. We recall the definition of the isomorphism given in \emph{loc. cit.}, and then show that it is compatible with the one given in the statement of Proposition \ref{prop_identification_of_stabiliser_and_2_torsion}. 

Let $L / k$ be the splitting field of $f(x)$, and let $\omega_1, \dots, \omega_{2g+1} \in L$ denote the roots of $f(x)$. The vector $1 \in V_L = L[x] / (f(x))$ decomposes as a sum $1 = \sum_{i=1}^{2g+1} e_i$ of primitive idempotents, which are also eigenvectors for $T_f$, satisfying $T_f e_i = \omega_i e_i$. The centralizer $\operatorname{Cent}_{\OO(V)}(T_f)$ is naturally identified the set of subsets $I \subset \{ 1, \dots, 2g+1\}$, via the map $I \mapsto \sum_{i \in I^c} e_i - \sum_{i \in I} e_i$. The morphism $\operatorname{Cent}_{\OO(V)}(T_f) \to J[2]$ sends $I$ to the  class of the divisor $\sum_{i \in I} \left( (\omega_i, 0) - P_\infty \right)$. This becomes an isomorphism after either restriction to $\operatorname{Cent}_{\SO(V)}(T_f)$, or passage to the quotient $\operatorname{Cent}_{\mathrm{PO}(V)}(T_f)$ (using the relation $[\sum_{i=1}^{2g+1} (\omega_i, 0) - P_\infty] = 0$), since the composite
\[ \SO(V) \to \OO(V) \to \mathrm{PO}(V) = \OO(V) / \mu_2 \]
is an isomorphism. 

    To prove the proposition as stated, we need to show that if $\sh{N} = \cO_C(\sum_{i \in I} ((\omega_i, 0) - P_\infty)) \in J[2]$, then the image of $\varphi^{\sh{N}}_{\sh{L}, \sh{L}}$ in $\mathrm{PO}(V)$ equals the image of $\sum_{i \in I^c} e_i - \sum_{i \in I} e_i$.
    Choose isomorphisms
    $\alpha\colon \O_C\rightarrow \sh{N}^{\otimes2}$, $\beta\colon \sh{N} \rightarrow \iota^*\sh{N}$ and $\gamma\colon \sh{N} \otimes \iota^*\sh{N}\rightarrow \O_C$.
    Let $\mathrm{can}$ be the canonical isomorphism $\sh{N}|_W \rightarrow (\iota^*\sh{N})|_W$.
    Then $\varphi_{\sh{L}, \sh{L}}^{\sh{N}}$ is the composite
    \[
    \sh{L}|_W 
    \xrightarrow{(\Id_{\sh{L}}\otimes \alpha)|_W} (\sh{L}\otimes \sh{N}^{\otimes 2})|_W
    \xrightarrow{\Id_{(\sh{L}\otimes \sh{N})|_W} \otimes \mathrm{can}}
    (\sh{L}\otimes \sh{N}\otimes \iota^*\sh{N})|_W
    \xrightarrow{(\Id_{\sh{L}}\otimes \gamma)|_W}
    \sh{L}|_W.
    \]
    The composition $\gamma\circ (\Id_{\sh{N}}\otimes \beta) \circ \alpha$ is an isomorphism $\O_C\rightarrow \O_C$, so equals multiplication by an element of $k^{\times}$.
    By combining the last two sentences, it is enough to check that $\mathrm{can}^{-1} \circ \beta|_W \colon \sh{N}|_W\rightarrow \sh{N}|_W$ acts as $+1$ over the points $(\omega_i,0)$ $(i\not\in I)$ and as $-1$ over the points $(\omega_i,0)$ $(i\in I)$.
    But if $Q = (\omega_i, 0)$, if $t$ is a generator of the stalk $\sh{N}_Q$ (which is a module under the local ring $\O_{C, Q}$ with maximal ideal $m_Q$) and if $\iota^*(t)/t \mod m_Q = \lambda \in k^{\times}$, then $\mathrm{can}^{-1} \circ \beta|_W$ acts as $\lambda$ over the point $(\omega_i, 0)$.
    If $i\not \in  I$, then we can take $t = 1$ and so $\lambda = 1$.
    If $i\in I$, we can take $t = 1/y$ and so $\iota^*(t)/t = -1 = \lambda$, as required.
\end{proof}
We now lift the embedding $J[2] \hookrightarrow \SO(V)$ to an embedding of the Mumford theta group $\mathcal{G}(2 \Theta)$ into $S \Gamma$. Recall (\cite[\S 1]{Mum66}) that the group $\mathcal{G}(2 \Theta)$ may be defined as the set of pairs $(x, \phi)$, where $x \in J$ and $\phi : \cO_J(2 \Theta) \to t_x^\ast \cO_J(2 \Theta)$ is an isomorphism of invertible sheaves. Then $\mathcal{G}(2 \Theta)$ is a linear algebraic group over $k$, extension of $J[2]$ by $\G_m$, which acts linearly on $\HH^0(J, \cO_J(2 \Theta))$. It is known (\cite[\S 1, Theorem 2]{Mum66}) that this representation is irreducible. 
\begin{proposition}\label{prop_isomorphism_of_theta_groups}
        The isomorphism of Proposition \ref{prop_identification_of_stabiliser_and_2_torsion} lifts to an isomorphism of short exact sequences
        \[ \xymatrix{ 1 \ar[r] & \G_m \ar[d]\ar[r] & \Stab_{S \Gamma}(T_f) \ar[r] \ar[d]& \Stab_{\SO(V)}(T_f)  \ar[d]\ar[r] & 1 \\ 1 \ar[r] &\G_m \ar[r] &  \mathcal{G}(2 \Theta) \ar[r] &   J[2] \ar[r] & 1. }
        \]
        Moreover, the map \[ \Psi^\ast : \HH^0(\OGr(g, V), \O_{\OGr(g, V)}(1)) \rightarrow \HH^0(J, \O_J(2\Theta)) \]
    is equivariant with respect to the action of $\Stab_{S \Gamma}(T_f)$, when it acts via $S \Gamma$ on the source and via the isomorphism $\Stab_{S \Gamma}(T_f) \cong \mathcal{G}(2 \Theta)$ on the target. 
\end{proposition}
\begin{proof}
    We construct the map $\Stab_{S\Gamma}(T_f) \to \mathcal{G}(2\Theta)$. Since $\O_{\OGr(g, V)}(1)$ is an $S \Gamma$-equivariant line bundle on $\OGr(g, V)$, any element $x \in S \Gamma$ determines an isomorphism $[x] : \O_{\OGr(g, V)}(1) \to x^\ast \OGr(g, V)(1)$. On the other hand, if $\sh{N} \in J[2]$ corresponds to $\eta \in \Stab_{\SO(V)}(T_f)$, then we have the formula $\Psi \circ t_{\sh{N}} = \eta \cdot \Psi$. Indeed, the left-hand side, by definition, sends $\sh{M} \in J$ to $\varphi^{\sh{M} \otimes \sh{N}}_{\sh{L} \otimes \sh{M}^{\otimes 2} \otimes \sh{N}^{\otimes 2}, \sh{L}}(F_{\sh{L} \otimes \sh{M}^{\otimes 2} \otimes \sh{N}^{\otimes 2}}) = \varphi^{\sh{M} \otimes \sh{N}}_{\sh{L} \otimes \sh{M}^{\otimes 2}, \sh{L}}(F_{\sh{L} \otimes \sh{M}^{\otimes 2}})$, while the right-hand side sends $\sh{M}$ to $\varphi^{\sh{N}}_{\sh{L}, \sh{L}} \circ \varphi^{\sh{M}}_{\sh{L} \otimes \sh{M}^{\otimes 2}, \sh{L}}(F_{\sh{L} \otimes \sh{M}^{\otimes 2}})$, so the formula therefore follows from the transitivity relation (valid for any $\sh{A}, \sh{B} \in J$):
    \[  \varphi^{\sh{A}}_{\sh{L} \otimes \sh{A}^{\otimes 2}, \sh{L}} \circ \varphi^{\sh{B}}_{\sh{L} \otimes \sh{A}^{\otimes 2} \otimes\sh{B}^{\otimes 2}, \sh{L}\otimes \sh{A}^{\otimes 2} } = \varphi^{\sh{B} \otimes \sh{A}}_{\sh{L} \otimes \sh{A}^{\otimes 2} \otimes \sh{B}^{\otimes 2}, \sh{L}}. \]
    If $x \in \Stab_{S \Gamma}(T_f)$ lifts $\eta$, then we are given canonical isomorphisms $t_{\sh{N}}^\ast \cO_J(2 \Theta) \cong \Psi^\ast(\eta^\ast \cO_{\OGr(g, V)}(1))$ and $[x]^\ast : \cO_{\OGr(g, V)}(1) \to \eta^\ast \cO_{\OGr(g, V)}(1)$, hence $\Psi^\ast([x]^\ast) : \O_J(2\Theta) \to \Psi^\ast(\eta^\ast \cO_{\OGr(g, V)}(1))$. Putting these together, we obtain an isomorphism $\cO_J(2 \Theta) \to t_{\sh{N}}^\ast \cO_J(2 \Theta)$ depending only on $x$. This defines the homomorphism $\Stab_{S \Gamma}(T_f) \to \mathcal{G}(2 \Theta)$. By construction, it fits into a commutative diagram as in the statement of the proposition, with exact rows; it must therefore be an isomorphism. The claimed equivariance is immediate from the construction. 
\end{proof}
The proof of Proposition \ref{prop_isomorphism_of_theta_groups} specifies an isomorphism $\mathcal{G}(2\Theta)\xrightarrow{\sim} \Stab_{S\Gamma}(T_f)$; in the remainder of this paper we use this choice to view $\mathcal{G}(2\Theta)$ as a subgroup of $S\Gamma$.
Finally, we can show that the map of Equation (\ref{eqn_Psi_pullback}) is indeed an isomorphism: 
\begin{proposition}
    \begin{enumerate} 
    \item The isomorphism $\Psi^*\O_{\OGr(g, V)}(1)\cong \O_J(2\Theta)$ of Proposition \ref{prop_firstpropertieskummerembedding} identifies $\Sigma \circ \Psi$ with the morphism associated to the complete linear system of $2\Theta$.
    \item $\Psi$ factors through a closed embedding of the Kummer variety $K = J/\langle \pm 1\rangle \to \OGr(g,V)$.
\end{enumerate}
\end{proposition}
\begin{proof}
    To prove Part 1, we need to prove that the restriction map 
    \[ \Psi^\ast \circ \Sigma^\ast : \HH^0(\P^{2^g-1}, \O_{\P^{2^g-1}}(1)) \to  \HH^0(\OGr(g, V), \O_{\OGr(g, V)}(1)) \rightarrow \HH^0(J, \O_J(2\Theta)) \]
   is an isomorphism. The map $\Sigma^\ast$ is an isomorphism. Proposition \ref{prop_isomorphism_of_theta_groups} shows that the image of $\Psi^\ast \circ \Sigma^\ast$ is a $\mathcal{G}(2 \Theta)$-invariant subspace of $\HH^0(J, \O_J(2\Theta))$. The image is nonzero, by Proposition \ref{prop_firstpropertieskummerembedding}. On the other hand, $\HH^0(J, \O_J(2\Theta))$ is an irreducible representation of $\mathcal{G}(2 \Theta)$. So $\Psi^\ast \circ \Sigma^\ast$ must be surjective. Since the source and target have the same dimension, namely $2^g$, it must be an isomorphism. 

    Part 2 follows from the fact that, for any abelian variety $A$ over a field of characteristic not 2, and any symmetric divisor $D \leq A$ defining a principal polarisation, the induced morphism $A \to \P(H^0(A, \cO_A(2D)^\vee)$ factors through an embedding of the Kummer variety $A / \{ \pm 1 \}$. (See e.g.\  \cite[Theorem 4.8.1]{BirkenhakeLangeAV}, whose proof is written for $k = \mathbb{C}$ but remains valid in general.)
\end{proof}

\subsection{The image of the Kummer embedding}\label{subsec_image_of_Kummer_embedding}

Our next goal is to determine the images of $\Psi$ and $\Sigma\circ \Psi$ explicitly.
To this end, we define a related morphism.
Let $\bar{W} = W \sqcup \{P_{\infty}\}$ be the Weierstrass locus of the hyperelliptic curve $C\rightarrow \P^1$, which can be considered as a closed subscheme of either $C$ or $\P^1$.
There is an isomorphism of $k$-algebras $\HH^0(\bar{W}, \O_{\bar{W}}) \simeq \HH^0(W, \O_W) \times k =  (k[x]/f(x)) \times k$.
Let $\sh{L}$ be an invertible sheaf on $C$ of degree $2g-1$ equipped with a rigidification $\sh{L}|_{P_{\infty}}\simeq k$.
Let $\tilde{V}_{\sh{L}} = \HH^0(\bar{W},\sh{L}|_{\bar{W}})$.
The rigidification at $P_{\infty}$ induces an isomorphism $\tilde{V}_{\sh{L}}  = V_{\sh{L}} \oplus k$ and a symmetric bilinear form $\psi_{\sh{L}}$ on $V_{\sh{L}}$ studied in \S\ref{subsec_basic_version_keyconstruction}.
Define symmetric bilinear forms on $\tilde{V}_{\sh{L}}$ by the formulae 
\[
b_1((v,z),(v',z')) = \psi_{\sh{L}}(v,v'), \quad b_2((v,z), (v',z')) = \psi_{\sh{L}}(v,xv') - zz',
\]
where $v,v'\in V_{\sh{L}}$, $z,z'\in k$.
Since multiplication by $x$ on $V_{\sh{L}}$ is self-adjoint with respect to $\psi_{\sh{L}}$, $b_2$ is indeed symmetric. 
Let $\tilde{F}_{\sh{L}}$ be the image of the restriction map $\HH^0(C, \sh{L}) \rightarrow \tilde{V}_{\sh{L}}$.

\begin{lemma}
    The subspace $\tilde{F}_{\sh{L}}\subset \tilde{V}_{\sh{L}}$ is $g$-dimensional and isotropic with respect to both the forms $b_1$ and $b_2$.
\end{lemma}
\begin{proof}
    Since $\HH^0(C, \sh{L}(-\bar{W}))=0$, the map $\HH^0(C, \sh{L}) \rightarrow \tilde{V}_{\sh{L}}$ is injective and so by Riemann--Roch we have $\dim \tilde{F}_{\sh{L}} = g$.
    To show that $\tilde{F}_{\sh{L}}$ is isotropic with respect to $b_1$ and $b_2$, we give an alternative interpretation of these forms.

    Let $I = \HH^0(\bar{W}, \O_{\P^1}((2g-1){P_\infty})|_{\bar{W}})$.
    As a module under $\HH^0(\bar{W}, \O_{\bar{W}})  = (k[x]/f(x)) \times k$, it decomposes as $I_W \oplus I_{\infty}$, where $I_W = \HH^0(W, \O_{\P^1}((2g-1)\infty)|_W))$ and $I_{\infty} = \O((2g-1){P_\infty})|_{P_\infty}$.
    Let $e_W\in I_W$ be the image of $1$ under the map $\HH^0(\P^1,\O_{\P^1}((2g-1){P_\infty}))\rightarrow I_W$, and $e_{\infty}$ the image of $x^{2g-1}$ under the map $\HH^0(\P^1,\O_{\P^1}((2g-1){P_\infty}))\rightarrow I_{\infty}$.
    Then $I$ is a free $\HH^0(\bar{W}, \O_{\bar{W}})$-module of rank $1$ with generator $(e_W, e_{\infty})$.
    Consequently, an element of $I$ is a pair $(ve_W, \lambda e_{\infty})$, where $v = \sum_{i=0}^{2g} a_i x^{i}\in k[x]/(f(x))$ and $\lambda\in k$.
    The assignment $(ve_W, \lambda e_{\infty}) \mapsto \tau(v)$ defines a linear map $\tau_1\colon I\rightarrow k$.
    The assignment $(ve_W, \lambda e_{\infty}) \mapsto \tau(xv) - \lambda$ defines another linear map $\tau_2\colon I\rightarrow k$.

    Consider the unique isomorphism $\sh{L} \otimes \iota^*\sh{L} \xrightarrow{\sim} \O_C((4g-2)P_{\infty})$ that respects the rigidifications on both sides. (The left hand side has a rigidification induced by the one on $\sh{L}$, and the right hand side has one induced by $x^{2g-1}$.)
    After restricting to $\bar{W}$ and passing to global sections, this induces a map  $\alpha\colon \tilde{V}_{\sh{L}} \times \tilde{V}_{\sh{L}} \rightarrow \HH^0(\bar{W}, \O_{C}((4g-2)P_{\infty})|_{\bar{W}}) = I$ that is bilinear with respect to the $\HH^0(\bar{W}, \O_{\bar{W}})$-action.
    A calculation shows that $b_i = \tau_i \circ \alpha$ for each $i=1,2$.

    We may now proceed in a way similar to the proof of Lemma \ref{lem_space_of_sections_is_isotropic}.
    For each $i=1,2$, the image of the composition $\tilde{F}_{\sh{L}} \times \tilde{F}_{\sh{L}} \rightarrow \tilde{V}_{\sh{L}} \times \tilde{V}_{\sh{L}} \xrightarrow{b_i} k$ equals the image of the composition 
    \[
    \HH^0(C, \sh{L}) \times \HH^0(C, \iota^*\sh{L}) \rightarrow
    \HH^0(C, \O_C((4g-2)P_{\infty})) \xrightarrow{\rho} I \xrightarrow{\tau_i} k,
    \]
    where the first map is induced by the isomorphism $\sh{L} \otimes \iota^*\sh{L} \xrightarrow{\sim} \O_C((4g-2)P_{\infty})$ and $\rho$ is given by restriction along $\bar{W}$.
    The vector space $\HH^0(C, \O_C((4g-2)P_{\infty}))$ has basis $1, x, \dots, x^{2g-1}, y, \dots, x^{g-2}y$, and we have $\rho(x^i y) = 0$ if $0\leq i\leq g-2$, $\rho(x^i) = (x^ie_W, 0)$ if $0\leq i \leq 2g-2$ and $\rho(x^{2g-1}) = (x^{2g-1}e_W,e_{\infty})$.
    Each of these elements lies in the kernel of $\tau_i$,  so we may conclude that  $\tilde{F}_{\sh{L}}$ is indeed isotropic with respect to $b_i$.
\end{proof}

Now suppose $\sh{L} = \O_{C}((2g-1)P_{\infty})$, in which case we write $\tilde{V} = \tilde{V}_{\sh{L}}$.
Concretely, $\tilde{V} = V \oplus k$ and the bilinear forms $b_1, b_2$ are given by $b_1((v,z), (v',z')) = \tau(vv')$ and $b_2((v,z), (v',z')) = \tau(xvv') -zz'$.
Let $\sh{N}$ be an invertible sheaf on $C$ of degree $0$ and let $\sh{M} = \sh{L} \otimes \sh{N}^{\otimes 2}$.
A choice of isomorphism $\sh{N} \otimes \iota^*\sh{N} \xrightarrow{\sim} \O_C$ induces an isomorphism $\tilde{\varphi}_{\sh{M}, \sh{L}}^{\sh{N}}\colon \tilde{V}_{\sh{M}} \rightarrow \tilde{V}$ and the image $\tilde{\varphi}_{\sh{M}, \sh{L}}^{\sh{N}}(\tilde{F}_{\sh{M}}) \subset \tilde{V}$ is isotropic with respect to $b_1$ and $b_2$.
This procedure can be carried out in families, so there exists a morphism 
\[
\tilde{\Psi}\colon J\rightarrow T,
\]
where $T = \OGr(g,\tilde{V},b_1) \cap \OGr(g, \tilde{V},b_2)$ is the Fano variety of $g$-dimensional linear subspaces of $\tilde{V}$ which are isotropic for both $b_1$ and $b_2$, and with the property that $\tilde{\Psi}([\sh{N}]) = \tilde{\varphi}_{\sh{M}, \sh{L}}^{\sh{N}}(\tilde{F}_{\sh{M}})$ for every line bundle of degree zero $\sh{N}$ on $C_K$ and every field extension $K/k$.

Given a subspace $\tilde{L}\in T(k)$, the condition that $\tilde{L}$ is isotropic for $b_2$ implies that $0 \oplus k\not\subset \tilde{L}$, so the projection of $\tilde{L}$ under $\tilde{V} = V \oplus k \rightarrow V$ is again of dimension $g$ and isotropic with respect to $\psi$.
This defines a morphism $\pi\colon T\rightarrow \OGr(g,V)$.
These morphisms fits into a commutative diagram 
\begin{equation} \label{diagram_psiandpsitilde}
\begin{tikzcd}
	J & T \\
	K & {\OGr(g,V)}.
	\arrow["{\tilde{\Psi}}", from=1-1, to=1-2]
	\arrow[from=1-1, to=2-1]
	\arrow["\pi", from=1-2, to=2-2]
	\arrow["{\Psi}", from=2-1, to=2-2]
\end{tikzcd}
\end{equation}
We now define an important closed subscheme $X\subset \OGr(g,V)$ that will turn out to be the image of $\Psi$.
Let $\sh{F} \subset \underline{V} = V \otimes \cO_{\OGr(g, V)}$ be the locally free sheaf corresponding to the tautological subbundle on $\OGr(g, V)$.
If $U\subset \OGr(g,V)$ is an affine open subset and $\phi\colon \sh{F}|_U \rightarrow \O_U^{\oplus g}$ an isomorphism, then the restriction of the form $(a,b)\mapsto \psi(xab)$ to $\sh{F}|_U$ has a Gram matrix $A_U$, which is a symmetric $g\times g$ matrix whose coefficients are elements of $\HH^0(U, \O_U)$.
There exists a unique closed subscheme $X\hookrightarrow \OGr(g,V)$ with the property that for every pair $(U, \phi)$, $X|_U\hookrightarrow U$ is the closed subscheme cut out by the determinants of all the $2\times 2$ minors of $A_U$. (This is an example of a symmetric determinantal variety.)
For every field extension $K/k$, $X(K)$ equals the set of $[L] \in \OGr(g,V)(K)$ such that the restriction of the form $(a, b)\mapsto \psi(xab)$ to $L\times L$ has rank $\leq 1$.

Let $\iota\colon T\rightarrow T$ be the involution that sends an isotropic subspace $L\subset \tilde{V}$ to its image under the map $\tilde{V} \rightarrow \tilde{V} = V \oplus k, (v,z)\mapsto (v,-z)$.
\begin{lemma}\label{lem_image_T_in_OGr}
    The scheme theoretic image of $\pi\colon T\rightarrow \OGr(g,V)$ is $X$.
    The induced morphism $\pi_1\colon T\rightarrow X$ identifies $X$ with the (scheme-theoretic) quotient of $T$ by $\iota$.
    Consequently, $T\rightarrow X$ is finite.
    It is \'etale away from the closed subscheme $X_0$ of those $L\in \OGr(g,V)$ for which $(a,b)\mapsto \psi(a,xb)$ is identically zero.
 \end{lemma}
 \begin{proof}
    First, we prove a general fact.
    Let $Y$ be a $k$-scheme. If $\sh{V}$ is a locally free sheaf on $Y$ of finite rank, let $\mathbb{V}(\sh{V})\rightarrow Y$ denote the associated vector bundle.
    Given such a $\sh{V}$ on $Y$, let $\mathcal{R}_1\subset\mathbb{V}(\Sym^2\sh{V}) \rightarrow Y$ be the closed subscheme of tensors of rank $\leq 1$, which is locally cut out by the determinants of all $2\times 2$ minors in a trivializing open cover of $\sh{V}$.
    Let $\nu\colon \mathbb{V}(\sh{V})\rightarrow \mathbb{V}(\Sym^2 \sh{V})$ be the map of $Y$-schemes sending a local section $v$ to $v^2$.
    Then we claim that the scheme-theoretic image of $\nu$ is $\mathcal{R}_1$ and the induced map $\nu_1\colon \sh{V}\rightarrow \mathcal{R}_1$ is the scheme-theoretic quotient of $\sh{V}$ by the involution $v\mapsto -v$.
    Indeed, by restricting to a trivializing open cover for $\sh{V}$, we may assume $\sh{V}$ is trivial, in which case the situation is pulled back from $Y\rightarrow \Spec(k)$.
    So we reduce to the case $Y = \Spec(k)$ and $\sh{V}$ is a finite dimensional $k$-vector space.
    In this case, the claim is equivalent to the statement that the map of $k$-algebras $k[\Sym^2 \sh{V}]\rightarrow k[\sh{V}]$ has kernel equal to the ideal cut out by $\mathcal{R}_1$ and image equal to the fixed point subalgebra under the involution of $k[\sh{V}]$ induced by $v\mapsto -v$.
    This is a classical computation with the Veronese embedding, which we omit.

    To apply this general fact to our situation, 
   let us write $B \in \Sym^2(V^{\vee})$ for the bilinear form $(a,b)\mapsto \psi(a,xb)$. 
    Then $B$ defines a section of the constant bundle $\Sym^2(\underline{V}^{\vee})$, hence (by restriction)  a section $\sigma$ of $\Sym^2(\sh{F}^{\vee})$.
    We claim that there is a diagram both of whose squares are cartesian:
\begin{equation}
\label{diagram_threesquares}
\begin{tikzcd}
	T & X & {\OGr(g,V)} \\
	{\mathbb{V}(\sh{F}^{\vee})} & {\mathcal{R}_1} & {\mathbb{V}(\Sym^2(\sh{F}^{\vee}))}
	\arrow["\pi_1", from=1-1, to=1-2]
	\arrow["\pi", bend left = 30, from=1-1, to=1-3]
	\arrow[from=1-1, to=2-1]
	\arrow[from=1-2, to=1-3]
	\arrow[from=1-2, to=2-2]
	\arrow["\sigma", from=1-3, to=2-3]
	\arrow["{\nu_1}", from=2-1, to=2-2]
	\arrow["\nu"', bend right = 30, from=2-1, to=2-3]
	\arrow[hook, from=2-2, to=2-3]
\end{tikzcd}
\end{equation}
    Indeed, the rightmost square is cartesian by the very definition of $X$.
    Consider the outer square with arrows $\pi, \sigma$ and $\nu$.
    We describe the remaining morphism $T\rightarrow \mathbb{V}(\sh{F}^{\vee})$.
    Given an isotropic subspace $L\subset V$ corresponding to an element of $\OGr(g,V)(k)$, let $\tilde{L}\subset \tilde{V}$ be a $g$-dimensional subspace whose projection to $V$ equals $L$.
    Then there exists a unique linear functional $f \colon L \rightarrow k$ such that $\tilde{L} = \{(\ell, f(\ell))\colon \ell \in L\}$; in this case we write $\tilde{L} \leftrightarrow (L,f)$.
    This procedure can be carried out in families, and the scheme $\mathbb{V}(\sh{F}^{\vee})$ parametrizes $g$-dimensional subspaces of $\tilde{V}$ whose projection in $V$ is $g$-dimensional and isotropic for $\psi$.
    The map $\alpha\colon T\rightarrow \mathbb{V}(\sh{F}^{\vee})$ sends an element $\tilde{L}\leftrightarrow (L,f)$ to $f$.
    We claim that this map makes the outer square into a pullback diagram.
    This follows from the fact that if $\tilde{L} \leftrightarrow (L,f)$ and $L \in \OGr(g,V)$, then $\tilde{L}$ is isotropic for both $b_1$ and $b_2$ if and only if $\sigma(L)=\nu(f)$.

    There exists a unique morphism $\pi_1\colon T\rightarrow X$ such that all the squares and triangles in the completed diagram are commutative.
    Since the rightmost square and outer square are Cartesian, it follows from formal diagram-chasing that the leftmost square (involving $\pi_1$ and $\nu_1$) is Cartesian. 
    By the general fact of the first paragraph of the proof, $\nu_1$ is a scheme-theoretic quotient by the involution $(L,f)\mapsto (L, -f)$ on $\mathbb{V}(\sh{F}^{\vee})$.
    Since the formation of quotients by involutions of $\Z[1/2]$-schemes commutes with arbitrary base change on the target, the lemma follows.
 \end{proof}

\begin{theorem}\label{theorem_imagePsi_X}
    The scheme-theoretic image of $\Psi\colon J\rightarrow \OGr(g,V)$ equals $X$.
    Consequently, an element $[L]\in \OGr(g,V)(k)$ lies in the image of $\Psi(k)$ if and only if the symmetric bilinear form $L \times L \to k$, $(a, b) \mapsto \psi(a, x b)$, has rank $\leq 1$. 
\end{theorem}
\begin{proof}
    Since $\pi\colon T\rightarrow \OGr(g,V)$ lands in $X$ and by the commutativity of the diagram \eqref{diagram_psiandpsitilde}, the image of $\Psi$ is contained in $X$.
    By Part 3 of Proposition \ref{prop_firstpropertieskummerembedding}, $\Psi\colon J\rightarrow \OGr(g,V)$ factors through a closed embedding $K\hookrightarrow X$.
    To show that this closed embedding an isomorphism, it suffices to show that $X$ is reduced, irreducible and $g$-dimensional.
    But it is classically known (see for example \cite[Section 2.2]{Wan18}) that $T$ is a torsor for $J$, so in particular is reduced, irreducible and $g$-dimensional.
    Since $X$ is the scheme-theoretic quotient of $T$ under the involution $\iota$ by Lemma \ref{lem_image_T_in_OGr}, $X$ also has these properties.
\end{proof}

\begin{proposition}\label{prop_J_iso_T}
    $\tilde{\Psi}$ is an isomorphism $J\rightarrow T$.
\end{proposition}
\begin{proof}
    We may assume $k$ is algebraically closed.
    Since $T$ is a torsor for $J$ \cite[Section 2.2]{Wan18}, we may choose an isomorphism $T\simeq J$ and view $\tilde{\Psi}$ as a morphism $J\rightarrow J$.
    Since a morphism between abelian varieties is the composition of a translation and a homomorphism, it suffices to prove that this morphism has generic degree $1$.
    This follows from the commutative diagram \eqref{diagram_psiandpsitilde} and the fact that $J\rightarrow K$ and $T\rightarrow X$ have degree $2$.
\end{proof}

\begin{corollary}
    Let $L\subset V$ be a $g$-dimensional isotropic subspace and let $A$ be the Gram matrix of $(a,b)\mapsto \psi(a,xb)$ restricted to $L\times L$ in a choice of basis of $L$.
    Then the point $[L] \in \OGr(g, V)(k)$ lies in the image of $K(k)\hookrightarrow \OGr(g,V)(k)$ if and only if $A$ has rank $\leq 1$.
    If $A$ has rank $0$, then $[L]$ lifts to a point of $J(k)$.
    If $A$ has rank $1$, then $A = \lambda v v^t$ for some nonzero $v\in k^g$ and $\lambda \in k$, and $[L]$ lifts to an elements of $J(k)$ if and only if $\lambda$ is a square in $(k^{\times})^2$.
\end{corollary}
\begin{proof}
    By Theorem \ref{theorem_imagePsi_X}, Proposition \ref{prop_J_iso_T} and the commutative diagram \eqref{diagram_psiandpsitilde}, if $[L] = \Psi(x)$ and $x\in K(k)$, then $x$ lifts to an element of $J(k)$ if and only if $[L]$ lifts to an element of $T(k)$.
    Since the outer square of \eqref{diagram_threesquares} involving $\pi, \sigma$ and $\nu$ is Cartesian, $[L]$ lifts to an element of $T(k)$ if and only if $A$ lies in the image of $\nu$.
\end{proof}
Theorem
\ref{theorem_imagePsi_X} can be used to show that $K$ is cut out, as a subvariety of $\Sigma(\OGr(g,V)) \leq \P(S)$, by quartic equations. This is the content of the following proposition. The possibility that the defining equations of a projectively embedded Jacobian variety (or its Kummer variety) can be studied via a realisation as a degeneracy locus is not new (see e.g.\ \cite{Gru13}, which treats a number of cases arising from Vinberg theory in detail), although we are not aware of anywhere in the literature where a complete description of the case of hyperelliptic curves has been written down. 
\begin{proposition}
    The subvariety $K  \leq \Sigma(\OGr(g,V))$ is defined scheme-theoretically by quartic equations; consequently, $K \leq \P(S)$ is defined scheme-theoretically by quadric and quartic equations. 
\end{proposition}
\begin{proof}
    Let $\sh{F}\subset V\otimes\O_{\OGr(g,V)}$ denote the tautological subsheaf on $\OGr(g, V)$. Define a symmetric bilinear form $\theta$ on $V \times V$ by the formula $\theta(a, b) = \psi(a, xb)$, and a multilinear form $\theta'$ on $V \times V \times V \times V$ by the formula $\theta'(a, b, c, d) = \det \left( \begin{smallmatrix} \theta(a, c) & \theta(a, d) \\ \theta(b, c) & \theta(b, d) \end{smallmatrix} \right)$. Let $s' \in \HH^0(\OGr(g,V),\sh{F}^{\otimes 4, \vee})$ be the restriction of $\theta'$ to $\sh{F}$.
    Theorem \ref{theorem_imagePsi_X} shows that $K$ may be defined scheme-theoretically as the vanishing locus of $s'$. 

    Under the equivalence of categories between $\SO(V)$-equivariant vector bundles on $\OGr(g,V)$ and representations of $\Stab_{\SO(V)}(F)$, $\sh{F}$ is the sheaf of sections of the bundle $\mathbb{V}(\sh{F})$ corresponding to the inflation of the defining representation $F$ of $\GL(F)$ under the Levi quotient map $\Stab_{\SO(V)}(F)\twoheadrightarrow \GL(F)$.
    In the notation of Lemma \ref{lem_action_of_maximal_parabolic_on_vacuum_vector}, this representation has highest weight $\varepsilon_0$.
    Let $D_4 \leq S_4$ denote the Sylow 2-subgroup containing the transposition $(12)$, and let $\chi : D_4 \to \{ \pm 1 \}$ denote the character trivial on double transpositions and satisfying $\chi((12)) = -1$. Let $\sh{E}$ denote the sheaf of sections of the isotypic subbundle $\mathbb{V}(\sh{E}) = \mathbb{V}(\sh{F})^{\otimes 4, \vee, \chi}$. 
    The symmetry properties of $\theta'$ imply that $s' \in \sh{E}$. The $\SO(V)$-equivariant bundle $\mathbb{V}(\sh{E})$ corresponds to the irreducible representation of $\GL(F)$ of highest weight $-2\varepsilon_{g-2} - 2\varepsilon_{g-1}$. (Note that the dimension of this representation is independent of the characteristic, assumed not 2, as this dominant weight is $p$-restricted. The computation of the highest weight can be done in characteristic 0, using Schur--Weyl duality.) To complete the proof, it's enough to show that $\sh{E}^\vee \otimes \cO_{\OGr(g, V)}(4)$ is globally generated, as then we can embed $\sh{E} \hookrightarrow \cO_{\OGr(g, V)}(4)^N$ for some $N \geq 1$, and identify $K$ with the zero locus of the image of $s'$ in $\HH^0( \OGr(g, V), \cO_{\OGr(g, V)}(4)^N)$. 
    
     This global generation is presumably well-known, although we are not aware of a convenient reference, so we give the argument here.  
    The line bundle $\cO_{\OGr(g,V)}(2)$ is $\SO(V)$-equivariant and associated to the representation $\det^{\vee} = (\wedge^gF)^{\vee}$, by Lemma \ref{lemma_spinorembedding_squareroot_plucker}. 
    Therefore the bundle $\sh{E}^\vee \otimes \cO_{\OGr(g, V)}(4)$ corresponds to the irreducible $\GL(F)$-module with highest weight $-2\varepsilon_2 - \cdots - 2\varepsilon_{g-1}$. 
    Since this bundle has lowest weight $-2\varepsilon_0 - \cdots - 2\varepsilon_{g-3}$ and since this weight is anti-dominant considered as a weight of $\SO(V)$, the Borel--Weil theorem shows that $\HH^0(\O_{\OGr(g,V)},\sh{E}^\vee \otimes \cO_{\OGr(g, V)}(4))$ is the irreducible $\SO(V)$-representation with highest weight $2\varepsilon_0 + \cdots + 2\varepsilon_{g-3}$. In particular, this space is nonzero! 
    By equivariance, the global generation will follow if we can show that the map 
    \[ \HH^0(\OGr(g, V), \sh{E}^\vee \otimes \cO_{\OGr(g, V)}(4)) \to (\sh{E}^\vee \otimes \cO_{\OGr(g, V)}(4))|_P \]
    is surjective. 
    However, the image is a nonzero submodule of an irreducible $P$-module. This proves the first sentence of the proposition. The second follows immediately from Proposition \ref{prop_OG_cut_out_by_quadrics}.
\end{proof}
The following proposition shows that the only quadrics in $\P(S)$ vanishing on $K$ already vanish on the whole of $\OGr(g, V)$. 
\begin{proposition}\label{prop_quadrics_vanishing_on_K}
    Define 
    \[ I_{\OGr(g, V)} = \ker(\HH^0(\P(S), \cO_{\P(S)}(2)) \to\HH^0(\OGr(g, V), \cO_{\OGr(g, V)}(2)), \]
    and 
    \[ I_K = \ker(\HH^0(\P(S), \cO_{\P(S)}(2)) \to\HH^0(J, \O_J(2\Theta)^{\otimes 2}). \]
    Then $I_{\OGr(g, V)} = I_K$.
\end{proposition}
\begin{proof}
    We consider the commutative diagram with exact rows
    \[ \xymatrix{ 0 \ar[r] & I_{\OGr(g, V)}\ar[d] \ar[r] &\HH^0(\P(S), \cO_{\P(S)}(2))  \ar[r]\ar[d] &\HH^0(\OGr(g, V), \cO_{\OGr(g, V)}(2)) \ar[r] \ar[d] & 0 \\
    0 \ar[r] & I_K \ar[r] &\HH^0(\P(S), \cO_{\P(S)}(2)) \ar[r] &\HH^0(J, \O_J(2\Theta)^{\otimes 2}).} \]
    By the snake lemma, it is equivalent to show that the map $\Psi^\ast :\HH^0(\OGr(g, V), \cO_{\OGr(g, V)}(2)) \to\HH^0(J, \O_J(2\Theta)^{\otimes 2})$ is injective. The Borel--Weil theorem (and Lemma \ref{lemma_spinorembedding_squareroot_plucker}) shows that $\dim_k\HH^0(\OGr(g, V), \cO_{\OGr(g, V)}(2))  = \dim_k \bigwedge^g V = \binom{2g+1}{g}$. On the other hand, \cite[Theorem 4]{Kem88} shows that the dimension of the image of $\HH^0(\P(S), \cO_{\P(S)}(2)) \to\HH^0(J, \O_J(2\Theta)^{\otimes 2})$ is equal to the number of points $T \in J[2]$ not lying on $\Theta$ (equivalently, of Mumford degree $g$); there are $\binom{2g+1}{g}$ of these. It follows that the map $\HH^0(\OGr(g, V), \cO_{\OGr(g, V)}(2)) \to\HH^0(J, \O_J(2\Theta)^{\otimes 2})$ must be injective, as required. 
\end{proof}
We conclude this section with an example. Suppose that $g = 2$; then the spinor map $\Sigma\colon \OGr(2,V)\rightarrow \P^3$ is an isomorphism. The symmetric bilinear form $\psi(a, x b)$ on $V$ is given, with respect to the standard basis $\mathcal{P}$, by the matrix
\[ T =  \left(
\begin{array}{ccccc}
 0 & 0 & 0 & 1 & 0 \\
 0 & 0 & 1 & 0 & 0 \\
 0 & 1 & -c_1 & -\frac{c_2}{2} & 0 \\
 1 & 0 & -\frac{c_2}{2} & -c_3 & -\frac{c_4}{2} \\
 0 & 0 & 0 & -\frac{c_4}{2} & -c_5 \\
\end{array}
\right). \]
In the affine open $(x_1 \neq 0)$ of $\P^3$, co-ordinates are given by $\xi_0 = -x_3 / x_1$, $\xi_1 = x_2/x_1$, and $\xi_{01} = - \frac{1}{2} x_4 / x_1$; the generic isotropic subspace is the span of the columns of the matrix
\[ L =  \left(
\begin{array}{cc}
 \xi_{01} -\frac{1}{2}\xi_0 \xi_1 & -\frac{1}{2} \xi_0^2 \\
 -\frac{1}{2} \xi_1^2 & -\xi_{01}-\frac{1}{2}\xi_0 \xi_1 \\
 \xi_1 & \xi_0 \\
 1 & 0 \\
 0 & 1 \\
\end{array}
\right). \]
We conclude that the intersection of the Kummer with this affine open is equal to the locus where the $2 \times 2$ matrix ${}^t L T L$ has rank $\leq 1$; equivalently, where the determinant is equal to 0. Computing this determinant and homogenising, we find that the Kummer equals the quartic hypersurface given by the equation
\begin{multline*}  (c_4^2-4 c_3 c_5) x_1^4-4 c_2 c_5  x_1^3x_2-2 c_2 c_4  x_1^3x_3-4 c_5  x_1^3 x_4-4 c_1 c_5 x_1^2 x_2^2+(c_2^2 -4 c_1 c_3 +2 c_4 )x_1^2 x_3^2\\+(4c_5-4 c_1 c_4) x_1^2 x_2 x_3 +2 c_4 x_1^2 x_2 x_4 +4 c_3 x_1^2x_3 x_4 -4 c_5  x_1 x_2^3-2 c_2  x_1x_3^3-4 c_3 x_1 x_2 x_3^2 \\-4 c_4 x_1 x_2^2 x_3 -4 c_1 x_1 x_3^2 x_4 +2 c_2 x_1 x_2 x_3 x_4 +4 x_1 x_3 x_4^2 +x_3^4-6 x_2 x_3^2+x_2^2 x_4^2, 
\end{multline*}
in agreement with the formula in e.g. \cite[Ch. 4]{CasselsFlynnProlegomena}.

\subsection{Complement: a twisted Kummer embedding}\label{subsec_twisted_Kummer}

Let $P \in J(k)$ be a rational point. Then the morphism $\pi_P : J \to J$, $\pi_P(Q) = [2](Q) + P$, is a soluble 2-covering, which is trivial if and only if $P \in 2J(k)$. The linear system $| \Theta + t_P^\ast \Theta|$ defines (after a choice of co-ordinate system) a morphism $\Phi_P : J \to \P^{2^g-1}$. If $k$ is a number field, and  we write $h : \P^{2^g-1}(k) \to \R$ for the standard absolute logarithmic Weil height, then we have a formula
\[ h \circ \Phi_P  = \frac{1}{4} h \circ (\Sigma \circ \Psi \circ \pi_P) + O(1), \]
where the implied constant depends both on the choice of co-ordinate system and on $f$ and $P$. A description of $\Phi_P$ can therefore be useful in searching for points of $J(k)$ of small height in the coset $P + 2 J(k)$; this point of view is developed e.g.\ in  \cite{Fis12}. The above relation also underlies, in some sense, the approach to naive heights taken in \cite{lagathorne2024smallheightoddhyperelliptic}. 

Our construction of the morphism $\Psi : J \to \OGr(g, V)$ admits a twisted variant, that can be used to study the linear system $| \Theta + t_P^\ast \Theta|$. Indeed, let $\sh{L}$ be any invertible sheaf of degree $2g-1$ on $C$, and let $P \in J(k)$ denote the class of $\sh{L}((1-2g) P_\infty)$. The formula (in the notation of \S\ref{subsec_basic_version_keyconstruction})
\[ \Psi_{\sh{L}}(\sh{N}) = \varphi^{\sh{N}}_{\sh{L} \otimes \sh{N}^{\otimes 2}, \sh{L}}(F_{\sh{L} \otimes \sh{N}^{\otimes 2}}) \]
defines a morphism $\Psi_{\sh{L}} : J \to \OGr(g, V_{\sh{L}})$. Moreover, we can use the formalism of the spin representation to define an embedding $\OGr(g, V_{\sh{L}}) \to \P^{2^g-1}$ (with a well-defined set of co-ordinates). Suppose that $\sh{L} = \cO_C(D)$, where $D = E - m P_\infty$ is a reduced divisor corresponding to a Mumford triple $(U, V, R)$ of degree $m$, where (for simplicity)  $U$ and $f$ are coprime. Then a model of $V_{\sh{L}}$ is given by underlying vector space $k[x] / (f(x))$, with symmetric bilinear form $\psi_{\sh{L}}(a, b) = \tau(ab / U)$. The basis 
\[ U, x U, \dots, x^{b-a} U, (y-R), x^{b-a} U, x (y-R), x^{b-a+2} U, \dots, x^{m-1} (y-R)), x^{b-a+m} U, x^{b-a+m+1} U, \dots, x^{2g-m} U, \]
where $a = \lfloor m / 2 \rfloor$, $b = \lfloor (2g+1-m) /2 \rfloor$, satisfies the hypotheses of Lemma \ref{lem_straightening_basis} (cf. the proof of \cite[Lemma 4.9]{lagathorne2024smallheightoddhyperelliptic}), so there is a well-defined basis $\mathcal{P}_{\sh{L}} = (p_{0, \sh{L}}, \dots, p_{2g, \sh{L}})$ of $V_{\sh{L}}$ satisfying $\psi_{\sh{L}}(p_{i, \sh{L}}, p_{j, \sh{L}}) = 0$ if $i + j \neq 2g$; then  $\langle p_{0, \sh{L}}, \dots, p_{g-1, \sh{L}}\rangle$ and $\langle p_{g+1, \sh{L}}, \dots, p_{2g, \sh{L}} \rangle$ are transverse isotropic subspaces, and the formalism of \S \ref{sec_odd_orthogonal_spaces} determines is a well-defined associated embedding $\Sigma_{\sh{L}} : \OGr(g, V) \to \P^{2^g-1}$. 
\begin{proposition}
Let $\sh{L}$ be any invertible sheaf of degree $2g-1$ on $C$, and let $P \in J(k)$ denote the class of $\sh{L}((1-2g) P_\infty)$. Then:
\begin{enumerate}
    \item There is an isomorphism $\Psi_{\sh{L}}^\ast \cO_{\OGr(g, V_{\sh{L}})}(1) \cong \cO_J(\Theta + t_P^\ast \Theta)$. Moreover, the morphism $\Sigma_{\sh{L}} \circ \Psi_{\sh{L}} : J \to \P^{2^g-1}$ may be identified with the one associated to the complete linear system $|\Theta + t_P^\ast \Theta|$. 
    \item Let $K_{\sh{L}}$ denote the quotient of $J$ by the involution $Q \mapsto P - Q$. Then $\Psi_{\sh{L}}$ factors through a closed embedding $K_{\sh{L}} \to \OGr(g, V_{\sh{L}})$.
    \item The image $K_{\sh{L}} \leq \OGr(g, V_{\sh{L}})$ equals the symmetric determinantal variety of subspaces $L\subset V_{\sh{L}}$ with the property that $\rank \psi_{\sh{L}}(a, T_f b)|_{L \times L} \leq 1$; the image $\Sigma_{\sh{L}}(K_{\sh{L}}) \leq \P^{2^g-1}$ may be defined scheme-theoretically by quadrics and quartics. \end{enumerate}
\end{proposition}
\begin{proof}
    We omit the proofs of these statements; they can either be proved directly or by reduction to the untwisted case, after passing to the algebraic closure so that $P \in 2 J(k)$.
\end{proof}

\section{Duplication on the Kummer}\label{sec_duplication_on_the_Kummer}

Let $k$ be a field of characteristic not 2, and let $f(x)  = x^{2g+1} +  c_1 x^{2g} + \dots + c_{2g+1} \in k[x]$ be a polynomial of nonzero discriminant, with associated hyperelliptic curve $C$, Jacobian $J$, and Kummer $K$. In \S \ref{sec_hyperelliptic_curves_and_kummers}, we have shown how the theory of spinors can be used to make explicit the embedding $K \hookrightarrow \P^{2^g-1}$ associated to the linear system $|2 \Theta|$ of $J$. 

Since the duplication map $[2] : J \to J$ commutes with the action of $-1$, it descends to an endomorphism $[2] : K \to K$ of the Kummer variety. The goal of this section is to make this endomorphism explicit. Thanks to Tate's formula $\widehat{h}_\Theta(P) = \frac{1}{2} \lim_{n \to \infty} 4^{-n} h([2^n](P))$ for the canonical height, this will have direct applications to canonical heights, which are worked out in \S \ref{sec_applications_to_heights} below. 

We first show, in \S \ref{subsec_duplication_on_the_kummer}, that $[2] : K \to K$ can be represented by a tuple of quartic polynomials (non-unique in general, since there are quartics that vanish on $K$ when $g \geq 2$). The proof of this fact shows that one can write down such quartics in terms of the action of the theta group $\mathcal{G}(2 \Theta)$ on $\HH^0(J, \cO_{J}(2 \Theta))$. In \S \ref{subsec_action_of_Heisenberg_group}, we study a section of the morphism $\mathcal{G}(2 \Theta) \to J[2]$ that has particularly nice properties, and in \S\ref{subsec_generic_spin_basis} we give an example where we compute quartics representing duplication for a curve of genus 4 over $\F_5$.  

\subsection{Constructing duplication polynomials}\label{subsec_duplication_on_the_kummer}

We have defined co-ordinates on the spin representation $S$; let $x_1, \dots, x_{2^g}$ denote the dual basis of $S^\vee = \HH^0(\P(S), \cO_{\P(S)}(1))$. Let $\infty = (0, 0, \dots, 1) \in S$; then $x_{2^g}(\infty) = 1$. We write $0_J : \Spec k \to J$ for the identity section, and  rigidify $\cO_J(2 \Theta)$ along $0_J$ using the isomorphism $\phi : 0_J^\ast \O_J(2 \Theta) \cong k$ that maps $0_J^\ast 1 = 0_J^\ast \Psi^\ast(x_1)$ to $1$.  
\begin{proposition}\label{prop_normalised_isomorphism}
    There is a unique isomorphism $\xi : [2]^\ast \O_J(2 \Theta) \cong \O_J(2 \Theta)^{\otimes 4}$ respecting the induced trivialisations of the fibres at $0_J$.
\end{proposition}
\begin{proof}
    The existence of such an isomorphism follows from the theorem of the square and the fact that $2 \Theta$ is symmetric \cite[\S 6, Corollary 3]{Mum66}. 
    Any two such isomorphisms differ by a $k^{\times}$-multiple, hence there is a unique one respecting the given rigidification. 
\end{proof}
\begin{proposition}\label{prop_beta_as_a_pullback_of_theta}
    Consider the pullback of $\beta \in S^\vee \otimes S^\vee$ under $\Psi \times \Psi$ as an element of $\HH^0(J \times J, \O_J(2 \Theta) \boxtimes \O_J(2 \Theta))$, and let $\mu : J \times J \to J \times J$ denote the isogeny $\mu(x, y) = (x+y, x-y)$. Then the divisor of $(\Psi \times \Psi)^\ast(\beta)$ equals $\mu^{\ast}( \Theta \times J + J \times \Theta)$.
\end{proposition}
\begin{proof}
    The seesaw principle shows that there is an isomorphism $\eta :   \O_J(2 \Theta) \boxtimes \O_J(2 \Theta) \cong \mu^\ast \cO_{J \times J}(\Theta \times J + J \times \Theta)$ (see e.g.\ \cite[\S 3, Proposition 1]{Mum66}). The space $\HH^0(J \times J, \cO_{J \times J}(\Theta \times J + J \times \Theta))$ contains a unique nonzero section $s$, up to scalar multiple; its divisor is $\Theta \times J + J \times \Theta$. We will show that $\eta((\Psi \times \Psi)^\ast(\beta))$ is a scalar multiple of $\mu^\ast(s)$. 

    The kernel of $\mu$ equals the diagonally embedded subgroup $\Delta J[2] \leq J \times J$. The isomorphism $\eta$ determines a section $\sigma_1 : \Delta J[2] \to \mathcal{G}(\O_J(2 \Theta) \boxtimes \O_J(2 \Theta))$ of the canonical surjection
    $\mathcal{G}(\O_J(2 \Theta) \boxtimes \O_J(2 \Theta)) \to J[2] \times J[2]$ over the subgroup $\Delta J[2] \leq J[2] \times J[2]$; $\eta^{-1} \mu^\ast(s)$ is the unique (up to scalar multiple) element of $\HH^0(J \times J, \O_J(2 \Theta) \boxtimes \O_J(2 \Theta))$ invariant under the induced action of $\Delta J[2]$.

    We can define a second section $\sigma_2 : \Delta J[2] \to \mathcal{G}(\O_J(2 \Theta) \boxtimes \O_J(2 \Theta))$ as follows. There is a natural surjection $\mathcal{G}(2 \Theta) \times \mathcal{G}(2 \Theta) \to \mathcal{G}(\O_J(2 \Theta) \boxtimes \O_J(2 \Theta))$ with kernel $\{ (t, t^{-1}) \} \leq \G_m \times \G_m$. Using our identification $\mathcal{G}(2 \Theta) = \Stab_{S \Gamma}(T_f)$, we define a map
    \[ \mathcal{G}(2 \Theta) \to \mathcal{G}(2 \Theta) \times \mathcal{G}(2 \Theta) \to \mathcal{G}(\O_J(2 \Theta) \boxtimes \O_J(2 \Theta)) \]
    where the first map is $\widetilde{T} \mapsto (\widetilde{T}, N(\widetilde{T})^{-1} \widetilde{T})$ and the second is the natural surjection. This map factors through a map
    \[ J[2] \to \mathcal{G}(\O_J(2 \Theta) \boxtimes \O_J(2 \Theta)), \]
    and this is our desired section $\sigma_2$. Note that  $\beta$ is invariant under the action of $\Delta J[2]$ by $\sigma_2$, by Lemma \ref{lem_covariance_of_beta}.  
   
    We have $\sigma_1 = \sigma_2 \cdot \chi$, for some character $\chi : \Delta J[2] \to \G_m$. If we can show that $\chi$ is trivial, we'll be done, using that $\eta^{-1} \mu^\ast(s)$ is the unique $\sigma_1$-invariant section of $\HH^0(J \times J, \O_J(2 \Theta) \boxtimes \O_J(2 \Theta))$. At this point, we note that our construction makes sense for the generic odd hyperelliptic curve. Over the generic point, the only character $\Delta J[2] \to \G_m$ is the trivial one. By specialisation, we find that $\chi$ is always trivial. This completes the proof. 
\end{proof}
\begin{proposition}\label{prop_Riemann_theta_relation_as_incidence_relation}
    Let $T \in J[2]$ be a 2-torsion point of Mumford degree $g$, and let ${\widetilde{T}}$ be a pre-image in $\mathcal{G}(2 \Theta)$. Then:
    \begin{enumerate}
        \item The bilinear form $\beta( \widetilde{T} \cdot x, y)$ is symmetric; we write $q_{\widetilde{T}}(x) = \beta( \widetilde{T} \cdot x,x) \in \Sym^2 S^\vee =\HH^0(\P(S), \cO_{\P(S)}(2))$ for the associated quadratic form.
        \item Let $f_{\widetilde{T}}(x) = \beta(\widetilde{T} \cdot \infty,x) \in S^\vee =\HH^0(
        \P(S), \cO_{\P(S)}(1))$. Then we have the identity
        \[ f_{\widetilde{T}}(\infty) \cdot \xi([2]^\ast \Psi^\ast f_{\widetilde{T}}) = \Psi^\ast(q_{\widetilde{T}})^2 \]
        in $\HH^0(J, \cO_J(2 \Theta)^{\otimes 4})$.
    \end{enumerate}
\end{proposition}
\begin{proof}
    (1) We split into cases depending on the parity of $g$. The given identity is invariant under replacing $\widetilde{T}$ by a $\G_m$-multiple, so we are free to choose our favourite lift of $T$. If $g$ is even, we choose $\widetilde{T} = \epsilon_{i_1} \dots \epsilon_{i_g}$ for some subset $I = \{ i_1 < \dots < i_g \} \subset \{ 1, \dots, 2g+1 \}$ (notation as in \S \ref{subsec_action_of_Heisenberg_group} below). Then we have (cf. Lemma \ref{lem_covariance_of_beta})
    \[ \beta(\widetilde{T} y,x) = (-1)^{g(g+1)/2} \beta(x,\widetilde{T} y) = (-1)^{g(g+1)/2} N(\widetilde{T})^{-1} \beta(\widetilde{T} x,\widetilde{T}^2 y ). \]
    Now observe that
    \[ \widetilde{T}^2 = \epsilon_{i_1} \dots \epsilon_{i_g} \epsilon_{i_1} \dots \epsilon_{i_g} = (-1)^{g(g-1)/2} N(\widetilde{T}), \]
    so we get
    \[ \beta(\widetilde{T} y,x) = (-1)^{g(g+1)/2 + g(g-1)/2} \beta(\widetilde{T} x,y) = \beta( \widetilde{T} x,y), \]
    since $g$ is even. Now suppose that $g$ is odd. In the case, we choose $\widetilde{T} = \epsilon_{i_1} \dots \epsilon_{i_{g+1}}$ for some subset $I = \{ i_1 < \dots < i_{g+1} \} \subset \{ 1, \dots, 2g+1 \}$ (the point being that we need a monomial in an even number of elements, in order to land in the special Clifford group $S \Gamma$). A similar computation to the one before now gives
    \[ \beta(\widetilde{T} y,x) = (-1)^{g(g+1)/2 + g(g+1)/2} \beta(\widetilde{T} x,y), \]
    with the sign again trivial, now because $g(g+1)$ is even. 
    
    (2)  By Proposition \ref{prop_beta_as_a_pullback_of_theta}, both sections $\xi([2]^\ast \Psi^\ast f_{\widetilde{T}})$ and $\Psi^\ast(q_{\widetilde{T}})^2$ of $\HH^0(J, \O_J(2 \Theta)^{\otimes 4})$ have associated divisor $2( [2]^{\ast}(t_T^\ast\Theta) )$. We must therefore have an equality
    \[ \lambda \cdot \xi([2]^\ast \Psi^\ast f_{\widetilde{T}}) = \Psi^\ast(q_{\widetilde{T}})^2 \]
    for some $\lambda \in k^\times$. To evaluate $\lambda$, we pull back both sides along $0_J$ and use the definition of $\xi$ (Proposition \ref{prop_normalised_isomorphism}).
\end{proof}
Since the morphism $[2] : J \to J$ commutes with the action of $[-1]$, it descends to a morphism $[2] : K \to K$ of the Kummer. We now seek to find quartic polynomials representing $[2]$, in the following sense.
\begin{definition}\label{def_represents_duplication}
    Let $\delta : S^\vee \to\HH^0(\P(S), \cO_{\P(S)}(4)) = \Sym^4 S^\vee$ be a $k$-linear map. We say that $\delta$ represents duplication if the induced map $\Psi^\ast \circ \Sigma^\ast \circ \delta : S^\vee \to\HH^0(J, \O_J(2 \Theta)^{ \otimes 4})$ coincides with the composite
\begin{equation}\label{eqn_duplication_on_kummer} \xi \circ (\Sigma \circ \Psi \circ [2])^\ast : S^\vee \to\HH^0(J, [2]^\ast \O_J(2 \Theta)) \to\HH^0(J,\O_J(2 \Theta)^{\otimes 4}). 
\end{equation}
\end{definition}
If $\delta$ represents duplication, and $x \in S - \{ 0 \}$ represents a point $[x] \in K(k)$, then $\operatorname{ev}_x \circ \delta \in \Hom_k(S^\vee, k)$ may be identified with a point of $S - \{ 0 \}$, that represents $[2]([x]) \in K(k)$. We write $\delta(x)$ for this point. It follows from the definition that we have $\delta(\infty) = \infty$.

Suppose that $\delta, \delta'$ represent duplication. Then $\delta - \delta'$ takes values in $\ker(\HH^0(\P(S), \cO_{\P(S)}(4)) \to\HH^0(J, \O_J(2 \Theta)^{\otimes 4})$. When $g \geq 2$, this kernel is nonzero, so $\delta$ is not unique. However, we do have the following useful property:
\begin{proposition}\label{prop_pairing_delta_with_torsion_points}
    Suppose that $\delta : S^\vee \to\HH^0(\P(S), \cO_{\P(S)}(4))$ represents duplication. Then for any $\widetilde{T} \in \mathcal{G}(2 \Theta)$ lifting a point $T \in J[2]$ of Mumford degree $g$, and any $x \in S - \{ 0 \}$ representing a point of $K(k)$, we have
    \[ \beta( \widetilde{T} \infty, \delta(x)) \beta(\widetilde{T} \infty, \infty) = \beta(\widetilde{T} x,x)^2\]
    as elements of $k$. 
\end{proposition}
\begin{proof}
    This follows from Proposition \ref{prop_Riemann_theta_relation_as_incidence_relation} on unwinding the definitions. 
\end{proof}
It remains to show that a choice of $\delta : S^\vee \to \Sym^4 S^\vee$ representing duplication actually exists; equivalently, that the image of the map 
\begin{equation}\label{eqn_restriction_of_quartic_forms} \Psi^\ast \circ \Sigma^\ast :\HH^0(\P(S), \cO_{\P(S)}(4)) \to\HH^0(J, \O_J(2 \Theta)^{\otimes 4}) 
\end{equation}
contains the image of the map
\[ \xi \circ  [2]^\ast \circ \Psi^\ast \circ \Sigma^\ast:\HH^0(\P(S), \cO_{\P(S)}(4)) \to\HH^0(J, \O_J(2 \Theta)^{\otimes 4}). \]
This seems to be a subtle question, since $K$ is not projectively normal in $\P(S)$ if $g\geq 3$, by \cite[\S2.9.3]{Khaled-projectivenormalitykummer}. The discussion in \cite[\S 2]{Stoll-heightsgenus3} shows that the map (\ref{eqn_restriction_of_quartic_forms}) is surjective if $g = 3$. We have verified that this map is also surjective when $g = 4$, and it is conceivable that one could show it is always surjective using the formalism of \cite{Gru13}. Here we take a different approach to the existence of $\delta$, suggested by Proposition \ref{prop_pairing_delta_with_torsion_points}, that will also be useful for applications to the study of canonical heights. 
\begin{definition}
    A generic spin basis is a sequence of elements $T_1, \dots, T_{2^{g}} \in J[2](k)$ with the following properties:
    \begin{enumerate}
    \item For each $i = 1, \dots, {2^g}$, $T_i$ has Mumford degree $g$.
    \item For any choice of pre-images $\widetilde{T}_i \in \mathcal{G}(2 \Theta)$, the points $\widetilde{T}_i(\infty) \in S$ are linearly independent. 
    \end{enumerate}
\end{definition}
\begin{lemma}\label{lem_from_spin_basis_to_duplication_polynomial}
    Let $(T_1, \dots, T_{2^g})$ be a generic spin basis. Then there exists a unique $k$-linear map $\delta : S^\vee \to \Sym^4 S^\vee$ such that for each $i = 1, \dots, 2^g$, we have
    \[ \beta( \widetilde{T}_i \infty,\delta(x)) \beta(\widetilde{T}_i \infty,\infty) = \beta( \widetilde{T}_i x,x)^2\]
    in $S\otimes \Sym^4 S^{\vee}$.
    Moreover, $\delta$ represents duplication, in the sense of Definition \ref{def_represents_duplication}.
\end{lemma}
\begin{proof}
    The constants $\beta( \widetilde{T}_i \infty, \infty)$ are non-zero, and the linear forms $f_{\widetilde{T}_i}$ form a basis for $S^\vee$. There is thus a unique linear map $S^\vee \to \Sym^4 S^\vee$ sending $f_{\widetilde{T}_i}$ to $\beta( \widetilde{T}_i \infty,\infty)^{-1} \beta(\widetilde{T}_i x,x)^2$ for each $i = 1, \dots, 2^g$. We need to explain why this $\delta$ represents duplication; equivalently, why the maps
    \[ \Psi^\ast \circ\Sigma^\ast  \circ \delta : S^\vee \to\HH^0(J, \O_J(2 \Theta)^{\otimes 4}) \]
    and 
    \[ \xi \circ [2]^\ast\circ \Psi^\ast \circ \Sigma^\ast: S^\vee \to\HH^0(J, \O_J(2 \Theta)^{\otimes 4}) \]
    are equal. However, they agree on the elements $f_{\widetilde{T}_i}$ of $S^\vee$, by construction, and these form a basis of $S^\vee$; so they are equal. 
\end{proof}
Following Lemma \ref{lem_from_spin_basis_to_duplication_polynomial}, we see that in order to demonstrate the existence of a choice of $\delta$ representing duplication, it is enough to demonstrate the existence of a generic spin basis, defined over some extension $L / k$. This is the content of the following proposition. 
\begin{proposition}\label{prop_existence_of_independent_torsion_points}
    Suppose that $k$ is separably closed. Then a generic spin basis exists. 
\end{proposition}
    The following corollary does not require the assumption that $k$ is separably closed.
\begin{corollary}\label{cor_existence_of_duplication}
    There exists a choice of $\delta : S^\vee \to \Sym^4 S^\vee$ representing duplication. 
\end{corollary}
\begin{proof}
    We begin with the elementary remark that if $U, V, W$ are finite-dimensional $k$-vector spaces, and $f \in \Hom_k(U, W)$, $g \in \Hom_k(V, W)$ are $k$-linear maps, then the condition that the image of $g \circ (\cdot) : \Hom_k(U, V) \to \Hom_k(U, W)$ contains $f$ can be checked after passage to any field extension $k' / k$. The existence of $\delta$ can be phrased in these terms; Proposition \ref{prop_existence_of_independent_torsion_points} and Lemma \ref{lem_from_spin_basis_to_duplication_polynomial} together show that the condition holds if $k$ is separably closed. This completes the proof. 
\end{proof}
To prove Proposition \ref{prop_existence_of_independent_torsion_points}, we first need to study more carefully the lifts $\widetilde{T} \in \mathcal{G}(2 \Theta)$ of generic 2-torsion points $T \in J[2]$. We do this in \S \ref{subsec_action_of_Heisenberg_group}, and then complete the proof in \S \ref{subsec_generic_spin_basis}. We conclude this section by recording, for later use, some further results concerning the quadrics $\beta( \widetilde{T} x,x)$.
\begin{proposition}\label{prop_q_tildes_inject_to_Kummer}
    \begin{enumerate}
        \item The action of $S \Gamma$ on $(\Sym^2 S^\vee) \otimes N$ factors through $\rho : S \Gamma \to \SO(V)$. Under the induced action of $J[2]$, there is an isomorphism $(\Sym^2 S^\vee) \otimes N \cong \oplus_S k(e_2(S, -))$, where $e_2 : J[2] \times J[2] \to \mu_2$ is the Weil pairing, and $S$ ranges over the set of points $S \in J[2]$ such that the theta-characteristic $(g-1) P_\infty + S$ is even. 
        \item Choose for each $T \in J[2]$ of Mumford degree $g$ a pre-image $\widetilde{T} \in \mathcal{G}(2 \Theta)$. Then the elements $\beta(\widetilde{T} x,x) \in\HH^0(\P(S), \cO_{\P(S)}(2))$ are linearly independent, and span a complement to the kernel $I_K$ of the map
        \[ \Psi^\ast :\HH^0(\P(S), \cO_{\P(S)}(2)) \to\HH^0(J, \O_J(2 \Theta)^{\otimes 2}). \]
    \end{enumerate}
\end{proposition}
\begin{proof}
    (1) By Proposition \ref{prop_decomposing_S_tensor_S}, there is an isomorphism $S \otimes_k S \otimes N^{-1} \cong \oplus_{0 \leq r \leq 2g+1 \text{ even}} \bigwedge^r V$ of $S \Gamma$-modules. Lemma \ref{lem_covariance_of_beta} shows that $S^\vee \cong S \otimes N^{-1}$ as $S \Gamma$-modules, hence $S^\vee \otimes_k S^\vee \otimes N \cong \oplus_{0 \leq r \leq 2g+1 \text{ even}} \bigwedge^r V$. Using the description of the action of $J[2]$ on $V$ implicit in Proposition \ref{prop_identification_of_stabiliser_and_2_torsion}, we see that $S^\vee \otimes_k S^\vee \otimes N$ is the regular representation of $J[2]$, isomorphic to the direct sum $\oplus_{S \in J[2]} k(e_2(S, -))$. By Corollary \ref{cor_decomposing_symmetric_square_of_S}, a 2-torsion point $S = [D_I]$ with $|I|$ even contributes to $\Sym^2 S^\vee \otimes N$ if and only if either $|I| \equiv g \text{ mod }4$ (if $g$ is even or $|I| \equiv g+1 \text{ mod }4$ (if $g$ is odd). Using e.g. \cite[Proposition 6.3]{MumfordTataII}, this is easily seen to be equivalent to the condition that $(g-1) P_\infty + S$ is an even theta characteristic.

    (2) The commutator pairing on $\mathcal{G}(2 \Theta)$ descends to the Weil pairing on $J[2]$. Therefore $\beta( \widetilde{T} x,x)$ lies in the $e_2(T, -)$-eigenspace of $\Sym^2 S^\vee \otimes N$. This shows that the forms $\beta( \widetilde{T} x,x)$, as $T$ varies, are linearly independent. Their images in $\HH^0(K, \cO_K(2 \Theta)^{\otimes 2})$ are non-zero because they are non-zero after evaluation at the point $[\infty] \in K$. Proposition \ref{prop_quadrics_vanishing_on_K} shows that the rank of the map $\HH^0(\P(S), \cO_{\P(S)}(2)) \to\HH^0(J, \O_J(2 \Theta)^{\otimes 2})$ is equal to the number of points $T \in J[2]$ of Mumford degree $g$, so the result follows. 
\end{proof}

\subsection{The action of the Heisenberg group}\label{subsec_action_of_Heisenberg_group}

We now show how to make explicit the action of the group $J[2]$ (or rather, its central extension $\mathcal{G}(2 \Theta)$) on $S$. Let $L / k$ be the Galois extension obtained by adjoining the roots $\omega_1, \dots, \omega_{2g+1}$ of $f(x)$. Then $J[2]$ is constant over $L$. We want to be able to write down, for any $T \in J[2](L)$, the matrix of a pre-image $\widetilde{T} \in \mathcal{G}(2 \Theta)$ in its action on $S$. This will naturally involve discriminants and resultants, so we introduce some notation. If $I = \{ i_1 < \dots < i_r \} \subset \{ 1, \dots, 2g+1 \}$ is a subset, then we define \begin{equation}\label{eqn_half_discriminant} \delta_I = \prod_{1 \leq j < k \leq r} (\omega_{i_k} - \omega_{i_j}).
\end{equation}
If $J$ is another such subset, then we define 
\begin{equation}\label{eqn_resultant} 
r_{I, J} = \prod_{i \in I, j \in J} (\omega_i - \omega_j). 
\end{equation}
If $I^c$ denotes the complement of $I$ in $\{ 1, 2, \dots, 2g+1 \}$, then we find $\Delta = (\delta_I \delta_{I^c} r_{I, I^c})^2$, where $\Delta$ is the usual discriminant of $f$. Note that $r_{I, I^c} = r_{I^c, I}$, as they differ by the sign $(-1)^{r(2g+1-r)} = (-1)^{r(r+1)} = 1$.

We recall that any $2$-torsion point $T$ can be represented as the divisor class of a divisor $D_I = \sum_{i \in I} (\omega_i, 0) - |I| P_\infty$, for a subset $I \subset \{ 1, \dots, 2g+1 \}$. This representation is unique if we insist that $|I|$ is even, or if we insist that $0 \leq |I| \leq g$. We have $[D_I] + [D_{J}] = [D_{I \Delta J}]$. 
\begin{lemma}\label{lem_representing_idempotent_in_J_basis}
    Let $\epsilon_i = \prod_{j \neq i}(x - \omega_j) \in L[x] / (f(x)) = V_L$. Expressing $\epsilon_i$ in the standard basis $\mathcal{P}$ of $V_L$, we have $\epsilon_i = p_{2g} + \omega_i p_{2g-1} + \dots + \omega_i^g p_g + \dots$, where the remaining coefficients are polynomials in the $\omega_j$ with coefficients in $\frac{1}{2}\Z$.
\end{lemma}
\begin{proof}
    We compute for any $j \geq 0$ that $\psi(\epsilon_i x^j) = \tau(\omega_i^j \epsilon_i) = \tr_{V_L / L}(\omega_i^j \epsilon_i / f'(\omega_i)) = \omega_i^j$, where the second equality uses a formula of Euler \cite[Chapter III, \S6]{serre-localfields}. This shows that the coefficients of $p_{2g}, \dots, p_g$ in $\epsilon_i$ are as claimed. We need to show that the remaining coefficients are polynomials in the $\omega_j$ with coefficients in $\frac{1}{2} \Z$. This follows immediately from Lemma \ref{lem_change_of_basis_matrix_is_integral}. 
\end{proof}
Given a subset $I = \{ i_1 < i_2 < \dots < i_r \} \subset \{1, \dots, 2g+1\}$, let us write $M_I \in M_{2^g \times 2^g}(L)$ for the matrix representing the action of $\delta_I^{-1} \epsilon_{i_1} \cdot \dots \epsilon_{i_r} \in C(V_L)$ on the spin representation $S$ (where the product is of vectors in the Clifford algebra, not as elements of $V_L = L[x] / (f(x))$!).
\begin{proposition}\label{prop_properties_of_heisenberg_matrix}
    \begin{enumerate}
        \item The matrix $M_I$ lies in the image of $\mathcal{G}(2\Theta)$, and its image in $J[2]$ equals $[D_I]$. 
        \item We have $M_I^2 = r_{I, I^c}$, and $M_{I^c} = M_I$. If $|I|$ is even then, identifying $M_I$ with an element of $S \Gamma$, we have $N(M_I) = (-1)^{|I|/2} r_{I, I^c}$.
        \item The coefficients of $M_I$ lie in $\Z[\omega_1, \dots, \omega_{2g+1}]$, and at least one matrix entry of $M_I$ equals $1$. 
        \item If $|I| = g$, then $\beta(M_I \infty, \infty) = 1$. If $|I| < g$, then $\beta (M_I \infty, \infty) = 0$.
    \end{enumerate}
\end{proposition}
\begin{proof}
    (1) We can assume that $k = L$. The vector $\epsilon_i \in V$ has the property that $\epsilon_i \cdot \epsilon_i = \psi(\epsilon_i, \epsilon_i) = f'(\omega_i)$ (product in  Clifford algebra $C(V)$). A standard calculation shows that any anisotropic vector in $V$ lies in the Clifford group $\Gamma$, when considered as an element of $C(V)$, and acts (via the representation $\rho : \Gamma \to \O(V)$) as the simple reflection in the hyperplane orthogonal to that vector. Looking at Proposition \ref{prop_identification_of_stabiliser_and_2_torsion}, we see that $\epsilon_i z \in \Stab_{S \Gamma}(T_f) = \mathcal{G}(2\Theta)$ is a pre-image of the 2-torsion point $[(\omega_i, 0) - P_\infty] \in J[2]$, where $z = p_g \prod_{i=0}^{g-1} (1 - p_{2g-i} p_i)$ is the central element of $\Gamma$ described in \S \ref{sec_odd_orthogonal_spaces}, that acts trivially on $S$; thus $M_{ \{ i \} }$ lies in the image of $\mathcal{G}(2\Theta)$. Taking products gives the result for a general matrix $M_I$. 

    (2) $M_I^2$ is the image of the element
    \[ \delta_I^{-2} \epsilon_{i_1} \dots \epsilon_{i_r} \epsilon_{i_1} \dots \epsilon_{i_r} = (-1)^{r(r-1)/2} \delta_I^{-2} \prod_{j=1}^r f'(\omega_{i_j}) = r_{I, I^c} \]
    of $\Gamma$.
    To compare $M_I$ and $M_{I^c}$, we first consider $M_B$, where $B = \{ 1, 2, \dots, 2g+1 \}$. This is the image of the element $\delta_B^{-1} \epsilon_1 \dots \epsilon_{2g+1}$ of $C(V)$, which has  odd degree and lies in the centre, so must in fact be a scalar multiple of $z$. To compute this scalar, we consider the image of each element under the natural projection $C(V) \to \bigwedge^{2g+1} V$ (already used in the  definition of $\beta$ in \S \ref{subsec_Clifford_algebra}). The image of $z$ is
    \[ (-1)^g p_g \wedge p_{2g} \wedge p_0 \wedge p_{2g-1} \wedge  p_1 \wedge \dots \wedge  p_{g+1} \wedge p_{g-1} = (-1)^g p_0 \wedge p_1 \wedge \dots \wedge p_{2g} = (-1)^g \cdot 1 \wedge x \wedge \dots \wedge x^{2g}.  \]
    On the other hand, the computation of the Vandermonde determinant shows that
    the image of $\delta_B^{-1} \epsilon_1 \dots \epsilon_{2g+1}$ is
    \[ \delta_B^{-2} \prod_{i=1}^{2g+1} f'(\omega_i) \cdot 1 \wedge x \wedge \dots \wedge x^{2g} = (-1)^g \cdot 1 \wedge x \wedge \dots \wedge x^{2g} \]
    (as $\delta_B^2 = \Delta(f) = (-1)^{2g(2g+1)/2} \prod_{i=1}^{2g+1} f'(\omega_i)$). It follows that $z = \delta_B^{-1} \epsilon_1 \dots \epsilon_{2g+1}$ and so $M_B = 1$.
    
    Let $I^c = \{ j_1 < \dots < j_{2g+1-r} \}$. We then compute that $M_I M_{I^c}$ is the image of
    \[ \delta_I^{-1} \delta_{I^c}^{-1} \epsilon_{i_1} \dots \epsilon_{i_r} \epsilon_{j_1} \dots \epsilon_{j_{2g+1-r}} = \delta_I^{-1} \delta_{I^c}^{-1}\delta_B \epsilon(\sigma) z, \]
    where $\sigma$ is the permutation putting $i_1, \dots, i_r, j_1, \dots, j_{2g+1-r}$ in increasing order. Note that
    \[ \delta_I^{-1} \delta_{I^c}^{-1}\delta_B  = \prod_{\substack{ i \in I, j \in I^c \\ i < j}}(\omega_j - \omega_i) \prod_{\substack{ i \in I, j \in I^c \\ i > j}}(\omega_i - \omega_j) = \epsilon(\sigma) r_{I, I^c},
    \]
    so we find $M_I M_{I^c} = r_{I, I^c}$, hence (multiplying by $M_I$) that $M_I = M_{I^c}$. Finally, if $|I|$ is even then we have
    \[ N(M_I) = \delta_I^{-2} \epsilon_{i_1} \dots \epsilon_{i_r} \epsilon_{i_r} \dots \epsilon_{i_1} = (-1)^{r(r-1)/2} M_I^2 = (-1)^{|I|/2} r_{I, I^c},\]
     as claimed. 
    
    (3)  Lemma \ref{lem_representing_idempotent_in_J_basis} shows that the entries of $M_{ \{ i \} }$ lie in $\Z[\omega_1, \dots, \omega_{2g+1}]$ (note especially that the entries of the matrix giving the action of one of $p_0, \dots, p_{g-1}$ on $S$, described in \S \ref{subsec_Clifford_algebra}, are in $2 \Z$, so the denominator of $\frac{1}{2}$ goes away).  On the other hand, if $i \neq j$ then $\epsilon_i \cdot \epsilon_j = - \epsilon_j \cdot \epsilon_i$ (using the defining relation of the Clifford algebra, because these vectors are orthogonal). It follows that the entries of the product $\epsilon_{i_1} \cdot \dots \cdot \epsilon_{i_r}$ are skewsymmetric in pairwise permutations of $i_1, \dots, i_r$. This shows that the matrix entries are divisible by $\delta_I$, and therefore that the entries of $M_I$ lie in $\Z[1/2, \omega_1, \dots, \omega_{2g+1}]$. For the second claim, we can assume that $|I| \leq g$. Lemma \ref{lem_representing_idempotent_in_J_basis} gives the formula
    \begin{equation}\label{eqn_image_of_D_i} \epsilon_i = p_{2g} + \omega_i p_{2g-1} + \dots + \omega_i^g p_g + \dots 
    \end{equation}
    By Vandermonde's identity, we have
    \[ \prod_{j=1}^r (p_{2g} + \omega_{i_j} p_{2g-1} + \dots + \omega_{i_j}^{r-1} p_{2g-r+1}) = \delta_I p_{2g} p_{2g-1} \dots p_{2g-r+1}  \]
    (product in $\bigwedge^\ast \langle p_{2g}, \dots, p_{g+1} \rangle \leq C(V)$, and where the terms in the product are ordered with $j$ increasing).  It follows that
    \[ \epsilon_{i_1} \cdot \dots \cdot \epsilon_{i_r} = \delta_I p_{2g} p_{2g-1} \dots p_{2g-r+1} + \dots \]
    where the remaining terms involve only monomials in the basis vectors $p_j$ either of length $\leq r$, or monomials of length $r$ other than $p_{2g} \dots p_{2g-r+1}$. It follows that the image of $1 \in S$ under $M_I$ is
    \begin{equation}\label{eqn_image_of_D_I} M_I \cdot 1 = p_{2g} p_{2g-1} \dots p_{2g-r+1} + \dots, 
    \end{equation}
    where the remaining terms are linear combinations of basis vectors of $S$ following $p_{2g} p_{2g-1} \dots p_{2g-{r+1}}$ in the ordering described in \S \ref{subsec_Clifford_algebra}. The claim follows. 

    (4) We use Lemma \ref{lem_computation_of_beta}. If $|I| = g$, then arguing as in the previous part we find
    \[ \beta(M_I \infty, \infty) = \beta(p_{\{ 2g, \dots, g+1 \}},p_{\emptyset}) =  1. \]
    If $|I| < g$, the result is 0. 
\end{proof}
Proposition \ref{prop_properties_of_heisenberg_matrix} implies in particular that if $[D_I] = [D_J]$, then $M_I = M_J$. We therefore feel free to write $M_T \in \mathcal{G}(2 \Theta)$ for the matrix $M_I$, whenever $T = [D_I]$. Similarly, the resultant $r_{I, I^c}$ depends only on $T$, and we write $r_T = r_{I, I^c}$.
\begin{corollary}\label{cor_properties_of_T_Heisenberg_matrices}
    \begin{enumerate}
        \item  The map $J[2] \to \mathcal{G}(2 \Theta), T \mapsto M_T$, is defined over $k$.
        \item Let $T \in J[2]$. Then:
        \begin{enumerate} \item We have $M_T^2 = r_T$ and $N(M_T) = q(T) r_T$, where $q : J[2] \to \mu_2$ is the quadratic form defined by $q(T) = (-1)^{|I| / 2}$ if $T = [D_I]$ and $|I|$ is even.
        \item The coefficients of $M_T$ lie in $\Z[\omega_1, \dots, \omega_{2g+1}]$.
        \item If $T$ has Mumford degree $g$, then $\beta(M_T \infty, \infty) = 1$; otherwise, $\beta(M_T \infty, \infty) = 0$.
        \end{enumerate}
    \end{enumerate}
\end{corollary}
The quadratic form $q$ is the quadratic refinement of the Weil pairing $e_2 : J[2] \times J[2] \to \mu_2$ associated to the theta characteristic $(g-1) P_\infty$. Part 2(a) of the Corollary thus shows that we can identify the intersection $\Spin(V)\cap \mathcal{G}(2 \Theta) \leq S \Gamma$ with the usual refinement of $\mathcal{G}(2 \Theta)$ to a $\mu_2$-extension of $J[2]$ associated to a quadratic refinement of the Weil pairing.
\begin{proof}
    For the first part, it is enough to show that if $I = \{ i_1 < \dots , i_r\} $, $|I|$ is even, and $\sigma \in S_{2g+1}$, then the image of the element
    \[  \prod_{1 \leq j < k \leq r} (\omega_{\sigma(i_k)} - \omega_{\sigma(i_j)})^{-1}  \epsilon_{\sigma(i_1)} \dots \epsilon_{\sigma(i_r)} \]
    of $C(V)$ in $\End(S)$ equals $M_{\sigma(I)}$. This is true, because if $\tau$ is the permutation of $\sigma(I)$ putting $\sigma(i_1), \dots, \sigma(i_r)$ in increasing order, then we have
    \[  \prod_{1 \leq j < k \leq r} (\omega_{\sigma(i_k)} - \omega_{\sigma(i_j)}) = \epsilon(\tau) \delta_{\sigma(I)}. \]
    The remaining assertions follow immediately from Proposition \ref{prop_properties_of_heisenberg_matrix}. 
\end{proof}

The next lemma will be useful in the proof of Proposition \ref{prop_general_reduction_height_vs_canonical_height}.
Note that we have fixed a basis of $S$ in \S\ref{subsec_Clifford_algebra}, which leads to a set of co-ordinates $x_1, \dots, x_{2^g}$ of $S$.
When $I\subset\{2g, \dots, g+1\}$ is a subset, we also write the coordinate corresponding to $p_I$ as $x_I$. 
\begin{lemma}\label{lemma_coordinate_resultant}
Let $0 \leq m \leq g$, let $M = \{2g, \dots, 2g-m+1\}$, let $T \in J[2](k)$ be a $2$-torsion point of Mumford degree $g-m$, and consider the rational function $s(Q) = x_1(M_T Q)/x_M(Q)$ on $J$. Then for any point $P \in J(k)$ of Mumford degree $m$, we have $s(P) = \mathrm{Res}(U_P, U_T)$. 
\end{lemma}
\begin{proof}
We can assume that $k$ is algebraically closed. Let $\mathrm{Div}^{m,0}\subset J$ be the locally closed subscheme of divisor classes of Mumford degree $m$. 
By Lemma \ref{lemma_Mumford_rep}, each $Q\in \mathrm{Div}^{m,0}$ has an associated Mumford triple $(U_Q, V_Q, R_Q)$, and this association defines an isomorphism between $\mathrm{Div}^{m,0}$ and the variety of triples $(U,V,R)$ satisfying the conditions outlined in that lemma.
Consider the map $\alpha\colon \mathrm{Div}^{m,0}\rightarrow \A^m$, which maps $Q$ to the coefficients $u_1, \dots, u_m$ of $U_Q$, where $U_Q = x^m + u_1x^{m-1} + \cdots + u_m$.
The image of $\alpha$ equals the open subscheme $Y\subset \A^m$ of polynomials $U(x)$ with the property that no $(x-\omega_i)^2$ divides $U$; the complement of this subscheme has codimension $\geq 2$ in $\A^m$, so $\mathcal{O}_Y(Y) = k[u_1, \dots, u_m]$ and the pullback map $\alpha^*\colon k[u_1, \dots, u_m]\rightarrow \cO_{\mathrm{Div}^{m,0}}(\mathrm{Div}^{m,0})$ is injective.

We claim that $s$ restricts to a section of $\cO_{\mathrm{Div}^{m,0}}(\mathrm{Div}^{m,0})$ that is in the image of $\alpha^\ast$.
Indeed, let $Q$ be of Mumford degree $m$, $L = \Psi(Q)$,  $\omega_L = \sum a_I p_I$ and $M_T\infty = \sum b_I p_I$ with $a_I, b_I \in k$.
Proposition \ref{prop_properties_of_explicit_Kummer_morphism} shows that $a_M\neq0$ and $a_I=0$ if $|I|>m$, so we can normalise $\omega_L$ so that $a_M = 1$.
In that normalisation, $x_M(Q)=1$ and we compute that 
\[ s(Q) = x_1(M_TQ) = \beta(M_T\omega_L, \infty) = \pm \beta(\omega_L, M_T\infty). \]
where we have used Lemma \ref{lem_computation_of_beta} and Corollary \ref{cor_properties_of_T_Heisenberg_matrices}. $M_T  \infty$ represents the point $\Sigma(\Psi(T))$, and Equation \eqref{eqn_image_of_D_I} shows that $b_{\{2g, \dots, g+m+1\}}=1$.
Using Lemma \ref{lem_computation_of_beta}, we conclude that the only (possibly) nonzero terms on the right hand side of the identity
$s(Q) = \pm\sum_{I,I'} a_I b_{I'} \beta(p_I, p_{I'})$ are those for which $|I| = m$ and $I' = I^c$.
So to prove the claim it remains to show that $a_I \in k[u_1, \dots, u_m]$ for all subsets $I\subset \{2g, \dots, g+1\}$ of size $m$. 
Note that $L \cap L_{J^c} = 0$, and so in the notation of Proposition \ref{prop_computation_of_pure_spinor_nongeneric} and the paragraph preceding it, there is a basis $\ell_{2g}, \dots, \ell_{g+1}$ of $L$ with $\ell_j = p'_j + \xi_{2g-j} p_g + \sum_{i=0}^{g-1} A_{i ,2g-j} p_i'$ and elements $\xi_{ij} = A_{ij} +\frac{1}{2} \xi_i \xi_j$ for $0\leq i,j\leq g-1$.
The proof of Part  3 of Proposition \ref{prop_properties_of_explicit_Kummer_morphism} shows that $\xi_{ij}\in k[u_1, \dots, u_m]$ if $0\leq i \leq g-1$ and $m\leq j\leq g-1$, and that $\xi_{ij} =0$ if $m\leq i,j\leq g-1$. 
Moreover, $a_I$ is equal, up to a sign and powers of $2$, to the Pfaffian $\xi_{K}$, where $K\subset \{0, \dots, g-1\}$ is the unique subset with the property that $\widehat{K} \Delta J = I$, where $\widehat{K} = \{2g-k\colon k\in K\}$.
The condition that $I$ has size $m$ implies that there is an $0\leq r\leq m$ such that $K = \{i_1< \cdots < i_r < j_1 < \cdots <j_r\}$ with $i_r \leq m-1$ and $j_1\geq m$. 
Since $\xi_{j_m,j_n} = 0$, the only nonzero terms of $\xi_K$ are (up to sign) of the form $\xi_{i_1,j_{\sigma(1)}}\cdots \xi_{i_r, j_{\sigma(r)}}$, where $\sigma\in S_r$ is a permutation. Since each term in the latter product is an element of $k[u_1, \dots, u_m]$, $\xi_K$ and hence $a_I$ is an element of $k[u_1, \dots, u_m]$, proving the claim.

Therefore, there exists a unique polynomial $G(U) = G(u_1, \dots, u_m) \in k[u_1, \dots, u_m]$ with the property that for all $Q$ of Mumford degree $m$, $G(U_Q) = s(Q)$. 
For such a $Q$, we know that $s(Q)=0$ if and only if the $x_1$-coordinate of $(\Sigma\circ\Psi)(Q + T)$ is zero, if and only if (by Proposition \ref{prop_properties_of_explicit_Kummer_morphism}) $Q+T$ has Mumford degree $<g$, if and only if $U_Q$ has a root in common with $U_T$, if and only if $U_Q(\omega_i)=0$ for $i \in I$, where $I \subset \{ 1, \dots, 2g+1 \}$ is the subset of size $g-m$ such that $U_T(x) = \prod_{i \in I}(x - \omega_i)$. Therefore, the polynomials $G(U)$ and $\mathrm{Res}(U_T, U) = \prod_{i\in I} U(\omega_i)$ (seen as elements of $\cO_Y(Y) = k[u_1, \dots, u_m]$) have the same vanishing locus. Since the latter polynomial is a product of linear factors, we must have $G(U) = \lambda \prod_{i\in I} U(\omega_i)^{n_i}$ for some $\lambda \in k^{\times}$ and for some integers $n_i\geq 1$. 

To complete the proof, we need to show that $\lambda = (-1)^{m(g-m)}$ and $n_i = 1$ for each $i \in I$. The rational function $s$ is a linear combination of functions $x_I / x_M$, with coefficients given by universal polynomials in the roots $\omega_i$ ($i \in I$), so we may assume that the roots $\omega_1, \dots, \omega_{2g+1}$ are in fact algebraically independent. 
Let $I'\subset \{1, \dots, 2g+1\}$ be any subset of size $m$ disjoint from $I$, and let $T' = [D_{I'}]$. Taking $Q = T'$, we compute using Corollary \ref{cor_properties_of_T_Heisenberg_matrices} that
\[ G(T') = \beta(M_T M_T' \infty, \infty) = r_{I', I} \beta(M_{T + T'} \infty, \infty) = r_{I' , I} = \mathrm{Res}(U_{T'}, U_T). \]
Therefore $\lambda = 1$ and $n_i = 1$ for each $i \in I$. This completes the proof. 
\end{proof}

\subsection{Proof of Proposition \ref{prop_existence_of_independent_torsion_points}}\label{subsec_generic_spin_basis}
The proof of Proposition \ref{prop_existence_of_independent_torsion_points} is given after the following purely linear algebraic lemma.
\begin{lemma}\label{lem_exterior_power_uniform_matroid}
    Let $a,b,n,r$ be integers satisfying $n\geq 1$ and $a,b\geq n$ and $1\leq r \leq n$.
    Let $J$ be a totally ordered finite set with a decomposition $J  = J_1\sqcup J_2$ such that $|J_1| = a$ and $|J_2| = b$.
    Then there exists a subset $\mathcal{F}\subset \{I\subset J \colon |I| = r\}$ with $|\mathcal{F}| = \binom{n}{r}$ with the following properties:
    \begin{enumerate}
        \item For every $I\in \mathcal{F}$, $|I\Delta J_1| \in \{a,a+1\}$. (Equivalently, $|I \cap J_1| = \lfloor r/2 \rfloor$.)
        \item For every $n$-dimensional vector space $W$ over a field and set of elements $\{w_i\}_{i\in J}$ in $W$ such that every $n$-element subset of $\{w_i\}_{i \in I}$ forms a basis, the elements $\{w_I\}_{I\in \mathcal{F}}$ form a basis of $\wedge^r W$, where for $I = \{i_1< \cdots < i_r\}$ we write $w_I = w_{i_1} \wedge \cdots \wedge w_{i_r}$.
    \end{enumerate}
\end{lemma}

\begin{proof}
    We use induction on $n$.
    If $n=1$, then $r=1$ and we may take $\mathcal{F} = \{\{j\}\}$, where $j$ is the minimal element of $J_2$.
    Suppose $n\geq 2$. 
    If $r=1$, then we may take $\mathcal{F} = \{\{j_1\}, \dots, \{j_n\}\}$ where $\{j_1< \cdots <  j_n\}$ are the $n$ minimal elements of $J_2$, so assume $r\geq 2$.
    If $r$ is even, let $j$ be the minimal element of $J_1$; if $r$ is odd, let $j$ be the minimal element of $J_2$.
    Let $J' = J\setminus \{j\}$, which induces a corresponding decomposition $J ' = J_1' \sqcup J_2'$. 
    Let $n' = |J'| = n-1$, $a' = |J_1'|$ and $b' = |J_2'|$.
    By the induction hypothesis, there exists a subset $\mathcal{F}_1' \subset \{I' \subset J' \colon |I'| = r-1\}$ satisfying the two properties of the lemma with respect to the data $(J_1',J_2',n',r-1)$.
    Let $\mathcal{F}_1 = \{ I' \cup \{j\} \colon I' \in \mathcal{F}'_1\}$.
    If $r=n$, let $\mathcal{F} = \mathcal{F}_1$.
    If $r\leq n-1$, let $\mathcal{F}_2\subset \{I' \subset J' \colon |I'| = r\}$ be a subset satisfying the two properties of the lemma with respect to the data $(a',b',n',r)$ (again such an $\mathcal{F}_2$ exists by the induction hypothesis) and let $\mathcal{F} = \mathcal{F}_1\sqcup \mathcal{F}_2$.

    We claim that $\mathcal{F}$ satisfies the two properties of the lemma. 
    The first property follows from the choice of $j$ depending on the parity of $r$.
    To check the second one, let $W$ be a vector space over a field and $\{w_i\}_{i \in J}$ a subset of $W$ such that each $n$-element subset forms a basis.
    Let $w = w_j$, which is in particular nonzero.
    The map $\wedge^{r-1}W \rightarrow \wedge^r W, x\mapsto x\wedge w$ factors through a map $\wedge^{r-1}(W/\langle w\rangle)\rightarrow \wedge^r W$ which fits into an exact sequence
    \[
    0 \rightarrow \bigwedge^{r-1}(W/\langle w\rangle) \xrightarrow{\alpha} \bigwedge^r W \xrightarrow{\beta} \bigwedge^r(W/\langle w \rangle ) \rightarrow 0,
    \]
    where $\beta$ is induced by the projection $W\rightarrow W/\langle w\rangle$. 
    Every $(n-1)$-element subset of $\{w_i\}_{i\in J'}$ projects to a basis of $W/\langle w\rangle$.
    The constructions of $\mathcal{F}_1, \mathcal{F}_2$ therefore imply that $\{w_I\}_{I\in \mathcal{F}_1}$ is a basis for the image of $\alpha$ and $\{w_I\}_{I \in \mathcal{F}_2}$ maps to a basis of $\wedge^r(W/\langle w\rangle)$ under $\beta$.
\end{proof}

\begin{proof}[Proof of Proposition \ref{prop_existence_of_independent_torsion_points}]
    Let $J_1 = \{1, \dots, g\}$, $J_2 = \{g+1, \dots, 2g+1\}$ and $J = \{1, \dots, 2g+1\}$.
    Lemma \ref{lem_exterior_power_uniform_matroid} guarantees the existence, for each $1\leq r\leq g$, of a subset $\mathcal{F}_r\subset \{I\subset J\colon |I| = r\}$ of size $\binom{g}{r}$ satisfying the two properties of the lemma. 
    Let $\mathcal{F}_0 = \{\emptyset\}$ and let $\mathcal{F} = \bigcup_{r\geq 0} \mathcal{F}_r$.
    We then claim that $\{[D_{I\Delta J_1}] \colon I\in \mathcal{F}\}$ is a generic spin basis for $C$.
    By construction $|I\Delta J_1| = g$ or $g+1$, so each point $D_{I\Delta J_1}$ has Mumford degree $g$.
    To prove that they are independent under the Kummer embedding $\Psi$, it suffices to prove that $\{D_I \colon I\in \mathcal{F}\}$ are independent under the Kummer embedding. 

    Consider the increasing filtration $\operatorname{Fil}^\bullet$ of $S = \bigwedge^\ast E$ given by $\operatorname{Fil}^r = \oplus_{0 \leq i \leq r} \bigwedge^i E$.
    If $|I|=r$ then $(\Sigma \circ \Psi)([D_I]) \in \P(\mathrm{Fil}^r)$ (see Equation (\ref{eqn_image_of_D_I})).
    By Equation (\ref{eqn_image_of_D_i}), the image of $(\Sigma \circ\Psi)([D_{\{i\}}])$ in $\P(\mathrm{Fil}^1/\mathrm{Fil}^0) = \P(E)$ equals $[v_i] = [1 : \omega_i : \dots : \omega_i^{g-1}]$ (in the standard basis $p_{2g}, \dots, p_{g+1}$ of $E$). 
    If $I = \{i_1 < \dots < i_r \}$, then we can similarly compute that the projection of $(\Sigma \circ \Psi)([D_I]) \in \P(\mathrm{Fil}^r)$ to $\P(\mathrm{Fil}^r/\mathrm{Fil}^{r-1}) =\P( \bigwedge^r E)$ may be identified with $[v_I] = [v_{i_1} \wedge \dots \wedge v_{i_r}]$.  By construction of $\mathcal{F}_r$, the set $\{v_I\}_{I\in \mathcal{F}_r}$ forms a basis of $\bigwedge^r E$.
    It follows that the points $\{(\Sigma \circ\Psi)(D_I) \colon I\in \mathcal{F}\}$ span $\P(\bigwedge^* E) = \P(S)$; this completes the proof. 
\end{proof}
    The proof of Proposition \ref{prop_existence_of_independent_torsion_points} gives an explicit inductive procedure to construct a generic spin basis for every $g$.
    For example, we exhibit a set $\mathcal{S}$ such that $\{D_I \colon I \in \mathcal{S}\}$ is a generic spin basis for small $g$:
    \begin{itemize}
        \item $g=1$: $\mathcal{S} = \{\{1\},\{1,2\}\}$.
        \item $g=2$: $\mathcal{S} = \{\{1,2\},\{1,2,3\}, \{1,2,4\},\{2,3\}\} $.
        \item $g=3:$ 
        $\mathcal{S} = 
        \{
        \{1,2,3\},
        \{1,2,3,4\}, \{1,2,3,5\}, \{1,2,3,6\},
        \{2,3,4\}, \{2,3,5\}, \{1,3,4\},
        \{2,3,4,5\}
        \}.$
        \item $g = 4:$ $\mathcal{S} = $
        \begin{multline*} \{ \{1,2,3,4\}, \{1,2,3,4,5\},\{1,2,3,4,6\},\{1,2,3,4,7\},\{1,2,3,4,8\},\\ \{2,3,4,5\}, \{2,3,4,6\}, \{2,3,4,7\},\{1,3,4,5\},\{1,3,4,6\},\{1,2,4,5\},\\ \{2,3,4,5,6\},\{2,3,4,5,7\},\{1,3,4,5,6\},\{2,3,4,6,7\},\{3,4,5,6\}\}. 
        \end{multline*}
    \end{itemize}
We now give an explicit example. Take $g = 4$, $k = \F_5$, and $f(x) = x^9 + 2 x^3 + x + 3$. In this case, the Jacobian admits the point $P \in J(\F_5)$ corresponding to 
\[ (U, V, R) = ( x^4 + 4+x^3 + x^2 + 2+x + 3, x^5+x^4+3x^2+2x+4, 3 x^3 + x^2 + x + 1). \]
Using the generic formula
\begin{equation}\label{eqn_generic_psi_in_genus_4} \Psi = [1: -u_1: u_2: -u_3: u_4: -u_3 - u_2 v_1 - v_3: 
u_4 + u_3 v_1 + v_4: -u_4 v_1 - v_5: \dots,
 \end{equation}
 we
 find 
 \[ \Psi(P) = [ 1: 1: 1: 3: 3: 4: 2: 3: 4: 2: 0: 4: 1: 2: 1: 3 ]. \]
The polynomial $f(x)$ is irreducible, and splits over $\F_{5^9}$. Using the explicit generic spin basis constructed above, we can compute polynomials $\delta'_1, \dots, \delta'_{16}(x) \in \F_{5^9}[x_1, \dots, x_{16}]$ satisfying the conclusion of Lemma \ref{lem_from_spin_basis_to_duplication_polynomial}; in particular, these polynomials representation duplication. If we define $\delta_i = \frac{1}{9} \tr_{\F_{5^9} / \F_5}(\delta_i) \in \F_5[x]$, then the resulting polynomials will also represent duplication, and furthermore be defined over $k = \F_5$.

The resulting polynomials have about 3000 nonzero monomials each, and are too large to be reproduced here. The first one is
\begin{multline*} \delta_1 = 4 x_1^3 x_2+2 x_1^2 x_2^2+x_1 x_2^3+3 x_2^4+3 x_1^3 x_3+4 x_1 x_2^2 x_3\\+x_2^3 x_3+3 x_1^2 x_3^2+3 x_2^2 x_3^2+3 x_1 x_3^3+2 x_2 x_3^3+3 x_3^4+3 x_1^3 x_4+2 x_1^2 x_2 x_4+2 x_1 x_2^2 x_4+x_2^3 x_4+ \dots 
\end{multline*} 
Repeatedly evaluating $\delta$ on $\Psi(P)$, we find
\[ \Psi([2](P)) = [0:1:2:2:3:2:1:1:3:1:4:0:4:1:4:3], \]
\[ \Psi([4](P)) = [1:0:0:0:2:1:1:3:4:4:4:0:2:2:3:2], \]
\[ \Psi([8](P)) = [0:1:1:0:1:1:4:1:4:0:1:4:2:3:2:4],\]
and
\[ \Psi([16](P)) = [1:3:1:2:1:2:4:3:1:2:2:3:0:4:4:4]. \]
As a sanity check, we may compute using the implementation of Cantor's algorithm in Magma \cite{Mag97} that $[16](P) = Q$ corresponds to the Mumford triple 
\[ (x^4 + 2 x^3 + x^2 + 3x + 1, x^5+3x^4+3x^3+2x^2+4x+4, 4 x^3 + 3 x^2 + 2). \]
Computing the image of $Q$ in $\P^{15}(\F_5)$ using the formula of Equation (\ref{eqn_generic_psi_in_genus_4}) gives the same same answer, as expected. 

\section{Applications to heights}\label{sec_applications_to_heights}

Let $k$ be a number field. We write $M_k$ for the set of absolute values on $k$ extending one of the standard ($p$-adic or archimedean) absolute values on $\Q$. This choice of system of absolute values satisfies the product formula ($x \in k^\times$):
\[ \prod_{v \in M_k} |x|_v^{d_v} = 1, \]
where $d_v = [k_v : \Q_v]$. 
We write $M_k^\infty$ (resp. $M_{k, \infty}$) for the subset of non-archimedean (resp. archimedean) absolute values.

If $x = [x_0 : x_1 : \dots : x_n] \in \P^n(k)$, then we define its (absolute, logarithmic) height
\[ h_k(x) = [k : \Q]^{-1} \sum_{v \in M_k} d_v \log \max_i |x_i|_v. \]
If $l / k$ is a finite extension, then $h_k(x) = h_l(x)$; we henceforth write simply $h(x) = h_k(x)$. 

Now let $f(x) = x^{2g+1} + c_1 x^{2g} + \dots + c_{2g+1} \in \cO_k[x]$ be a polynomial of non-zero discriminant $\Delta(f) \in \cO_k$, with associated curve $C$ and Jacobian $J$, defined over $k$. We now have access to the following four well-defined height functions on $J(k)$:
\begin{itemize}
    \item The dagger height $h^\dagger$, defined by $h^\dagger(P) = h( 1 : u_1 : \dots : u_m )$ if $P$ is a divisor class corresponding to the Mumford triple $(U, V, R)$ of Mumford degree $m$, $U(x) = x^m + u_1 x^{m-1} + \dots + u_m$. (This height has been studied in \cite[\S6]{holmes-arakelovapproachnaiveheighthyperelliptic}.)
    \item The reduction height $\widetilde{h}$, defined by
    \[ \widetilde{h}(P) = [k : \Q]^{-1} \left(  \sum_{v \in M_k^\infty} \frac{d_v}{2} \log \max (1, |u_1|_v, \dots, |u_m|_v) + \sum_{v \in M_{k, \infty}} d_v \widetilde{h}_v(P) \right). \]
    Here if $v$ is an archimedean place of $k$ corresponding to an embedding $\sigma : k \to \bC$, and $\omega_1, \dots, \omega_{2g+1}$ are the roots of $\sigma(f)$ in $\bC$, then we define 
    \[ \widetilde{h}_v(P) = \frac{1}{2} \log \sum_{i=1}^{2g+1} \frac{|U(\omega_i)|}{|f'(\omega_i)|}. \]
    We observe that, like the naive height, the reduction height is compatible with passage to a finite extension $l / k$.  (This height has been studied in \cite{lagathorne2024smallheightoddhyperelliptic}.)
    \item The naive height $h$, induced by the morphism $J \to K \to \P(S) = \P^{2^g-1}$ made explicit in \S\ref{sec_hyperelliptic_curves_and_kummers}, together with the standard naive height on $\P^{2^g-1}$.
    \item The canonical height $\widehat{h}_\Theta$ associated to the divisor $\Theta \subset J$. Following Tate, this may be defined by the formula
    \[ \widehat{h}(P) =  \frac{1}{2}  \lim_{n \to \infty} \frac{1}{4^n}  h([2^n] P). \]
\end{itemize}
Our goal now is to find relations between these heights. We first relate $h^\dagger$ and $h$:
\begin{theorem}\label{thm_naive_height_lower_bound}
    For any $P \in J(k)$, we have $h^\dagger(P) \leq h(P)$. 
\end{theorem}
\begin{proof}
    This follows immediately from Proposition \ref{prop_properties_of_explicit_Kummer_morphism}.
\end{proof}

\subsection{Relating $h$ and $\widehat{h}$}

We now follow the method of Stoll \cite{Sto99} to relate $h$ and $\widehat{h}$. Let $\widetilde{K}$ denote the pre-image of $K$ in $S - \{ 0 \}$, and $\pi : \widetilde{K} \to K$ the canonical morphism. Our choice of basis for $S$ (see \S \ref{subsec_Clifford_algebra}) determines an isomorphism $S \cong \A^{2^g}$, and we will write $x_1, \dots, x_{2^g}$ for the corresponding co-ordinates on $S$; when convenient, we will also write $x_I$ for these co-ordinates, as $I$ ranges over subsets of $\{ 2g, \dots, g+1 \}$. In particular, we have $x_1 = x_{\{ 2g, \dots, g+1 \}}$ and $x_{2^g} = x_\emptyset$.

For $v \in M_k$, let $h_v : \widetilde{K}(k_v) \to \R_{\geq 0}$ be defined by $h_v(x) = \log \max_i |x_i|_v$. Let $\delta : S^\vee \to \Sym^4 S^\vee$ represent duplication, in the sense of Definition \ref{def_represents_duplication}; using our co-ordinates, we can write $\delta = (\delta_1, \dots, \delta_{2^g})$, where each $\delta_i \in k[x_1, \dots, x_{2^g}]$ is homogeneous of degree $4$. We recall that $\delta(\infty) = \infty$, where $\infty = (0, 0, \dots, 0, 1) \in \widetilde{K}(k)$ is the vector representing the image of the identity element of $J(k)$. 

If $v \in M_k$, then we define functions $\varepsilon_v, \mu_v : K(k_v) \to \R$ by the formulae
\[ \varepsilon_v(P) =  h_v(\delta(x)) - 4 h_v(x), \]
\[ \mu_v(P) = \sum_{n=0}^\infty 4^{-(n+1)} \varepsilon_v([2^n](P)), \]
where $x \in \widetilde{K}(k_v)$ is any pre-image of $P$. Note that these functions are independent of the choice of $\delta$. 
\begin{proposition}\label{prop_local_height_difference}
    \begin{enumerate} \item The functions $\varepsilon_v, \mu_v : K(k_v) \to \R$ are well-defined, continuous, bounded, and satisfy $\varepsilon_v(\infty) = \mu_v(\infty) = 0$.
    \item The function $\lambda_{2 \Theta, v} : (J - 2 \Theta)(k_v) \to \R$ defined by the  formula $\lambda_{2 \Theta, v}(P) = \log \max_i( |x_i| / |x_1| ) + \mu_v(\Psi(P))$ is a N\'eron function for the divisor $2 \Theta$ in the sense of \cite[Ch. 11, Theorem 1.1]{Lan83}.
    \end{enumerate} 
\end{proposition}
\begin{proof}
    (1) It is clear that $\varepsilon_v$ is a well-defined continuous function, satisfying $\varepsilon(\infty)  = 0$; it is bounded since $K(k_v)$ is compact. The analogous result for $\mu_v$ follows by uniform convergence.

    (2) The function $\delta_1(x) / x_1^4$ on $J$ has divisor $[2]^\ast 2 \Theta - 8 \Theta$ by Proposition \ref{prop_firstpropertieskummerembedding}. By construction, $\lambda_{2 \Theta, v}$ satisfies the formula $\lambda_{2 \Theta, v}([2](P)) = 4 \lambda_{2 \Theta, v}(P) - \log |(\delta_1(x) / x_1^4)(P)|_v$. The characterisation of \cite[Ch. 11, Proposition 1.4]{Lan83} shows that $\lambda_{2 \Theta, v}$ is a N\'eron function. 
\end{proof}
N\'eron functions for the divisor $2 \Theta$ are determined up to additive constant. In some cases, we may know a priori that the function $\log \max_i( |x_i|_v / |x_1|_v )$ is also a N\'eron function for the divisor $2 \Theta$. In this case, it follows that $\mu_v(P)$ is constant, and therefore (by the first part of Proposition \ref{prop_local_height_difference}) that $\mu_v \equiv 0$. 
The next result describes such a situation. 
\begin{theorem}\label{theorem_weierstrassmodelregular_muv_zero}
    Let $v$ be a finite place of $k$ of odd residue characteristic.
    Let $\mathcal{C}\rightarrow \Spec(\O_{k_v})$ be the Weierstrass model of $C$ over $\O_{k_v}$, given by the same equation $y^2 = x^{2g+1}+ c_1x^{2g} +\cdots c_{2g+1}$.
    Suppose that $\mathcal{C}$ is a regular scheme. 
    Then $\log \max_i(|x_i|_v/|x_1|_v)$ is a N\'eron function for $2\Theta$.
    Consequently, $\mu_v\equiv 0$. 
\end{theorem}
\begin{proof}
    Write the base change of $C$ and $J$ to $k_v$ by the same letter.
    Write $S = \Spec(\O_{k_v})$ and let $\eta, s\in S$ denote the generic and closed point, respectively.
    Let $\sh{J}\rightarrow S$ be the N\'eron model of $J$.
    Since $\mathcal{C}$ is regular by assumption and the fibers of $\mathcal{C}\rightarrow S$ are geometrically integral, a theorem of Raynaud \cite[\S9.5, Theorem 1]{BLR-neronmodels} shows that there is an isomorphism of $S$-schemes $\Pic_{\mathcal{C}/S}^0 \simeq \sh{J}$ extending the identification $\Pic_{C/k_v}^0  = J$ on the generic fiber.
    In particular, the special fiber $\sh{J}_s$ is connected.

    Consider a smooth and separated morphism $\mathcal{X}\rightarrow S$, effective Cartier divisor $D$ on $\mathcal{X}$ and element $P \in \mathcal{X}_{\eta}(k_v)$ that extends to a section $\mathcal{P}$ of $\mathcal{X}\rightarrow S$ not contained in the support of $D$.
    Then the pullback $M = \mathcal{P}^*\O_{\mathcal{X}}(D)$ is an invertible $\O_{k_v}$-module and the element $m = \mathcal{P}^*1 \in M$ (where $1\in \HH^0(\mathcal{X}, \O_{\mathcal{X}}(D))$ is the tautological section) is nonzero.
    The intersection multiplicity $i_{\mathcal{X}}(D,P)\in \Z_{\geq 0}$ is defined to be the length of the $\O_{k_v}$-module $M/(\O_{k_v}\cdot m)$; it is a nonnegative integer.

    Let $\bar{\Theta}\subset \sh{J}$ be the closure of $\Theta$. 
    Since every element $P\in J(k_v)$ extends to a section $\mathcal{P} \in \sh{J}(S)$ by the N\'eron mapping property, the integer $i_{\sh{J}}(2\bar{\Theta},P)$ is defined for all $P\in (J-\Theta)(k_v)$.
    Let $\pi\in \O_{k_v}$ be a uniformiser. 
    Since $\sh{J}_s$ is connected, it follows from \cite[Ch. 11, Theorem 5.1]{Lan83} that the association $P\mapsto -i_{\sh{J}}(2\bar{\Theta}, P)\log |\pi|_v$ (for $P\in (J-\Theta)(k_v)$) is a N\'eron function for $2\Theta$.
    To prove the proposition, it therefore suffices to prove that $-i(2\bar{\Theta},P)\log |\pi|_v - \log\max_i(|x_i|_v/|x_1|_v)$ is a constant function, where we write $\Psi(P) = [x_1 : \cdots : x_{2^g}]$.

    Proposition \ref{proposition_universal_kummerembedding} (together with the remarks preceding Equation \eqref{eq_universal_kummer_embedding}) shows that $\Psi\colon J \rightarrow \P^{2^g-1}_{k_v}$ extends to a morphism of $S$-schemes $\sh{J} \rightarrow \P^{2^g-1}_S$
    that we continue to denote by $\Psi$. 
    Let $\mathcal{H}\subset \P^{2^g-1}$ denote the Cartier divisor defined by $(x_1 = 0)$.
    A calculation shows that for all $Q = [x_1: \cdots: x_{2^g-1}]\in (\P^{2^g-1}-\mathcal{H}_{\eta})(k_v)$, we have $\log \max_i(|x_i|_v/|x_1|_v) = - i_{\P^{2^g-1}}(\mathcal{H},Q) \log |\pi|_v$.

    It therefore remains to show that the function $P\mapsto i_{\sh{J}}(2\bar{\Theta},P) - i_{\P^{2^g-1}}(\mathcal{H},\Psi(P))$ (for $P\in (J-\Theta)(k_v)$) is constant.
    By Part 1 of Proposition \ref{prop_firstpropertieskummerembedding}, we have an equality $\Psi_{\eta}^*(\mathcal{H}_{\eta}) = 2\Theta$ of Cartier divisors on $\sh{J}_{\eta}$. 
    Since the special fiber $\sh{J}_s$ is connected and smooth, hence irreducible, we must have $\Psi^*\mathcal{H} = 2\bar{\Theta} + n  \sh{J}_s$ for some $n\in \Z_{\geq 0}$.
    The definition of intersection multiplicity then immediately shows that for $P\in (J-\Theta)(k_v)$ we have
    \[
    i_{\P^{2^g-1}}(\mathcal{H}, \Psi(P))
    = i_{\sh{J}}(\Psi^*\mathcal{H}, P) 
    = i_{\sh{J}}(2\bar{\Theta},P) + i_{\sh{J}}(n \sh{J}_s, P) = i_{\sh{J}}(2\bar{\Theta},P ) + n,
    \]
    as claimed.
    (In fact, we have $n=0$, but we do not need this here.)
\end{proof}
Writing $\widehat{h}_\Theta$ as a telescoping series gives
\begin{multline*} 2 \widehat{h}_\Theta(P) = h(P) + \sum_{n=0}^\infty 4^{-(n+1)} (h([2^{n+1}](P) - 4 h([2^n](P))) \\ = h(P) + [k : \Q]^{-1} \sum_{n=0}^\infty \sum_{v \in M_k} d_v 4^{-(n+1)} \varepsilon_v([2^n](P)) = h(P) + [k : \Q]^{-1} \sum_{v \in M_k} d_v \mu_v(P). 
\end{multline*}
If we can find a constant $b_v \in \R$ such that $\varepsilon_v(P) \geq b_v$ for all $P \in K(k_v)$, then summing the geometric series shows that  $\mu_v(P) \geq \frac{1}{3} b_v$. This observation forms the basis for the proof of the following result (cf. \cite[Theorem 6.1]{Sto99}).
\begin{theorem}\label{thm_canonical_height_difference}
    \begin{enumerate} \item Let $v$ be a finite place of $k$, and let $P \in J(k_v)$. Then: 
    \begin{enumerate}
        \item If $\ord_v(\Delta) \leq 1$ and $v$ has odd residue characteristic, then $\mu_v \equiv 0$.
        \item Without any assumptions on $v$ or $\ord_v(\Delta)$, we have $\mu_v(P) \geq  \frac{1}{3} \log |2^{4 g} \Delta|_v$. 
    \end{enumerate}
    \item Let $v$ be an infinite place of $k$, corresponding to an embedding $\sigma : k \to \bC$, and let $P \in J(k_v)$. Then there is a constant $c = c(g) \in \R$ (independent of $f$ and $k$) such that $\mu_v(P) \geq \frac{1}{3} \log |\Delta|_v - \frac{1}{6} g(7g+5) \log \height(\sigma(f)) + c$.
    \end{enumerate} 
\end{theorem}
\begin{proof}
    (1) Part (a) follows from Theorem \ref{theorem_weierstrassmodelregular_muv_zero}, once we verify that the assumptions imply that the Weierstrass model $\mathcal{C}/\O_{k_v}$ is regular. 
    Since the double cover $\mathcal{C} \rightarrow \P^1_{\O_{k_v}}$ is \'etale outside the ramification locus $\mathcal{W}\sqcup \{P_{\infty}\}$ and since $\mathcal{C}\rightarrow \Spec(\O_{k,v})$ is smooth in a neighbourhood of the section $P_{\infty}$, it suffices to prove that every point of $\mathcal{W}$ is a regular point of $\mathcal{C}$.
    Since $\mathcal{W}$ is cut out by the regular element $y$, it suffices to prove that $\mathcal{W}$ is itself regular.
    But $\mathcal{W} = \Spec\O_{k_v}[x]/(f(x))$ and since $\ord_v(\Delta)\leq 1$, $\O_{k_v}[x]/(f(x))$ is a maximal order in the \'etale $k_v$-algebra  $k_v[x]/(f(x))$.
    It follows that $\O(\mathcal{W})$ is regular, hence $\mathcal{W}$ is regular, hence $\mathcal{C}$ is regular, as claimed. (This argument shows that the condition $\ord_v(\Delta) \leq 1$ could be weakened to the condition that $\cO_{k_v} / (f(x))$ is maximal order.)

    We now consider part (b). In this case, we will show that $\varepsilon_v(P) \geq \log | 2^{4g} \Delta |_v$ for all $P\in  J(k_v)$. We are free to enlarge $k$, so we can assume without loss of generality that all roots of $f$ lie in $k$. We consider the short exact sequence of $J[2]$-modules (cf. Proposition \ref{prop_q_tildes_inject_to_Kummer})
    \[ 0 \to I_K \otimes N \to (\Sym^2 S^\vee) \otimes N \to\HH^0(J, \O_J(2\Theta)^{\otimes 2}) \otimes N. \]
    According to Proposition \ref{prop_q_tildes_inject_to_Kummer}, there is an isomorphism of $J[2]$-modules 
    \[ (\Sym^2 S^\vee) \otimes N \cong \oplus_S k(e_2(S, -)), \]
    where the sum runs over the set of all $S \in J[2]$ such that $(g-1) P_\infty + S$ is an even theta-characteristic. Given such an $S$, let $\pi_S$ denote the projection to the $e_2(S, -)$-eigenspace; thus we have the formula
    \[ \pi_S = \frac{1}{2^{2g}} \sum_{T \in J[2]} e_2(S, T) N(M_T) M_T^{-1} = \frac{1}{2^{2g}} \sum_{T \in J[2]} e_2(S, T) N(M_T)^{-1} M_T = \frac{1}{2^{2g}} \sum_{T \in J[2]} e_2(S, T) (\pm r_{T})^{-1} M_T, \]
    where we have used Corollary \ref{cor_properties_of_T_Heisenberg_matrices}. Proposition \ref{prop_q_tildes_inject_to_Kummer} shows that there are constants $\lambda_{S, j} \in k$ such that if $S$ has Mumford degree $g$, we have
    \[ \pi_S(x_j^2) = \lambda_{S, j} q_{M_S}. \]
    Evaluating both sides at the point $\infty$, we find
    \[ \lambda_{S, j} =\frac{1}{2^{2g}} \sum_{T \in J[2]} e_2(S, T) (\pm r_{T}^{-1}) (M_T \infty)_j^2, \]
    hence $|\lambda_{S, j}|_v \leq | 2^{4g} \Delta |_v^{-1/2}$. 
    
    On the other hand, Proposition \ref{prop_Riemann_theta_relation_as_incidence_relation} shows that if $z \in \widetilde{K}(k_v)$, $w = \delta(z)$, then
    \[ \beta(M_S \infty, w) = q_{M_S}(z)^2. \]
    If $S$ is non-generic, then $\pi_S(x_j^2)$ lies in $I_K$, so vanishes at $z$. We find that
    \[ z_j^2 = \sum_{S \text{ generic}} \pi_S(x_j^2)|_{x = z} = \sum_{S \text{ generic}} \lambda_{S, j} q_{M_S}(z), \]
    hence 
    \[ \max_j |z_j|^2 \leq | 2^{4g} \Delta |_v^{-1/2} \max_S |q_{M_S}(z)|_v, \]
    hence
    \[ \max_j |z_j|^4 \leq |2^{4g} \Delta|_v^{-1} \max_S |\beta( M_S \infty,w)|_v \leq |2^{4g} \Delta|_v^{-1} \max_j |w_j|_v. \]
    Taking logarithms and re-arranging, this says exactly that, if $P \in J(k_v)$ is a point mapping to $[z] \in K(k_v)$, then
    \[ \varepsilon_v(P) \geq \log |2^{4g} \Delta|_v. \]
    This completes the proof of part (b).

    (2) When $v$ is an infinite place, we can argue in a similar way to Part 1(b), using the usual triangle inequality in lieu of the non-archimedean one.
    Identifying each term with its image under $\sigma$, and therefore dropping the subscript $v$ on each absolute value, we obtain
    \[ | \lambda_{S, j} | \ll \max_{T \in J[2]} \max_j | (M_T \infty)_j |^2 |\delta_T \delta_{T^c}| |\Delta|^{-1/2}.  \]
    Analysis of the behaviour under scaling shows that $| (M_T \infty)_j | \ll \height(\sigma(f))^{g(g+1)/2}$. Since $|\delta_T| \ll \height(\sigma(f))^{g(g-1)/2}$, we find 
    \[ | \lambda_{S, j} | \ll \height(\sigma(f))^{g(3g+2)/2} |\Delta|^{-1/2}. \]
    It follows that
    \[ \max_j | z_j |^4 \ll |\Delta|^{-1} \height(\sigma(f))^{g(3g+2) + g(g+1)/2} \max_j | w_j |. \]
    Re-arranging now gives
    \[ \varepsilon_v(P) \geq \log | \Delta | - \frac{1}{2} g(7g+5) \log \height(\sigma(f))  + O_g(1), \]
    as required. 
\end{proof}

\subsection{Relating $\widetilde{h}$ and $\widehat{h}$}

In order to understand the relation between the reduction height and the canonical height, we first consider more carefully the contribution to the canonical height of the infinite places. We introduce some notation. Let $\mathfrak{S}_g$ denote the usual Siegel upper half-space. If $\tau \in \mathfrak{S}_g$, then $\bC^g / (\Z^g \oplus \tau \Z^g)$ is a polarisable complex torus. If $z \in \bC^g$, $\tau \in \mathfrak{S}_g$, then we write 
\[ \theta(z, \tau) = \sum_{n \in \Z^g} \exp(\pi i {}^t n \tau n + 2 \pi i {}^t n z) \]
for the standard theta function. If $\eta = (\eta', \eta') \in \Q^g \times \Q^g$, then we define the theta function with characteristic
\[ \theta[\eta](z, \tau) = \exp(\pi i {}^t \eta' \tau \eta' + 2 \pi i {}^t \eta' (z + \eta'')) \theta(z + \tau \eta' + \eta'', \tau). \]
\begin{lemma} Let $\eta \in \Q^g \times \Q^g$, $z = x + i y \in \bC^g$, $\tau = a + i b \in \mathfrak{S}_g$.
    \begin{enumerate}
        \item The function $|\theta[\eta](z, \tau)|$ depends only on the image of $\eta$ in $\Q^g / \Z^g \times \Q^g / \Z^g$.
        \item The function $\exp(- \pi {}^t y b^{-1} y) |\theta[\eta](z, \tau)|$ depends only on the image of $z$ in $\bC^g / (\Z^g \oplus \tau \Z^g)$.
    \end{enumerate}
\end{lemma}
\begin{proof}
    These formulae are immediate from the quasi-periodicity of the $\theta[\eta]$ (see \cite[p. 123]{MumfordTataI}). 
\end{proof}
We define 
\[ \| \theta \|(z, \tau) = (\det b)^{1/4} \exp(- \pi {}^t y b^{-1} y) |\theta(z, \tau)| \]
and 
\[ \| \theta[\eta] \|(z, \tau) = (\det b)^{1/4} \exp(- \pi {}^t y b^{-1} y) |\theta[\eta](z, \tau)|. \]
If $\tau$ is understood, we simply write $\|\theta\|(z)$ and $\|\theta[\eta]\|(z)$.
We  have the formula
\[ \| \theta[\eta] \|(z, \tau) = \| \theta \| (z + \tau \eta' + \eta'', \tau). \]
We sometimes write $\eta = \tau\eta' + \eta''$, so that we have the formula $\|\theta[\eta]\|(z,\tau) = \|\theta\|(z + \eta, \tau)$.
\begin{lemma}\label{lem_sup_of_theta_functions}
\begin{enumerate}
    \item Let $\gamma = \left(\begin{smallmatrix} A & B \\ C & D \end{smallmatrix}\right) \in \Sp_{2g}(\Z)$ act on $\mathfrak{S}_g$ by the formula $\gamma \tau = (A \tau + B) (C \tau + D)^{-1}$, on $\bC^g$ by the formula $\gamma z = {}^t(C \tau + D)^{-1} z$, and on $\Q^g \times \Q^g / \Z^g \times \Z^g$ by the formula $\gamma \eta = \eta \gamma^{-1} + \frac{1}{2}( ( C {}^t D)_0, (A {}^t B)_0 )$, where for $X \in M_g(\Z)$, $X_0$ denotes the diagonal of $X$, considered as a row vector in $\Z^g$. Then for any $\eta \in \frac{1}{2} \Z^g \times \frac{1}{2} \Z^g / \Z^g \times \Z^g$, we have the formula
    \[ \| \theta[ \gamma \eta ] \|(\gamma z, \gamma \tau) = \| \theta[\eta] \|(z, \tau). \]
    \item The function $f(z, \tau) =  \max_{\eta \in (\frac{1}{2} \Z^g \times \frac{1}{2} \Z^g) / \Z^g \times \Z^g} \| \theta \|[\eta](z, \tau)$ is $\Sp_{2g}(\Z)$-invariant. 
\end{enumerate}
\end{lemma}
\begin{proof}
    \cite[Ch. II, Theorem 6, Corollary]{Igu72} gives the transformation formula for $\theta[\eta]$ under an element $\gamma  \in \Sp_{2g}(\Z)$. After a calculation similar to the one on \cite[p. 192]{MumfordTataI}, we obtain from this the formula
    \[ \| \theta[\eta \gamma^{-1} + \frac{1}{2} ((C {}^t D)_0, (A {}^t B)_0)]\| ({}^t(C \tau + D)^{-1} z, (A \tau + B)(C \tau + D)^{-1}) = \| \theta[\eta] \|(z, \tau).  \]
    Since the map $\eta \mapsto \eta \gamma^{-1} + \frac{1}{2} ((C {}^t D)_0, (A {}^t B)_0)$ is a permutation of $ (\frac{1}{2} \Z^g \times \frac{1}{2} \Z^g) / \Z^g \times \Z^g$, the invariance of $f(z, \tau)$ follows.
\end{proof}
Fix an embedding $\sigma : k \to \bC$, inducing a place $v$ of $k$, and an ordering $\omega_1, \dots, \omega_{2g+1}$ of the roots of $\sigma(f)$ in $\bC$. Associated to this ordering is a canonical choice of symplectic basis $\{ A_i, B_i \}_{1\leq i\leq g}$ for $\HH^1(C(k \otimes_\sigma \bC), \Z)$ (see \cite[Ch. IIIa, \S5]{MumfordTataII}). We define period matrices $\sigma, \sigma'$ by the formulae
\[ \sigma_{ij} = \int_{A_j} \frac{x^{i-1} dx}{2 y}, \sigma'_{ij} = \int_{B_j} \frac{x^{i-1} dx}{2y}. \]
We define $\tau = \sigma^{-1} \sigma' \in \mathfrak{S}_g$. Then the Abel--Jacobi map gives an isomorphism
\[ \mathrm{AJ}_\sigma : J(k \otimes_\sigma \bC) \to \bC^g / (\Z^g \oplus \tau \Z^g), \]
\[  \left[ \sum_{i} n_i P_i \right] \mapsto \sum_i n_i \sigma^{-1} \left( \int_{P_\infty}^{P_i} \frac{x^{j-1} dx}{2y} \right)_{j = 1, \dots, g}. \]
\begin{proposition}\label{prop_local_archimedean_canonical_height}
    With notation as above, define $\delta = (\delta', \delta'') \in \Q^g \times \Q^g$, where
    \[ \delta' = ( 1/2, 1/2, \dots, 1/2), \]
    \[ \delta'' = (g/2, (g-1)/2, \dots, 2/2, 1/2). \]
    Then we have an equality of N\'eron functions associated to the divisor $2 \Theta$, for any $P \in (J - \Theta)(k \otimes_\sigma \bC)$:
    \[  \log \max_i |x_i| / |x_1| + \mu_v(P) = - 2 \log \| \theta \|(\tau \delta' + \delta'' + \mathrm{AJ}_\sigma(P), \tau) + \frac{g+1}{8g+4} \log |\Delta| + \log \left( |\varphi(\tau)|^{1/4r} \det(b)^{1/2} \right), \]
    where $\varphi(\tau)$ is as given by \cite[Definition 3.1]{Loc94} and $r = \binom{2g+1}{g+1}$.
\end{proposition}
\begin{proof}
    The left-hand side is a N\'eron function by construction; the function $- 2 \log \| \theta \|(\tau \delta' + \delta'' +  \mathrm{AJ}_\sigma(P), \tau)$ is a N\'eron function by \cite[Ch. 13, Theorem 1.1]{Lan83} and the fact that $\theta[\delta](z,\tau)$ has divisor $\Theta$ by \cite[Ch. IIIa, \S5, Theorem 5.3(1) and top of p. 82]{MumfordTataII}. We therefore have a formula
    \begin{equation}\label{eqn_evaluation_of_height_constant}  \log \max_i |x_i| / |x_1| + \mu_v(P) = - 2 \log \| \theta[\delta] \|(z, \tau) + c_v 
    \end{equation}
    for a constant $c_v \in \R$ that we must evaluate. We evaluate both sides at the generic torsion point $T \in J(k \otimes_\sigma \bC)$ with $\mathrm{AJ}_\sigma(T) = \delta$.
    In fact, by \cite[Ch. IIIa, \S5, Lemma 5.6]{MumfordTataII} we have $T = [\sum_{i=0}^g (\omega_{2i+1}, 0) - (g+1) P_\infty]$. Let us write  $w \in \widetilde{K}(k \otimes_\sigma \bC)$ for the lift of $T$ with $w_1 = 1$; thus $w = M_T \infty$ by Part 2(c) of Corollary \ref{cor_properties_of_T_Heisenberg_matrices}. From the definition, we find that the left-hand side of Equation (\ref{eqn_evaluation_of_height_constant}) equals
    \[ \log \max_i |w_i| + \frac{1}{4} \varepsilon_v(T) = \frac{1}{4}\log \max_i |\delta_i(w) | = \frac{1}{4} \log | \delta_{2^g}(w) |, \]
    as $\delta_i(w) = 0$ if $i < 2^g$. The quantity $|\delta_{2^g}(w)|$ may be computed using the formula of Proposition \ref{prop_Riemann_theta_relation_as_incidence_relation}. We find that
    \[ |\delta_{2^g}(w)| = | \beta(\delta(w), M_T \cdot \infty) | = | \beta(w, M_T \cdot w) |^2 = | \beta(w, M_T^2 \cdot \infty) |^2 = \prod_{\substack{1 \leq i, j \leq 2g+1 \\ i \text{ odd, } j\text{ even}}}|\omega_i - \omega_j|^2,  \]
    by Proposition \ref{prop_properties_of_heisenberg_matrix}. On the other hand, we have by Thomae's formula (cf. \cite[Eqn. (3.6)]{Loc94})
    \[ \| \theta\|(0, \tau) = \det(b)^{1/4} | \det(\sigma)| ^{1/2} \pi^{-g/2} \prod_{\substack{1 \leq i < j \leq 2g+1\\ \text{ odd }}} |\omega_i - \omega_j|^{1/4} \prod_{\substack{1 \leq i < j \leq 2g+1\\ \text{ even }}} |\omega_i - \omega_j|^{1/4}. \]
    Re-arranging, we find that $c_v$ is equal to
    \begin{multline*} \frac{1}{2} \log \det(b) + \log |\det(\sigma)| - g \log \pi + \frac{1}{2} \log \prod_{\substack{1 \leq i, j \leq 2g+1 \\ i \text{ odd, } j\text{ even}}}|\omega_i - \omega_j|+ \frac{1}{2} \log \prod_{\substack{1 \leq i < j \leq 2g+1\\ \text{ odd }}} |\omega_i - \omega_j| \prod_{\substack{1 \leq i < j \leq 2g+1\\ \text{ even }}} |\omega_i - \omega_j| \\ 
    = \frac{1}{2} \log V_v - g \log \pi + \frac{1}{4} \log |\Delta|,
    \end{multline*}
    where $V_v = \det(b) |\det(\sigma)|^2$. This quantity satisfies the formula (\cite[Proposition 3.3]{Loc94})
\[ |\Delta(f)| V_v^{(4g+2)/g} = \pi^{8g+4} |\varphi(\tau)|^{1/n} \det(b)^{(4g+2)/g}, \]
where $n = \binom{2g}{g+1}$ and $\varphi(\tau)$ is the product of suitable $\theta[\eta](0, \tau)^8$ defined in \cite[Definition 3.1]{Loc94}. Re-arranging, we find
\begin{align}\label{eq_formula_Vv_theta}
 V_v  = \pi^{2g} |\Delta(f)|^{-g / (4g+2)} |\varphi(\tau)|^{1/2r} \det(b),
 \end{align}
where $r = \binom{2g+1}{g+1}$. We can therefore rewrite the constant $c_v$ as
\[ c_v = \frac{g+1}{8g+4} \log |\Delta| + \log \left( |\varphi(\tau)|^{1/4r} \det(b)^{1/2} \right). \]
This completes the proof. 
\end{proof}
When $g = 1$, we have $\varphi(\tau) = 2^8 \Delta(\tau)$, where $\Delta(\tau)$ is the usual modular discriminant. In this case, the formula of Proposition \ref{prop_local_archimedean_canonical_height} reduces to the one obtained from e.g.\ \cite[Proposition 4.5]{dejong-arakelovtheoryellcurves}. 

We now make a comparison to the reduction height $\widetilde{h}$:
\begin{proposition}\label{prop_archimedean_red_height_theta}
    Let $P \in J(k \otimes_\sigma \bC)$ correspond to a divisor class of Mumford degree $g$, with associated Mumford representation $(U, V, R)$. Then we have an equality
    \begin{multline*}
        \widetilde{h}_v(P) = 
        \\ = \frac{1}{2} \log \max_i |x_i| / |x_1| + \frac{1}{2}  \mu_v(P) -  \frac{g+1}{8(2g+1)} \log |\Delta|  -  \log \left( |\varphi(\tau)|^{1/8r} \det(b)^{1/4} \right)+\frac{1}{2} \log \sum_{k=1}^{2g+1} \frac{\| \theta \|(z + \tau \delta' + \delta'' + z_k)^2}{|f'(\omega_k)|^{1/2}}, 
    \end{multline*}
    where $z_k = \mathrm{AJ}_\sigma([(\omega_k, 0) - P_\infty])$. 
\end{proposition}
\begin{proof}
    Let $z = \mathrm{AJ}_\sigma(P)$ and $T = \{1,3,\dots, 2g+1\}$. We begin with the formula of \cite[Theorem 7.6]{MumfordTataII}: for any $1 \leq k \leq 2g+1$ and subset $S \subset \{1, \dots, 2g+1\}$ of size $g$ with $k \not\in S$ we have
    \[ U(\omega_k) = \pm r_{S, \{ k \}} \left( \frac{ \theta[\eta_{(S\cup\{k\}) \Delta T} + \eta_k](0, \tau)^2 \theta[\delta + \eta_k](z, \tau)^2}{\theta[\eta_{(S\cup\{k\}) \Delta T}](0,\tau) \theta[\delta](z,\tau)^2} \right). \]
    We leave the definition of the theta characteristics $\eta_{W}$ for a subset $W\subset\{1, \dots, 2g+1\}$ and $\eta_k$ to \emph{op. cit.}, but note that $z_k = \tau \eta'_k + \eta''_k$.
    Let $W = \{ 1, \dots, 2g+1\} - (S \cup \{ k \})$. Applying Thomae's formula \cite[Eqn. (3.6)]{Loc94} and using $f'(\omega_k) = \pm r_{\{ k \}^c, \{ k \}}$, we get that $| U(\omega_k) |/|f'(\omega_k)|$ equals
    \begin{align}\label{eq_substep_reductionheight_theta}
    |r_{W, \{ k \}}|^{-1} \left( \frac{| \delta_S \delta_{W} r_{W, \{ k \}}| }{|\delta_S r_{S, \{ k \}} \delta_W|}\right)^{1/2} \frac{ \| \theta[\delta+\eta_k]\|(z)^2}{\| \theta[\delta] \|(z)^2}
    = \| \theta \|(z + \tau \delta' + \delta'')^{-2} \cdot  \frac{\| \theta \|(z + \tau \delta' + \delta'' + z_k)^2}{|f'(\omega_k)|^{1/2}}
    ,\end{align}
    hence 
    \[ \sum_{k=1}^{2g+1} \frac{ | U(\omega_k) |}{|f'(\omega_k)|} = \| \theta \|(z + \tau \delta' + \delta'')^{-2} \cdot \sum_{k=1}^{2g+1} \frac{\| \theta \|(z + \tau \delta' + \delta'' + z_k)^2}{|f'(\omega_k)|^{1/2}},  \]
    hence
    \[ \frac{1}{2} \log \sum_{k=1}^{2g+1} \frac{ | U(\omega_k) |}{|f'(\omega_k)|} = - \log \| \theta \|(z + \tau \delta' + \delta'') + \frac{1}{2} \log \sum_{k=1}^{2g+1} \frac{\| \theta \|(z + \tau \delta' + \delta'' + z_k)^2}{|f'(\omega_k)|^{1/2}} \]
    \[ = \frac{1}{2} \log \max_i |x_i| / |x_1| + \frac{1}{2}  \mu_v(P) - \frac{1}{2} \frac{g+1}{8g+4} \log |\Delta|  - \frac{1}{2} \log \left( |\varphi(\tau)|^{1/4r} \det(b)^{1/2} \right) + \frac{1}{2} \log \sum_{k=1}^{2g+1} \frac{\| \theta \|(z + \tau \delta' + \delta'' + z_k)^2}{|f'(\omega_k)|^{1/2}}, \]
    using Proposition \ref{prop_local_archimedean_canonical_height}. This completes the proof. 
\end{proof}

Proposition \ref{prop_archimedean_red_height_theta} allows us to give a lower bound for the local archimedean canonical height in terms of the local reduction height, under an assumption that the roots of $f$ are `not too close' to each other.

\begin{corollary}\label{cor_generic_reduction_height_vs_canonical_height}
There is an absolute constant $c(g)$ with the following property: fix an embedding $\sigma : k \to \bC$ corresponding to a place $v \in M_{k, \infty}$, and let $\omega_1, \dots, \omega_{2g+1} \in \bC$ be the roots of $\sigma(f)$. Suppose given $X > 0$ and $0< \delta < 1$ such that $\max_i | c_i |_v^{1/i} \leq X$ and $|\omega_i - \omega_j| \geq X^{1-\delta}$ for all $i \neq j$. Then for any point $P \in J(k \otimes_\sigma \bC)$ of Mumford degree $g$, we have
\[  \frac{1}{2} \log \max_i | (x_i / x_1)(P) |_v + \frac{1}{2} \mu_v(P) \geq \widetilde{h}_v(P) + \frac{g(g+3)}{4} \log X - \delta \frac{g(2g+3)}{4} \log X + c(g). \]
\end{corollary}
\begin{proof}
       Plugging in the lower bounds $|\Delta|_v \geq X^{2g(2g+1)(1-\delta)}$ and $|f'(\omega_k)|_v \geq X^{2g(1-\delta)}$ in Proposition \ref{prop_archimedean_red_height_theta}, we obtain
    \[  \frac{1}{2} \log \max_i | (x_i / x_1)(P) |_v + \frac{1}{2} \mu_v(P) \geq \widetilde{h}_v(P) + g(g+3)/4 \cdot (1-\delta) - \frac{1}{2} \log \sum_{k=1}^{2g+1} \left( \frac{\| \theta \|(z + \tau \delta' + \delta'' + z_k)}{|\varphi(\tau)|^{1/8r} \det(b)^{1/4}} \right)^2. \]
    Define
    \[ \Psi(z, \tau) = \frac{1}{2} \log \sum_{k=1}^{2g+1} \left( \frac{\| \theta \|(z + \tau \delta' + \delta'' + z_k)}{|\varphi(\tau)|^{1/8r} \det(b)^{1/4}} \right)^2.\]
    We need to give an upper bound for $\Psi$. We have 
    \[ \Psi(z, \tau)  \leq \log \left( \frac{f(z, \tau)}{|\varphi(\tau)|^{1/8r}\det(b)^{1/4}} \right) + \frac{1}{2} \log(2g+1), \]
where $f(z, \tau)$ is as defined in Lemma \ref{lem_sup_of_theta_functions}. 

Take $\gamma \in \Sp_{2g}(\Z)$  such that $\tau' = \gamma \tau$ lies in the standard fundamental domain $\mathfrak{F}_g \subset \mathfrak{S}_g$ (described e.g.\ in \cite[\S 2.1]{Paz12}). By \cite[Proposition 5.5(2)]{Paz12} and Lemma \ref{lem_sup_of_theta_functions}(2), we then have $f(z, \tau) = f(\gamma z, \tau') \leq c_1(g) \det \operatorname{Im}(\tau')^{1/4}$ for an absolute constant $c_1(g)$. 

Next, we observe by definition of $\varphi(\tau)$ (\cite[Definition 3.1]{Loc94}) and $\|\theta[\eta]\|(z,\tau)$ that
\[ |\varphi(\tau)|^{1/8r}\det(b)^{1/4} = \left(\prod_{S} \| \theta[\eta_S]\|(0, \tau) \right)^{1/r}, \]
where the product runs over the set of non-vanishing theta characteristics (i.e. those $\eta_S \in (\frac{1}{2} \Z^g \times \frac{1}{2} \Z^g) / (\Z^g \times \Z^g)$ such that $\theta[\eta_S](0, \tau) \neq 0$).

Thomae's formula shows, together with our hypothesis on the numbers $|\omega_i - \omega_j|$, that for any non-vanishing theta characteristics $\eta_S, \eta_{S'}$ we have
\begin{equation}\label{eqn_thomae_applied_to_theta_ratios} \| \theta[\eta_S] \|(0, \tau)  / \| \theta[\eta_{S'}] \|(0, \tau) = \frac{ | \delta_{S \Delta U} \delta_{(S \Delta U)^c}|^{1/4}}{ | \delta_{S' \Delta U} \delta_{(S' \Delta U)^c}|^{1/4}} \geq (4X^\delta)^{-g^2/4}. 
\end{equation}
(We're also using here that $|\omega_i| \leq 2X$ for each $i$, by \cite[Lemma 2.1]{lagathorne2024smallheightoddhyperelliptic}.) By Lemma \ref{lem_sup_of_theta_functions}, we have
\[ \prod_S \| \theta [\eta_S] \|(0, \tau) = \prod_S \| \theta[\gamma \eta_S]\|(0, \tau'). \]
By \cite[Proposition 5.5(1)]{Paz12}, together with Equation (\ref{eqn_thomae_applied_to_theta_ratios}), we have
\[ \left(\prod_S \| \theta [\eta_S] \|(0, \tau) \right)^{1/r} \geq (4X^\delta)^{-g^2/4} \det \operatorname{Im}(\tau')^{1/4}. \]

Combining the above estimates, we conclude that $\Psi(z, \tau) \leq c_2(g) + \delta \frac{g^2}{4} \log X$ for an absolute constant $c_2(g)$. Combining this estimate with the first formula of the proof now gives the result. 
\end{proof}
The next result extends Corollary \ref{cor_generic_reduction_height_vs_canonical_height} to points $P$ of Mumford degree $m < g$.
\begin{proposition}\label{prop_general_reduction_height_vs_canonical_height}
    There are constants $a(g), b(g) \in \R$ with the following property: fix an embedding $\sigma : k \to \bC$ corresponding to a place $v \in M_{k, \infty}$, and let $\omega_1, \dots, \omega_{2g+1} \in \bC$ be the roots of $\sigma(f)$. Suppose given $X > 0$ such that $\max_i | c_i |_v^{1/i} \leq X$ and $|\omega_i - \omega_j| \geq X^{1-\delta}$ for all $i \neq j$. Then for any point $P \in J(k \otimes_\sigma \bC)$ of Mumford degree $m$ ($0 \leq m \leq g$), we have
\[  \frac{1}{2} \log \max_i | (x_i / x_K)(P) |_v + \frac{1}{2} \mu_v(P) \geq \widetilde{h}_v(P) + ((1 + m/2) g - m(m+1)/4 + a(g) \delta) \log X + b(g),  \]
where $x_K$ is the co-ordinate corresponding to the subset $\{ 2g, \dots, 2g+1-m \} \subset \{ 2g, \dots, g+1\}$.
\end{proposition}
\begin{proof}
   We can assume by continuity that $P$ corresponds to the class of a reduced divisor not meeting the Weierstrass points. 
   Let $(U, V, R)$ be the associated Mumford triple. For each $k = 1, \dots, 2g+1$, choose a subset $T_k \subset \{ 1, \dots, 2g+1 \}$ of order $g-m$ such that $k \not\in T_k$, and let $U_{T_k}(x) = \prod_{i \in T_k} ( x - \omega_i )$. We identify $T_k$ with the corresponding 2-torsion point; then $P + T_k$ has Mumford degree $g$, and the first entry of the Mumford triple is $U(x) U_{T_k}(x)$. 
   Specializing the expression \eqref{eq_substep_reductionheight_theta} to $P + T_k$ shows that
   \begin{align}\label{eq_proof_reductionheight_nongeneric1}
     \frac{|U(\omega_k)| |U_{T_k}(\omega_k) |}{|f'(\omega_k)|} = \| \theta\|(z + \delta + t_k)^{-2} \frac{\| \theta\|(z + \delta + t_k + \eta_k)^2}{|f'(\omega_k)|^{1/2}},
     \end{align}
    where $z = \mathrm{AJ}_\sigma(P)$ and $t_k=\mathrm{AJ}_{\sigma}(T_k)$. Let $K = \{ 1, \dots, m \}$, and let $x_K$ be the corresponding co-ordinate on $\P^{2^g-1}$, and let $Z_K \subset J$ denote its vanishing locus; then $Z_K$ and $2 \Theta$ are linearly equivalent. We need to compare $- 2 \log \| \theta\|(z + \delta + T_k)$ to the N\'eron function
    \[ \lambda_{Z_K, v}(Q) = \log \max_i |x_i| / |x_K| + \mu_v(Q). \]
    The function $f_k(Q) = x_1(M_{T_k} Q) / x_K(Q)$ has divisor $(2 t_{T_k}^\ast \Theta) - Z_K$, while $- 2 \log\| \theta\|(z + \delta + T_k)$ is a N\'eron function for the divisor $(2 t_{T_k}^\ast \Theta)$. It follows that there is a constant $c_k$ such that \begin{equation}\label{eqn_equality_of_canonical_heights} \lambda_{Z_K, v}(Q) = -2 \log \| \theta \|(\mathrm{AJ}_\sigma(Q) + \delta + T_k)  + \log |f_k(Q)|_v + c_k \end{equation}
    for all points $Q \in J(k \otimes_\sigma \bC)$ such that both sides are defined. Note that $P$ is an example of such a point. To compute the constant, we choose $Q = S_k$ for $S_k \subset \{ 1, \dots, 2g+1\}$ a subset of order $m$, disjoint from $T_k \cup \{ k \}$. Then the left-hand side of Equation (\ref{eqn_equality_of_canonical_heights}) equals
    \[ \frac{1}{4} \log \max_i | \delta_i(S_k) | / |x_K(S_k)|^4 \]
    Since $S_k$ is a 2-torsion point, Propositions \ref{prop_pairing_delta_with_torsion_points} and \ref{prop_properties_of_heisenberg_matrix} show that $\delta(S_k) = \delta_{2^g}(S_k) \infty = \pm r_{S_k, S_k^c}^2 \infty$ if we normalise so that $x_K(S_k) = 1$. Therefore the left-hand side equals $\frac{1}{2} \log |r_{S_k, S_k^c}|$. The right-hand side equals
    \[ -2 \log \| \theta(S_k + \delta + T_k) \| + \log |x_1(M_{T_k} S_k) / x_K(S_k) |_v + c_k. \]
    By Proposition \ref{prop_properties_of_heisenberg_matrix} again, we have 
    \[x_1(M_{T_k} S_k) / x_K(S_k) 
    = \beta(M_{T_k}M_{S_k}\infty, \infty) 
    = \beta(\pm r_{S_k, T_k} M_{T_k + S_k}\infty, \infty)
    = \pm r_{S_k, T_k},\]
    while Thomae's formula gives
    \[\| \theta\|(S_k + \delta + T_k) = \pi^{-g/2} \det(b)^{1/4} \det(\sigma)^{1/2} | \delta_{S_k \cup T_k} \delta_{(S_k \cup T_k)^c}|^{1/4}. \]
    We conclude that $c_k$ is equal to 
    \begin{multline*} \frac{1}{2} \log |r_{S_k, S_k^c}| + \frac{1}{2} \log \det b + \log \det \sigma + \frac{1}{2} \log \delta_{S_k \cup T_k} \delta_{(S_k \cup T_k)^c} - g \log \pi - \log |r_{S_k, T_k}|
    \\ = \frac{1}{2} \log V_v + \frac{1}{2} \log \frac{|\delta_{S_k \cup T_k} \delta_{(S_k \cup T_k)^c} r_{S_k, S_k^c}|}{|r_{S_k, T_k}^2|}  - g\log \pi\\
    = - \frac{g}{4(2g+1)} \log |\Delta|_v + \log\left( |\varphi(\tau)|^{1/4r} \det(b)^{1/2} \right) + \frac{1}{4} \log |\Delta|_v - \frac{1}{2} \log |r_{T_k, T_k^c}|
    \\ = \frac{g+1}{4(2g+1)} \log |\Delta|_v + \log\left( |\varphi(\tau)|^{1/4r} \det(b)^{1/2} \right) - \frac{1}{2} \log |r_{T_k, T_k^c}|. 
    \end{multline*}
    To derive these equalities, we have used \eqref{eq_formula_Vv_theta} for $V_v$.  
    Combining \eqref{eq_proof_reductionheight_nongeneric1} and \eqref{eqn_equality_of_canonical_heights} shows that
    \[ \frac{|U(\omega_k)|}{|f'(\omega_k)|} = \exp(\lambda_{Z_K, v}(P) - c_k )|f_k(P)|^{-1}_v \frac{ \| \theta \|(z + \delta + T_k + \eta_k)^2 }{|U_{T_k}(\omega_k)| |f'(\omega_k)|^{1/2}}. \]
    According to Lemma \ref{lemma_coordinate_resultant}, we have $|f_k(P)|_v = |\operatorname{Res}(U, a_k)|_v$. We can sum and take logs to obtain
    \[\widetilde{h}_v(P) = \frac{1}{2} \lambda_{Z_K, v}(P) - \frac{g+1}{8(2g+1)} \log |\Delta|_v + \frac{1}{2} \log \sum_{k=1}^{2g+1} |\operatorname{Res}(U, U_{T_k})|_v^{-1}  \frac{|r_{T_k, T_k^c}|^{1/2}}{|f'(\omega_k)|^{1/2} |U_{T_k}(\omega_k)|} \left( \frac{\| \theta \| (z + \delta + T_k + \eta_k) }{ |\varphi(\tau)|^{1/8r} \det(b)^{1/4}} \right)^2.
\]
Re-arranging, we find
\[ \frac{1}{2} \lambda_{Z_K, v}(P) = \widetilde{h}_v(P) + \frac{g+1}{8(2g+1)} \log |\Delta|_v - \frac{1}{2} \log \sum_{k=1}^{2g+1} |\operatorname{Res}(U, U_{T_k})|_v^{-1}  \frac{|r_{T_k, T_k^c}|^{1/2}}{|f'(\omega_k)|^{1/2} |U_{T_k}(\omega_k)|} \left( \frac{\| \theta \| (z + \delta + T_k + \eta_k) }{ |\varphi(\tau)|^{1/8r} \det(b)^{1/4}} \right)^2. \]
Using the inequalities satisfied by the $\omega_i$, we find that
\[ \frac{1}{2} \lambda_{Z_K, v}(P) \geq \widetilde{h}_v(P) + (g + \frac{m(m-1)}{4} - a(g) \delta) \log X - \frac{1}{2} \log \sum_{k=1}^{2g+1} |\operatorname{Res}(U, U_{T_k})|^{-1}_v + b(g), \]
where $a(g), b(g) \in \R$ is are constants. To complete the proof, we need to make a more careful choice of $T_k$. Let $\alpha_1, \dots, \alpha_m \in \bC$ be the roots of $U$. If $|\alpha_i - \omega_r| \leq \frac{1}{2} X^{1-\delta}$, then the triangle inequality shows that $|\alpha_i - \omega_s| \geq \frac{1}{2} X^{1-\delta}$ for all $r \neq s$. In particular, each disc $D(\alpha_i, \frac{1}{2} X^{1-\delta})$ contains at most one of the roots $\omega_r$. We can therefore choose $T_k$ to be supported on points $\omega_r$ satisfying $|\alpha_i - \omega_r| \geq \frac{1}{2} X^{1-\delta}$ for all $i = 1, \dots, m$, in which case we have $| \operatorname{Res}(U, U_{T_k}) |_v \geq (\frac{1}{2} X^{1-\delta})^{m(g-m)}$. This in turn gives
\[ \frac{1}{2} \log \sum_{k=1}^{2g+1} |\operatorname{Res}(U, U_{T_k})|_v^{-1} \geq -\frac{m(g-m)}{2} (1-\delta) \log X  + O_g(1). \]
Combining this with the previous equality now gives the result in statement of the proposition. 
\end{proof}

\subsection{An application to lower bounds for canonical heights}

We can now prove our `density 1' lower bound for canonical heights, which is one of the main theorems of this paper. We first introduce some notation. The integer $g \geq 1$ being fixed, define $\mathcal{F}(X)$, for any $X > 0$ to be the set of polynomials $f(x) = x^{2g+1} + c_2 x^{2g-1} + \dots + c_{2g+1} \in \Z[x]$ of nonzero discriminant and satisfying $\height(f) \leq X$. 
\begin{theorem}\label{thm_density_1_lower_bound_for_canonical_height}
    Let $g \geq 1$. Then for any $\epsilon > 0$, we have
    \[ \lim_{X \to \infty} \frac{\# \{ f \in \mathcal{F}(X) \mid \forall P \in J(\Q) - \{ 0 \}, \widehat{h}_\Theta(P) \geq ( \frac{3g-1}{2} - \epsilon ) \log \height(f) \}}{ \# \mathcal{F}(X)} = 1. \]
\end{theorem}
To prove Theorem \ref{thm_density_1_lower_bound_for_canonical_height}, we first introduce a subfamily of curves depending on a parameter $0< \delta <1/2$. Let $\mathcal{F}_\delta(X) \subset \mathcal{F}(X)$ denote the subset satisfying the following further conditions: 
\begin{itemize}
    \item $f(X)$ is irreducible.
    \item If $\omega_1, \dots, \omega_{2g+1} \in \bC$ are the roots of $f(x)$, then we have $|\omega_i - \omega_j| \geq X^{1-\delta}$ for all $i \neq j$. 
    \item If we factor $\Delta = \Delta_0 \Delta_1$ in $\Z$, where $(\Delta_0, \Delta_1) = 1$, $\Delta_0$ is squarefree, and $\Delta_1$ is squarefull, then $|\Delta_1| \leq X^\delta$.
\end{itemize}
\begin{proposition}\label{prop_good_family_has_density_1}
    For any $\delta \in (0,1/2)$, we have $\lim_{X \to \infty} \# \mathcal{F}_\delta(X) / \# \mathcal{F}(X) = 1$.
\end{proposition}
\begin{proof}
    We deduce this from the results of \cite{Bhargava-squarefree}. 
    Note that the first condition in the definition of $\mathcal{F}_{\delta}(X)$ is implied by the second and the third, since a nontrivial factorisation $f = f_1 f_2$ implies that $\mathrm{Res}(f_1, f_2) \mid \Delta_1$.
    For $f \in \mathcal{F}(X)$, let $\operatorname{rad}(\Delta_1)$ denote the radical of $\Delta_1$. Fix $M > 1$. For any $X^\delta > M$, $\mathcal{F}(X) - \mathcal{F}_\delta(X)$ is contained in the union of the 3 sets
    \[ S_1(X) = \{ f \in \mathcal{F}(X) \mid  \operatorname{rad}(\Delta_1) > X^\delta \}, \]
    \[ S_2(X) =\{ f \in \mathcal{F}(X) \mid |\Delta_1| > X^\delta, \operatorname{rad}(\Delta_1) \in [M, X^\delta] \}, \]
    and
    \[ S_3(X) = \{ f \in \mathcal{F}(X) \mid |\Delta_1| > X^\delta, \operatorname{rad}(\Delta_1) \in [2, M] \}.\]
    We have $\# \mathcal{F}(X) = 2^{2g} X^{g(2g+3)} + O(X^{g(2g+3)-1})$ \cite[Lemma 3.7]{lagathorne2024smallheightoddhyperelliptic}.
     It follows from \cite[Theorem 4.4]{Bhargava-squarefree} (see also \cite[Theorem 5.4]{Bhargava-squarefree}) that we have
    \[ \lim_{X \to \infty} \# S_1(X) / X^{g(2g+3)} = 0. \]
    If $f \in S_2(X)$, then there is a square-free integer $m \in [M, X^\delta]$ such that $m^2 | \Delta(f)$. If $p$ is prime, then the number of monic polynomials of degree $2g+1$ in $\Z / p^2 \Z[x]$ with $c_1 = 0 $ and of discriminant congruent to 0 modulo $p^2$ is at most $c p^{4g-2}$, where $c$ is an absolute constant. We find
    \[ \# S_2(X) / X^{g(2g+3)} \ll \sum_{\substack{m \in [M, X^\delta] \\ \text{squarefree}}} c^{\omega(m)} / m^2, \]
    where $\omega(m)$ denotes the number of prime factors of $m$. using the estimate $c^{\omega(m)} \ll m^\epsilon$ for any $\epsilon > 0$ we find that
    \[ \limsup_{X \to \infty} \# S_2(X) / X^{g(2g+3)}  = O_\epsilon(M^{\epsilon - 1}). \]
    Finally, if $f \in S_3(X)$, then there exists a prime $p \leq M$ such that $p^{\lfloor \log_M X^\delta \rfloor} | \Delta_1$. It follows that for any $N > 1$ we have, when $X$ is sufficiently large, an inclusion
    \[ S_3(X) \subset \{ f \in \mathcal{F}(X) \mid \exists p \leq M,  p^N | \Delta_1 \}. \]
    For any $\gamma > 0 $, we can find an $N > 1$ such that for each prime $p \leq M$, we have
    \[ p^{-2g N} \# \{ f(x) = x^{2g+1} + c_2 x^{2g_1} + \dots \in \Z / p^N \Z[x] \mid \Delta(f) \equiv 0 \text{ mod }p^N \} \leq \gamma. \]
    We see that we have
    \[ \limsup_{X \to \infty} \# S_3(X) / X^{g(2g+3)} \leq \gamma \pi(M) \]
    for any $\gamma > 0$. Since $\gamma > 0$ was arbitrary, this shows that
    \[ \lim_{X \to \infty} \# S_3(X) / X^{g(2g+3)} = 0. \]
    Combining our three estimates, we see that we have shown
    \[ \limsup_{X \to \infty} \# ( \mathcal{F}(X) - \mathcal{F}_\delta(X) ) / X^{g(2g+3)} = O_\epsilon(M^{\epsilon - 1}). \]
    Since $M$ was arbitrary, this shows that the limit is in fact 0, as required. 
\end{proof}
\begin{lemma}\label{lem_second_lower_bound_on_canonical_height}
    Let $g \geq 1$. Then there are constants $a(g), b(g) \in \R$ such that for any  $f(x) \in \mathcal{F}_\delta(X)$, and any $P \in J(\Q) - 0$, we have
    \[ \widehat{h}_\Theta(P) \geq \widetilde{h}(P) + \frac{3g-1}{2}  \log X + a(g) \delta \log X + b(g).  \]
\end{lemma}
\begin{proof}
   This follows on combining Proposition \ref{prop_general_reduction_height_vs_canonical_height} with Theorem \ref{thm_canonical_height_difference}, noting that if $f(m) = g(1+m/2)-m(m+1)/4$, then $f(m) \geq f(1) = (3g-1)/2$ for each $m = 1, \dots, g$. 
\end{proof}
\begin{proof}[Proof of Theorem \ref{thm_density_1_lower_bound_for_canonical_height}]
    In view of the quadratic property of the height, it suffices to show that a density 1 set of polynomials $f(x)$ satisfy $\widehat{h}_\Theta(P) \geq \left(\frac{3g-1}{2} -\epsilon \right)\height(f)$ for all $P \in J(\Q) - 2 J(\Q)$. However, \cite[Corollary 3.11]{lagathorne2024smallheightoddhyperelliptic} and \cite[Theorem 4.10]{lagathorne2024smallheightoddhyperelliptic} together show that, for a density 1 set of polynomials $f(x)$, all points $P \in J(\Q) - 2 J(\Q)$ satisfy $\widetilde{h}(P) > - \epsilon \log \height(f)$. 
    (Note that \cite[Theorem 4.10]{lagathorne2024smallheightoddhyperelliptic} contains the condition that $(U, f) = 1$, but this is implied by the irreducibility of $f$.) 
    The result now follows on combining this fact with Proposition \ref{prop_good_family_has_density_1} and Lemma \ref{lem_second_lower_bound_on_canonical_height}.
\end{proof}

\section{Acknowledgements}

Funded by the European Union (ERC  CoG-101169866). Views and opinions expressed are however those of the author(s) only and do not necessarily reflect those of the European Union or the European Research Council. Neither the European Union nor the granting authority can be held responsible for them.

\bibliographystyle{abbrv}
\bibliography{references}
\end{document}